\documentclass{article}

\usepackage{arxiv}
\usepackage[utf8]{inputenc} 
\usepackage[T1]{fontenc}    
\usepackage{hyperref}       
\usepackage{url}            
\usepackage{booktabs}       
\usepackage{amsfonts}       
\usepackage{nicefrac}       
\usepackage{microtype}      
\usepackage{lipsum}
\usepackage{amsmath,amssymb}
\usepackage{amsthm}
\usepackage{graphicx}
\usepackage{listings}
\usepackage{color}
\usepackage{fancyhdr}
\usepackage[framemethod=tikz]{mdframed}
\usepackage{subfig}
\usepackage[nottoc,notlot,notlof]{tocbibind}
\usepackage{csvsimple}
\usepackage{longtable}
\usepackage{colortbl}
\usepackage{float}
\usepackage[linesnumbered,commentsnumbered,ruled]{algorithm2e}
\usepackage{verbatim}
\usepackage{lscape}
\usepackage{multirow}
\usepackage{hhline}

\title{An Adaptive Finite Difference Method for Total Variation Minimization}

\author{
  Thomas JACUMIN \\
  LTH, University of Lund, Sweden \\
  \texttt{thomas.jacumin@math.lth.se} \\
  \and
  \textbf{Andreas LANGER} \\ 
  LTH, University of Lund, Sweden \\
  \texttt{andreas.langer@math.lth.se}
}

\newcommand{\R}{\mathbb{R}}
\newcommand{\N}{\mathbb{N}}

\DeclareMathOperator{\divop}{div}

\DeclareMathOperator{\argmin}{argmin}

\DeclareMathOperator{\TV}{TV}
\DeclareMathOperator{\BV}{BV}

\mdfdefinestyle{thm}{
	nobreak=true,
	hidealllines=true,
	leftline=true,
	innerleftmargin=10pt,
	innerrightmargin=10pt,
	innertopmargin=0pt,
}

\surroundwithmdframed[style=thm]{definition}
\newtheorem{proposition}{Proposition}[section]\surroundwithmdframed[style=thm]{proposition}
\surroundwithmdframed[style=thm]{theorem}
\newtheorem{lemma}{Lemma}[section]\surroundwithmdframed[style=thm]{lemma}
\surroundwithmdframed[style=thm]{property}
\surroundwithmdframed[style=thm]{problem}
\newtheorem*{problem*}{Problem}\surroundwithmdframed[style=thm]{problem*}

\theoremstyle{remark}

\newtheorem*{note}{\textbf{Remark}}
\newtheorem*{example}{\textbf{Example}}

\begin{document}

\maketitle

\begin{abstract}
In this paper, we propose an adaptive finite difference scheme in order to numerically solve total variation type problems for image processing tasks. The automatic generation of the grid relies on indicators derived from a local estimation of the primal-dual gap error. This process leads in general to a non-uniform grid for which we introduce an adjusted finite difference method. Further we quantify the impact of the grid refinement on the respective discrete total variation. In particular, it turns out that a finer discretization may lead to a higher value of the discrete total variation for a given function. To compute a numerical solution on non-uniform grids we derive a semi-smooth Newton algorithm in 2D for scalar and vector-valued total variation minimization. We present numerical experiments for image denoising and the estimation of motion in image sequences to demonstrate the applicability of our adaptive scheme.
\end{abstract}


\let\thefootnote\relax\footnotetext{The work was supported by the Crafoord Foundation through the project ``Locally Adaptive Methods for Free Discontinuity Problems''.}
\section{Introduction}

In \cite{Rudin1992} total variation (TV) was introduced as a regularization technique to imaging problems. 
Given the effectiveness of this regularization, in particular thanks to its property of preserving discontinuities in the solution, since then total variation is widely used in image processing. 
For example, it shows its efficiency for image inpainting \cite{Shen2002}, image denoising \cite{Buades2005, Chambolle1997}, image zooming \cite{Chambolle2004} and optical flow estimation \cite{Zach2007}, only to mention a few applications where total variation regularization is successfully applied. 
Thereby the problem typically states as follows: 
let $\Omega \subset \R^d$ be an open, bounded and simply connected domain with Lipschitz boundary, where $d\in\N$ denotes the spatial dimension, then one solves
\begin{equation}\label{Eq:ProblemTV}
\inf_{\mathbf{u}\in L^2(\Omega)^m\cap \BV(\Omega)^m} \mathcal{D}(T\mathbf{u}, g) + \lambda\int_\Omega|D \mathbf{u}|_{\text{F},r}. 
\end{equation}
Here $m\in\N$ denotes the number of channels, e.g.\ $m=1$ for scalar images or $m=d$ for motion fields, $g\in L^2(\Omega)$ is the given data, $T:L^2(\Omega)^m \to L^2(\Omega)$ is a bounded linear operator, $\mathcal{D}$ is a data fitting term enforcing the consistency between $g$ and the solution, and $\lambda>0$ is the regularization parameter. 
Further $\BV(\Omega)^m$ denotes the space of $m$-vector-valued functions with bounded variation, i.e.\ $\BV(\Omega)^m := \{ \mathbf{u}\in L^1(\Omega)^m\ |\ \int_\Omega |D\mathbf{u}|_{\text{F},r} < +\infty \}$ with $\int_\Omega |D\mathbf{u}|_{\text{F},r}$ being the total variation of $\mathbf{u}$ in $\Omega$ defined by 
$$ 
\int_\Omega |D\mathbf{u}|_{\text{F},r} := \sup \bigg\{ \int_\Omega \mathbf{u}\cdot \divop\mathbf{v}\ \bigg|\ \mathbf{v}\in\big(C_0^\infty(\Omega)^d\big)^m,\ |\mathbf{v}|_{*}\leq 1\ \text{for}\ \mu\text{-almost every}\ x\in\Omega \bigg\}. 
$$
The operator $\divop : \big(C^\infty_0(\Omega)^d\big)^m \to C^\infty_0(\Omega)^m$ describes the divergence, $|\cdot|_{\text{F},r} : \R^{d\times m} \to \R$ denotes the $r$-Frobenius norm defined as  
$$ 
|\mathbf{v}|_{\text{F},r} := \left( \int_\Omega \sum_{n=1}^m \sum_{k=1}^d |v_{nk}|^r \ dx \right)^{1/r}
$$
for $\mathbf{v}=\big((v_{11}, \hdots v_{1d}), \hdots, (v_{m1}, \hdots v_{md})\big)\in\big(C_0^\infty(\Omega)^d\big)^m$ and $r\geq 1$ , and $|\cdot|_{*}$ is the dual norm of the $r$-Frobenius norm.

In the presence of noise, the choice of the data fitting term $\mathcal{D}$ typically depends on the nature of the noise contamination. 
Specifically, a quadratic $L^2$-norm is used if the data is corrupted by Gaussian noise \cite{Chambolle2010}, whereas an $L^1$-norm is preferred for data affected by impulse noise \cite{Alliney1997, Nikolova2002, Nikolova2004}. Using these data fitting terms, we refer to \eqref{Eq:ProblemTV} as the $L^2$-TV model and the $L^1$-TV model, respectively. In cases where both Gaussian and impulse noise are present simultaneously in the data, a combined $L^1$/$L^2$ data fitting term has been shown to be effective \cite{Hintermuller2013, Langer2017Jul, Langer2019}. Then \eqref{Eq:ProblemTV} reads as
\begin{equation}\label{Eq:L1L2TV}
\inf_{\mathbf{u}\in L^2(\Omega)^m\cap \BV(\Omega)^m} \alpha_1\|T\mathbf{u} - g\|_{L^1(\Omega)} + \frac{\alpha_2}{2}\|T\mathbf{u} - g\|_{L^2(\Omega)}^2 + \lambda\int_\Omega|D \mathbf{u}|_{\text{F},r}, 
\end{equation}
where $\alpha_1,\alpha_2 \geq 0$, which we call the $L^1$-$L^2$-TV model. In this paper we will consider the $L^1$-$L^2$-TV model due to its generality, since it includes the $L^2$-TV model ($\alpha_1=0$) and the $L^1$-TV model ($\alpha_2=0$).

When looking at a concrete digital image, the underlying analog image is usually discretized by square pixels of uniform size. For further image enhancement (e.g.\ removing noise) or image analysis (e.g.\ motion estimation), this strongly supplies a finite difference implementation of image processing algorithms. Hence, for finding a numerical solution of \eqref{Eq:L1L2TV} this means discretizing \eqref{Eq:L1L2TV} by utilizing a finite difference method on a uniform grid. 
Despite the good quality of the reconstruction, using such a uniform discretization can lead to numerous degrees of freedom and, as a result, to an expensive computational cost, which is particularly problematic for large scale images.

It is worth mentioning that other discretization methods for total variation minimization have been considered, such as the finite element method \cite{Chambolle2021} and a neural network based method \cite{Langer2024}. In the finite element framework, adaptive discretization based on a posteriori error estimates has been employed to reduce the degrees of freedom in the problem \cite{Bartels2023, Hintermuller2014}.
The $L^1$-$L^2$-$\TV$ model has also been studied in this context \cite{Alkamper2024, Alkamper2017, Hilb2023Jul}. In particular, in \cite{Alkamper2024} an adaptive finite element coarse-to-fine scheme has been proposed showing a significant speed-up in computational time.
Challenges arise when selecting the finite element space. When using the space of piecewise constant functions, $\Gamma$-convergence of the finite dimensional problem to the infinite dimensional cannot be expected, as demonstrated by a counterexample in \cite{Bartels2012}. 
In contrast, such a convergence result holds when using the space of continuous and piecewise linear functions \cite{Langer2011, Langer2024}, though this approach prevents the reconstruction from exhibiting discontinuities \cite{Alkamper2017}. Furthermore, using finite differences for images may seem more natural than using finite elements, as the data (i.e.\ pixels) are organized on a uniform discrete grid. 

 Adaptive finite difference methods have long been explored in the literature for solving a variety of partial differential equations (PDEs); see, for example, \cite{Berger1989, Min2006} and references therein. More recently, in \cite{Oberman2016} a finite difference scheme on non-uniform grids within a general framework for solving nonlinear PDEs is introduced. Although this approach is quite flexible, it does not extend to total variation minimization. As such, to the best of our knowledge, no adaptive finite difference method has been proposed for \eqref{Eq:L1L2TV}.
 
In this paper, we suggest to solve the $L^1$-$L^2$-$\TV$ model by combining a finite difference method with an adaptive refinement strategy. This allows us on the one hand to reduce the number of degrees of freedom, while on the other hand we do not need to make a choice of a function space. In our refinement process we use a cell-centered sampling, that is the degrees of freedom are in the middle of the elements, allowing us to mimic a piecewise constant function, which is well suited for discontinuity-preserving problems.
Further we opted for a coarse-to-fine approach. That is we start with a coarse grid and subsequently subdivide the marked elements into four smaller elements to enhance the accuracy of the reconstruction. 
The marking of the elements is based on a local error indicator which is derived from a primal-dual gap error estimate following the idea in \cite{Bartels2020}. 
We note that the error of the total variation due to a finite difference discretization has been estimated in some works \cite{Caillaud2023, Wang2011}. 
However, since these estimates are global, they are not suitable for grid adaptivity.
For computing a solution on the obtained non-uniform grid we adapt the primal-dual semi-smooth Newton method derived and analyzed in \cite{Hilb2023Jul} to our finite difference setting. 

We theoretically analyze the influence of the grid refinement on a discretized total variation. From considerations in \cite{Caillaud2023, Condat2017} one implies that the discrete total variation of a function can be mesh dependent, although this is not highlighted there. In this paper we make an effort and quantify the impact of a refinement on the discrete total variation. In fact we derive a pixel-wise constant function that is bounded above by one, indicating that a refinement either keeps or increases the total variation of a function. 

The article is organized as follows: in Section \ref{sec:formulation}, we introduce the primal-dual setting of the considered problem before we address in Section \ref{sec:discretization} its discretization on non-uniform grids. Section \ref{sec:afdm} is devoted to the derivation of adaptive finite difference schemes. The impact of mesh refinement on the total variation is analyzed and quantified in Section \ref{sec:analysis-adaptive}. Section \ref{sec:error} focuses on computing the primal-dual gap error and proposes a local indicator essential for mesh adaptivity. In Section \ref{sec:algorithms}, we present a semi-smooth Newton algorithm to effectively solve the $L^1$-$L^2$-$\TV$ problem, followed by numerical experiments in Section \ref{sec:experiments} to validate our approach.

\section{Formulation of the Model}
\label{sec:formulation}

For a Hilbert space H we denote by $\langle \cdot, \cdot \rangle_H$ its inner product and by $\|\cdot\|_H$ its induced norm. The duality pairing between $H$ and its continuous dual space $H^*$ is written as $\langle \cdot, \cdot \rangle_{H,H^*}$. 
For functions $\mathbf{v}=(v_1,\hdots,v_m), \mathbf{w}=(w_1,\hdots,w_m)\in L^2(\Omega)^m$ we define
$$ 
\langle \mathbf{v}, \mathbf{w} \rangle_{L^2(\Omega)^m} := \sum_{k=1}^m \langle v_k, w_k \rangle_{L^2(\Omega), L^2(\Omega)} \qquad\text{and} \qquad \|\mathbf{v}\|_{L^2(\Omega)^m}^2 := \langle \mathbf{v}, \mathbf{v} \rangle_{L^2(\Omega)^m}, 
$$
and for $\mathbf{v}=(\mathbf{v}_1,\hdots,\mathbf{v}_m), \mathbf{w}=(\mathbf{w}_1,\hdots,\mathbf{w}_m)\in\big(L^2(\Omega)^d\big)^m$ we write
$$ \langle \mathbf{v}, \mathbf{w} \rangle_{\big(L^2(\Omega)^d\big)^m} := \sum_{k=1}^m \langle \mathbf{v_k}, \mathbf{w_k} \rangle_{L^2(\Omega)^d}\qquad \text{and}\qquad \|\mathbf{v}\|_{\big(L^2(\Omega)^d\big)^m}^2 := \langle \mathbf{v}, \mathbf{v} \rangle_{\big(L^2(\Omega)^d\big)^m}. $$
%
For a bounded and linear operator $\Lambda$ between two Banach spaces we denote by $\Lambda^*$ its adjoint operator.

As in \cite{Hilb2023Jul}, we consider the following penalized version of \eqref{Eq:L1L2TV}
\begin{equation} \tag{$\mathcal{P}$} \label{eq:primal}
	\inf_{\mathbf{u}\in H^1(\Omega)^m} \left\{ E(\mathbf{u}) := \alpha_1\|T\mathbf{u}-g\|_{L^1(\Omega)} + \frac{\alpha_2}{2} \|T\mathbf{u} - g\|_{L^2(\Omega)}^2 + \frac{\beta}{2} \|S\mathbf{u}\|_{V_S}^2 + \lambda\int_\Omega |\nabla \mathbf{u}|_{\text{F},r}\ dx \right\},
\end{equation}
where $S:H^1(\Omega)^m\to V_S$ is a bounded linear operator for some Hilbert space $V_S$ and $\beta\geq 0$ typically chosen very small such that \ref{eq:primal} is a close approximation of \eqref{Eq:L1L2TV}.
A primal-dual framework for \eqref{eq:primal} has been developed in \cite{Hilb2023Jul} for the following settings of $S$ and its related space $V_S$: 
\begin{itemize}
	\item $S=Id$, then $V_S\subseteq L^2(\Omega)^m$ and $\|\cdot\|_{V_S} = \|\cdot\|_{L^2(\Omega)^m}$,
	\item $S=\nabla$, then $V_S\subseteq \big(L^2(\Omega)^d\big)^m$ and $\|\cdot\|_{V_S} = \|\cdot\|_{\big(L^2(\Omega)^d\big)^m}$.
\end{itemize}
We suppose that $B:= \alpha_2 T^*T + \beta S^*S$ is invertible, which is, for example, the case for $\beta > 0$ and $S=I$, even if $T$ is not injective. Further, we define $ \|\cdot\|_{B^{-1}}^2 := \langle\ \cdot, B^{-1}\cdot\ \rangle_{L^2(\Omega)^m}$. The dual problem of \eqref{eq:primal} reads
\begin{equation} \tag{$\mathcal{D}$} \label{eq:dual}
	\begin{split}
		\sup_{(p_1,\mathbf{p_2})\in L^2(\Omega)\times \big(L^2(\Omega)^d\big)^m} \bigg\{ D(p_1,\mathbf{p_2}) := -\frac{1}{2}\|T^* p_1 + \nabla^*\mathbf{p_2} - \alpha_2 T^* g\|_{B^{-1}}^2 + \frac{\alpha_2}{2}\|g\|_{L^2(\Omega)}^2 - \langle g, p_1 \rangle_{L^2(\Omega)} \\ - \chi_{|p_1|\leq \alpha_1} - \chi_{|\mathbf{p_2}|_{F,s}\leq \lambda} \bigg\},
	\end{split}
\end{equation}
cf.\ \cite{Hilb2023Jul}, where $r^{-1} + s^{-1} = 1$.
The function $\chi$ is defined for a given predicate $\omega:\Omega\to \{\text{true}, \text{false}\}$ by
$$ \chi_\omega (x) := \begin{cases}
	0, & \text{if}\ \omega(x)\ \text{is true}, \\
	+\infty, & \text{otherwise},
\end{cases} $$
for $x\in\Omega$.
\section{Problem Discretization on Non-Uniform Grid}
\label{sec:discretization}

In this section, we introduce a finite difference discretization of \eqref{eq:primal}%
, focusing on the two-dimensional case ($d=2$), despite an extension to other dimensions seems possible. This discretization triangulates the domain $\Omega$ by a family of squares $(E_i)_{i=1}^N$, $N\in\N$ denoting the degrees of freedom, where element $E_i$ has side length $h_i$ and center $\mathbf{x}_i\in\Omega$. In addition, we define the discrete domain $\Omega_h$ as the collection of centers, i.e.\ $\Omega_h := \{ \mathbf{x}_i\in\Omega\ |\ i=1,\hdots,N \}$.
%

The mesh is implemented using a \textit{quad-tree} data structure \cite{DeBerg2000}. A quad-tree is a tree-shaped structure where each non-leaf node has four branches. Each node in the quad-tree represents a square. If a node has branches, then its branches represent the four smaller squares that make up the larger square represented by the node. We \textit{refine} an element by dividing it into four sub-elements, which involves adding four child nodes to its corresponding leaf node. To ensure specific mesh patterns and simplify the derivation of the finite difference scheme, we impose the following constraint: each element of the mesh must either be the same length as its neighboring elements or be twice or half as large. We refer to Figure \ref{fig:counter-example} for an example of such a mesh.
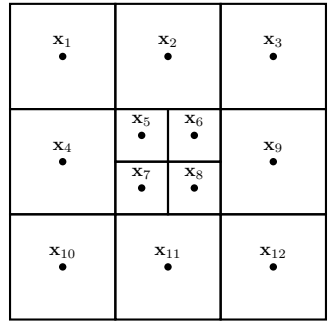
\begin{figure}[h]
	\centering
	\begin{tikzpicture}[thick,scale=0.7, every node/.style={scale=0.7}]
		\draw[draw=black] (1,1) rectangle ++(2,2);
		\draw[draw=black] (1,3) rectangle ++(2,2);
		\draw[draw=black] (1,5) rectangle ++(2,2);
		\draw[draw=black] (3,1) rectangle ++(2,2);
		\draw[draw=black] (3,3) rectangle ++(1,1);
		\draw[draw=black] (4,3) rectangle ++(1,1);
		\draw[draw=black] (3,4) rectangle ++(1,1);
		\draw[draw=black] (4,4) rectangle ++(1,1);
		\draw[draw=black] (3,5) rectangle ++(2,2);
		\draw[draw=black] (5,1) rectangle ++(2,2);
		\draw[draw=black] (5,3) rectangle ++(2,2);
		\draw[draw=black] (5,5) rectangle ++(2,2);
		\fill (2,6) circle[radius=2pt] node[yshift=8pt] {$\mathbf{x}_1$};
		\fill (4,6) circle[radius=2pt] node[yshift=8pt] {$\mathbf{x}_2$};
		\fill (6,6) circle[radius=2pt] node[yshift=8pt] {$\mathbf{x}_3$};
		\fill (2,4) circle[radius=2pt] node[yshift=8pt] {$\mathbf{x}_4$};
		\fill (3.5,4.5) circle[radius=2pt] node[yshift=8pt] {$\mathbf{x}_5$};
		\fill (4.5,4.5) circle[radius=2pt] node[yshift=8pt] {$\mathbf{x}_6$};
		\fill (3.5,3.5) circle[radius=2pt] node[yshift=8pt] {$\mathbf{x}_7$};
		\fill (4.5,3.5) circle[radius=2pt] node[yshift=8pt] {$\mathbf{x}_8$};
		\fill (6,4) circle[radius=2pt] node[yshift=8pt] {$\mathbf{x}_9$};
		\fill (2,2) circle[radius=2pt] node[yshift=8pt] {$\mathbf{x}_{10}$};
		\fill (4,2) circle[radius=2pt] node[yshift=8pt] {$\mathbf{x}_{11}$};
		\fill (6,2) circle[radius=2pt] node[yshift=8pt] {$\mathbf{x}_{12}$};
	\end{tikzpicture}
	\caption{Example of an adaptive grid.}
	\label{fig:counter-example}
\end{figure} 
A function $v$ mapping from $\Omega$ to $\R$ is discretized by a vector $v_h\in\R^N$ where each component $(v_h)_i$ is defined as 
$ h_i^{-2}\int_{E_i} v\ dx $
, for $i=1,\hdots,N$. The following scalar product and norm are then used for $v_h, w_h\in\R^N$
$$ \langle v_h, w_h \rangle_{L^2(\Omega_h)} := \sum_{i=1}^N h_i^2 (v_h)_i (w_h)_i,\quad\text{and}\quad \| v_h \|_{L^2(\Omega_h)}^2 := \langle v_h, v_h \rangle_{L^2(\Omega_h)}. $$
Note that the definition above is the exact numerical integration for a piecewise constant function, in particular $\langle v_h, w_h \rangle_{L^2(\Omega_h)} = \langle I_h v_h, I_h w_h \rangle_{L^2(\Omega)}$, where $I_h$ is the piecewise constant interpolation. Moreover, $I_h g_h$ is the $L^2$-projection of $g$ into the space of piecewise constant functions. The main motivation for such a choice of interpolation is that, since a digital image is a set of pixels, the discontinuities are located between the pixels. As a consequence, using a linear interpolation would lead to a smoothing of the discontinuities, which we want to avoid, cf. \cite[Figure~2]{Alkamper2017} in a finite element setting. We discretize a function $\mathbf{v}=(v_1,\hdots,v_m)$ from $\Omega$ to $\R^m$ by a vector $\mathbf{v_h}\in\R^{mN}$ such that $(\mathbf{v_h})_{(k-1)N+i} := v_k(\mathbf{x}_i)$, for $i=1,\hdots,N$ and $k=1,\hdots,m$. We use the following scalar product and norm for $\mathbf{v_h},\mathbf{w_h}\in\R^{mN}$
$$ \langle \mathbf{v_h}, \mathbf{w_h} \rangle_{L^2(\Omega_h)^m} := \sum_{i=1}^N h_i^2 \sum_{k=1}^m (\mathbf{v_h})_{(k-1)N+i} (\mathbf{w_h})_{(k-1)N+i}\quad\text{and}\quad \| \mathbf{v_h} \|_{L^2(\Omega_h)^m}^2 := \langle \mathbf{v_h}, \mathbf{v_h} \rangle_{L^2(\Omega_h)^m}. $$
%
%
A gradient-like function $\mathbf{w} = \big( (w_{11},w_{12}), \hdots, (w_{m1},w_{m2}) \big)$ from $\Omega$ to $(\R^2)^m$ is discretized by a vector $\mathbf{w_h}\in\R^{2 m N}$ such that $(\mathbf{w_h})_{2(n-1)N + (k-1)N + i} := w_{kn}(\mathbf{x}_i)$, for $n=1,\hdots,m$ and $k=1,2$. For $\mathbf{w_h},\mathbf{q_h}\in\R^{2mN}$, we introduce the scalar product
$$ \langle \mathbf{w_h}, \mathbf{q_h} \rangle_{\big(L^2(\Omega_h)^2\big)^m} := \sum_{i=1}^N h_i^2  \sum_{n=1}^m\sum_{k=1}^2 (\mathbf{w_h})_{2(n-1)N + (k-1)N + i} (\mathbf{q_h})_{2(n-1)N + (k-1)N + i},$$
and the norm
$$\| \mathbf{w_h} \|_{\big(L^2(\Omega_h)^2\big)^m}^2 := \langle \mathbf{w_h}, \mathbf{w_h} \rangle_{\big(L^2(\Omega_h)^2\big)^m}. $$
The discretization of the $r$-Frobenius norm becomes
$$ (|\mathbf{w_h}|_{\text{F},r})_i := \left( \sum_{n=1}^{m} \sum_{k=1}^{2} \big( (\mathbf{w_h})_{2(n-1)N + (k-1)N + i} \big)^r \right)^{1/r},\ \forall i\in\{1,\hdots,N\}. $$ 
%
%
Since discrete functions are represented by vectors in a finite difference framework, it is necessary to discretize the corresponding operators as well. We propose to define the discretization $T_h$ of $T$ such that, for a vector $v_h$, we have
$ (T_h v_h)_i := (T I_h v_h)(\mathbf{x}_i)$, for $ i\in\{1,\hdots,N\} $.
For our problem \eqref{eq:primal}, we represent by $g_h\in \R^N$ the discretization of the data $g$ and let $T_h \in \R^{N\times mN}$, $T_h^* \in \R^{mN\times N}$, $S_h \in \R^{N\times mN}$, $S_h^* \in \R^{mN\times N}$, $\nabla_h \in \R^{2mN\times mN}$, $\nabla_h^* \in \R^{mN\times 2mN}$, $B_h\in\R^{mN\times mN}$ and $B_h^{-1}\in\R^{mN\times mN}$ be the discretization of $T$, $T^*$, $S$, $S^*$, $\nabla$, $\nabla^*$, $B$ and $B^{-1}$ respectively. Note that $B_h^{-1}$ is not necessarily the inverse of $B_h$.
For $\mathbf{w_h}\in\R^{mN}$, we define $ \|\mathbf{w_h}\|_{B_h^{-1}}^2 := \langle \mathbf{w_h}, B_h^{-1} \mathbf{w_h} \rangle_{L^2(\Omega_h)^m} $.
%
With the above notations the discretized primal problem reads
\begin{equation} \tag{$\mathcal{P}_h$} \label{eq:primal_h}
	\min_{\mathbf{v_h}\in\R^{mN}} \bigg\{ E_h(\mathbf{v_h}) := \alpha_1\sum_{i=1}^N h_i^2 |(T_h\mathbf{v_h}-g_h)_i| + \frac{\alpha_2}{2} \| T_h\mathbf{v_h} - g_h \|_{L^2(\Omega_h)}^2 + \frac{\beta}{2} \| S_h\mathbf{v_h} \|_{V_{S,h}}^2 + \lambda \sum_{i=1}^N h_i^2 (|\nabla_h \mathbf{v_h}|_{\text{F},r})_i \bigg\},
\end{equation}
where $S_h$ and its associated spaces $V_{S,h}$ can be the following:
\begin{itemize}
	\item $S_h=Id$, then $V_{S,h}=\R^{mN}$ and $\|\cdot\|_{V_{S},h} = \|\cdot\|_{L^2(\Omega_h)^m}$,
	\item $S_h=\nabla_h$, then $V_{S,h}=\R^{2mN}$ and $\|\cdot\|_{V_{S,h}} = \|\cdot\|_{\big(L^2(\Omega_h)^d\big)^m}$.
\end{itemize}
Analogously, the discrete dual problem is given by
\begin{equation} \tag{$\mathcal{D}_h$} \label{eq:dual_h}
	\begin{split}
		\max_{(q_{h,1},\mathbf{q_{h,2}})\in\R^{N}\times\R^{2mN}} \bigg\{ D_h(q_{h,1},\mathbf{q_{h,2}}) := -\frac{1}{2}\| T_h^* q_{h,1} + \nabla_h^*\mathbf{q_{h,2}} - \alpha_2 T_h^* g_h \|_{B_h^{-1}}^2 + \frac{\alpha_2}{2} \| g_h \|_{L^2(\Omega_h)}^2 \\ - \langle g_h, q_{h,1} \rangle_{L^2(\Omega_h)} - \chi_{|q_{h,1}|\leq \alpha_1} - \chi_{|\mathbf{q_{h,2}}|_{F,s}\leq \lambda} \bigg\}.
	\end{split}
\end{equation}
Here, the predicate $|q_{h,1}|\leq \alpha_1$ is to be understood as a condition on the degree of freedom, i.e.\ $|(q_{h,1})_i|\leq \alpha_1$, for all $i=1,\hdots,N$. The same applies for the predicate $|\mathbf{q_{h,2}}|_{F,s}\leq \lambda$.

\section{Finite Difference Method on Non-uniform Grid}
\label{sec:afdm}

As traditional finite differences cannot be directly applied to non-uniform grids, we derive in this section a respective finite difference scheme for such grids
tailored to discretize gradient and divergence operators. For $\mathbf{x}=(x,y)\in\Omega_h$, we denote by $h$ the length of the element with center $\mathbf{x}$, and we set $\mathbf{x_\text{E}} := (x+h,y)\in\Omega_h$ (see Figure \ref{fig:afd-regular}(a)) and $\mathbf{x_\text{S}} := (x,y+h)\in\Omega_h$ (Figure \ref{fig:afd-regular}(b)). A simple application of Taylor expansion yields the forward $x$- and $y$-derivative of $u\in C^1(\Omega)$ at $\mathbf{x}$, given by
$$ \partial_x u(\mathbf{x}) = \frac{u(\mathbf{x_\text{E}}) - u(\mathbf{x})}{h} + O(h) \qquad \text{and} \qquad \partial_y u(\mathbf{x}) = \frac{u(\mathbf{x_\text{S}}) - u(\mathbf{x})}{h} + O(h). $$
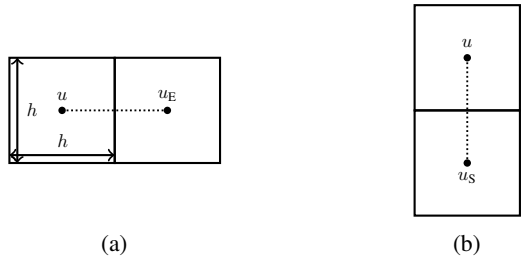
\begin{figure}[h]
	\centering
	\subfloat[]{ \label{fig:afd-regular:x}
		\begin{tikzpicture}[thick,scale=0.7, every node/.style={scale=0.7}]
			\draw[draw=white] (-0.7,-1.1) rectangle ++(5.4,4.2);
			\draw[draw=black] (0,0) rectangle ++(2,2);
			\draw[draw=black] (2,0) rectangle ++(2,2);
			\fill (1,1) circle[radius=2pt] node[yshift=8pt] {$u$};
			\fill (3,1) circle[radius=2pt] node[yshift=8pt] {$u_\text{E}$};
			\draw[line width=0.25mm, densely dotted] (1,1) -- (3,1);
			\draw[<->] (0,0.15) -- (2,0.15) node[yshift=8pt, xshift=-28pt] {$h$};
			\draw[<->] (0.15,0) -- (0.15,2) node[yshift=-28pt, xshift=8pt] {$h$};
		\end{tikzpicture}
	}
	\qquad
	\subfloat[]{ \label{fig:afd-regular:y}
		\begin{tikzpicture}[thick,scale=0.7, every node/.style={scale=0.7}]
			\draw[draw=white] (-0.7,-1.1) rectangle ++(5.4,4.2);
			\draw[draw=black] (1,1) rectangle ++(2,2);
			\draw[draw=black] (1,-1) rectangle ++(2,2);
			\fill (2,2) circle[radius=2pt] node[yshift=8pt] {$u$};
			\fill (2,0) circle[radius=2pt] node[yshift=-8pt] {$u_\text{S}$};
			\draw[line width=0.25mm, densely dotted] (2,2) -- (2,0);
		\end{tikzpicture}
	}
	\caption{Illustration of a regular node with $u:=u(\mathbf{x})$, $u_\text{E}:=u(\mathbf{x_\text{E}})$ and $u_\text{S}:=u(\mathbf{x_\text{S}})$.}
	\label{fig:afd-regular}
\end{figure}

In the following, we call \textit{dangling nodes} the nodes that have neighbors of a different size than its element's size. Since we assume that neighboring elements are only allowed to be one level finer or coarser, see Figure \ref{fig:counter-example}, we end up with 3 different dangling nodes.

 We call constellations as depicted in Figure \ref{fig:afd-dangling-1} \textit{dangling node 1}. To obtain a finite difference derivative for the situation in Figure \ref{fig:afd-dangling-1:x}, for $\mathbf{x}=(x,y)\in\Omega_h$, we set $\mathbf{x_\text{NE}} := (x+\frac{3}{4}h, y - \frac{1}{4}h)\in\Omega_h$ and $\mathbf{x_\text{SE}} := (x+\frac{3}{4}h, y + \frac{1}{4}h)\in\Omega_h$. Applying Taylor expansion, we get
$$ u(\mathbf{x_\text{NE}}) = u(\mathbf{x}) + \frac{3}{4}h\,\partial_x u(\mathbf{x}) - \frac{1}{4}h\,\partial_y u(\mathbf{x}) + O(h^2), \quad  u(\mathbf{x_\text{SE}}) = u(\mathbf{x}) + \frac{3}{4}h\,\partial_x u(\mathbf{x}) + \frac{1}{4}h\,\partial_y u(\mathbf{x}) + O(h^2). $$
Summing up the latter two equations yields the forward $x$-derivative of $u$ at $\mathbf{x}$, given by
$$ \partial_x u(\mathbf{x}) = \frac{\frac{u(\mathbf{x_\text{NE}}) + u(\mathbf{x_\text{SE}})}{2} - u(\mathbf{x})}{(3/4)\,h} + O(h). $$
Similarly, by denoting $\mathbf{x_\text{SW}} := (x-\frac{1}{4}h,y+\frac{3}{4}h)\in\Omega_h$ and $\mathbf{x_\text{SE}} := (x+\frac{1}{4}h,y+\frac{3}{4}h)\in\Omega_h$ (Figure \ref{fig:afd-dangling-1:y}), the forward $y$-derivative of $u$ at $\mathbf{x}$ reads
$$ \partial_y u(\mathbf{x}) = \frac{\frac{u(\mathbf{x_\text{SW}}) + u(\mathbf{x_\text{SE}})}{2} - u(\mathbf{x})}{(3/4)\,h} + O(h). $$ 
\begin{figure}[h]
	\centering
	\subfloat[]{ \label{fig:afd-dangling-1:x}
		\begin{tikzpicture}[thick,scale=0.7, every node/.style={scale=0.7}]
			\draw[draw=white] (-0.7,-1.1) rectangle ++(5.4,4.2);
			\draw[draw=black] (0,0) rectangle ++(2,2);
			\draw[draw=black] (2,0) rectangle ++(1,1);
			\draw[draw=black] (3,0) rectangle ++(1,1);
			\draw[draw=black] (2,1) rectangle ++(1,1);
			\draw[draw=black] (3,1) rectangle ++(1,1);
			\fill (1,1) circle[radius=2pt] node[yshift=8pt] {$u$};
			\fill (2.5,1.5) circle[radius=2pt] node[yshift=8pt] {$u_\text{NE}$};
			\fill (2.5,0.5) circle[radius=2pt] node[yshift=-8pt] {$u_\text{SE}$};
			\fill (2.5,1) circle[radius=2pt] node[xshift=8pt, yshift=8pt] {$u'$};
			\draw[line width=0.25mm, densely dotted] (2.5,1.5) -- (2.5,0.5);
			\draw[line width=0.25mm, densely dotted] (1,1) -- (2.5,1);
			\draw[<->] (0,0.15) -- (2,0.15) node[yshift=8pt, xshift=-28pt] {$h$};
			\draw[<->] (0.15,0) -- (0.15,2) node[yshift=-28pt, xshift=8pt] {$h$};
		\end{tikzpicture}
	}
	\qquad
	\subfloat[]{ \label{fig:afd-dangling-1:y}
		\begin{tikzpicture}[thick,scale=0.7, every node/.style={scale=0.7}]
			\draw[draw=white] (-0.7,-1.1) rectangle ++(5.4,4.2);
			\draw[draw=black] (1,1) rectangle ++(2,2);
			\draw[draw=black] (1,-1) rectangle ++(1,1);
			\draw[draw=black] (2,-1) rectangle ++(1,1);
			\draw[draw=black] (1,0) rectangle ++(1,1);
			\draw[draw=black] (2,0) rectangle ++(1,1);
			\fill (2,2) circle[radius=2pt] node[yshift=8pt] {$u$};
			\fill (1.5,0.5) circle[radius=2pt] node[yshift=-8pt] {$u_\text{SW}$};
			\fill (2.5,0.5) circle[radius=2pt] node[yshift=-8pt] {$u_\text{SE}$};
			\fill (2,0.5) circle[radius=2pt] node[yshift=8pt, xshift=8pt] {$u'$};
			\draw[line width=0.25mm, densely dotted] (1.5,0.5) -- (2.5,0.5);
			\draw[line width=0.25mm, densely dotted] (2,2) -- (2,0);
		\end{tikzpicture}
	}
	\caption{Illustration of a \textit{dangling node 1} with $u:=u(\mathbf{x})$. For (a): $u_\text{NE}:=u(\mathbf{x_\text{NE}})$, $u_\text{SE}:=u(\mathbf{x_\text{SE}})$ and $u':=\frac{1}{2}(u_\text{NE}+u_\text{SE})$. For (b): $u_\text{SW}:=u(\mathbf{x_\text{SW}})$, $u_\text{SE}:=u(\mathbf{x_\text{SE}})$ and $u':=\frac{1}{2}(u_\text{SW}+u_\text{SE})$.}
	\label{fig:afd-dangling-1}
\end{figure}
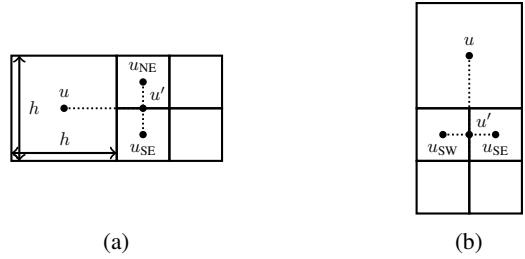

We call \textit{dangling node 2} the patterns as shown in Figure \ref{fig:afd-dangling-2-3}(a) and (b). To determine the finite difference derivative for the case illustrated in Figure \ref{fig:afd-dangling-2-3}(a), we set for $(x,y)\in\Omega_h$, $\mathbf{x_\text{N}} := (x,y-h)\in\Omega_h$ and $\mathbf{x_\text{E}} := (x+\frac{3}{2}h, y-\frac{1}{2}h)\in\Omega_h$. Applying Taylor expansion results in
\begin{equation} \label{eq:afd-dangling-2:1}
	u(\mathbf{x_\text{N}}) = u(\mathbf{x}) - h\, \partial_y u(\mathbf{x}) + O(h^2),
\end{equation}
\begin{equation} \label{eq:afd-dangling-2:2}
	u(\mathbf{x_\text{E}}) = u(\mathbf{x}) + \frac{3}{2}h\, \partial_x u(\mathbf{x}) - \frac{1}{2} h \,\partial_y u(\mathbf{x}) + O(h^2).
\end{equation}
We subtract two times \eqref{eq:afd-dangling-2:2} from \eqref{eq:afd-dangling-2:1} and we get the forward $x$-derivative of $u$ at $\mathbf{x}$
$$ \partial_x u(\mathbf{x}) = \frac{2 u(\mathbf{x_\text{E}}) - u(\mathbf{x_\text{N}}) - u(\mathbf{x})}{3\,h} + O(h). $$
In a similar manner, setting $\mathbf{x_\text{S}} := (x+\frac{1}{2}h, y+\frac{3}{2}h)\in\Omega_h$ and $\mathbf{x_\text{E}} := (x+h,y)\in\Omega_h$ (Figure \ref{fig:afd-dangling-2-3}(b)), we obtain the forward $y$-derivative of $u$ at $\mathbf{x}$
$$ \partial_y u(\mathbf{x}) = \frac{2 u(\mathbf{x_\text{S}}) - u(\mathbf{x_\text{E}})- u(\mathbf{x})}{3\,h} + O(h). $$
\begin{figure}[h]
	\centering
	\subfloat[]{ \label{fig:afd-dangling-2:x}
		\begin{tikzpicture}[thick,scale=0.7, every node/.style={scale=0.7}]
			\draw[draw=white] (-0.7,-1.1) rectangle ++(5.4,4.2);
			\draw[draw=black] (0,0) rectangle ++(1,1);
			\draw[draw=black] (1,0) rectangle ++(1,1);
			\draw[draw=black] (0,1) rectangle ++(1,1);
			\draw[draw=black] (1,1) rectangle ++(1,1);
			\draw[draw=black] (2,0) rectangle ++(2,2);
			\fill (1.5,0.5) circle[radius=2pt] node[yshift=8pt] {$u$};
			\fill (1.5,1.5) circle[radius=2pt] node[yshift=8pt] {$u_\text{N}$};
			\fill (3,1) circle[radius=2pt] node[yshift=8pt] {$u_\text{E}$};
			\fill (4.5,0.5) circle[radius=2pt] node[yshift=8pt] {$u'$};
			\draw[line width=0.25mm, densely dotted] (1.5,0.5) -- (4.5,0.5);
			\draw[line width=0.25mm, densely dotted] (1.5,1.5) -- (4.5,0.5);
			\draw[<->] (1,0.15) -- (2,0.15) node[yshift=-10pt, xshift=-14pt] {$h$};
			\draw[<->] (1.15,0) -- (1.15,1) node[yshift=-14pt, xshift=-12pt] {$h$};
		\end{tikzpicture}
	}
	\subfloat[]{ \label{fig:afd-dangling-2:y}
		\begin{tikzpicture}[thick,scale=0.7, every node/.style={scale=0.7}]
			\draw[draw=white] (-0.7,-1.1) rectangle ++(5.4,4.2);
			\draw[draw=black] (1,1) rectangle ++(1,1);
			\draw[draw=black] (2,1) rectangle ++(1,1);
			\draw[draw=black] (1,2) rectangle ++(1,1);
			\draw[draw=black] (2,2) rectangle ++(1,1);
			\draw[draw=black] (1,-1) rectangle ++(2,2);
			\fill (1.5,1.5) circle[radius=2pt] node[yshift=8pt] {$u$};
			\fill (2.5,1.5) circle[radius=2pt] node[yshift=8pt] {$u_\text{E}$};
			\fill (2,0) circle[radius=2pt] node[xshift=8pt] {$u_\text{S}$};
			\fill (1.5,-1.5) circle[radius=2pt] node[yshift=6pt, xshift=-8pt] {$u'$};
			\draw[line width=0.25mm, densely dotted] (1.5,1.5) -- (1.5,-1.5);
			\draw[line width=0.25mm, densely dotted] (2.5,1.5) -- (1.5,-1.5);
		\end{tikzpicture}
	}
	\hfill
	\subfloat[]{ \label{fig:afd-dangling-3:x}
		\begin{tikzpicture}[thick,scale=0.7, every node/.style={scale=0.7}]
			\draw[draw=white] (-0.7,-1.1) rectangle ++(5.4,4.2);
			\draw[draw=black] (0,0) rectangle ++(1,1);
			\draw[draw=black] (1,0) rectangle ++(1,1);
			\draw[draw=black] (0,1) rectangle ++(1,1);
			\draw[draw=black] (1,1) rectangle ++(1,1);
			\draw[draw=black] (2,0) rectangle ++(2,2);
			\fill (1.5,0.5) circle[radius=2pt] node[yshift=-8pt] {$u_\text{S}$};
			\fill (1.5,1.5) circle[radius=2pt] node[yshift=8pt] {$u$};
			\fill (3,1) circle[radius=2pt] node[yshift=8pt] {$u_\text{E}$};
			\fill (4.5,1.5) circle[radius=2pt] node[yshift=8pt] {$u'$};
			\draw[line width=0.25mm, densely dotted] (1.5,1.5) -- (4.5,1.5);
			\draw[line width=0.25mm, densely dotted] (1.5,0.5) -- (4.5,1.5);
			\draw[<->] (1,1.15) -- (2,1.15) node[yshift=-10pt, xshift=-14pt] {$h$};
			\draw[<->] (1.15,1) -- (1.15,2) node[yshift=-14pt, xshift=-12pt] {$h$};
		\end{tikzpicture}
	}
	\subfloat[]{ \label{fig:afd-dangling-3:y}
		\begin{tikzpicture}[thick,scale=0.7, every node/.style={scale=0.7}]
			\draw[draw=white] (-0.7,-1.1) rectangle ++(5.4,4.2);
			\draw[draw=black] (1,1) rectangle ++(1,1);
			\draw[draw=black] (2,1) rectangle ++(1,1);
			\draw[draw=black] (1,2) rectangle ++(1,1);
			\draw[draw=black] (2,2) rectangle ++(1,1);
			\draw[draw=black] (1,-1) rectangle ++(2,2);
			\fill (1.5,1.5) circle[radius=2pt] node[yshift=8pt] {$u_\text{W}$};
			\fill (2.5,1.5) circle[radius=2pt] node[yshift=8pt] {$u$};
			\fill (2,0) circle[radius=2pt] node[xshift=-8pt] {$u_\text{S}$};
			\fill (2.5,-1.5) circle[radius=2pt] node[yshift=6pt, xshift=8pt] {$u'$};
			\draw[line width=0.25mm, densely dotted] (1.5,1.5) -- (2.5,-1.5);
			\draw[line width=0.25mm, densely dotted] (2.5,1.5) -- (2.5,-1.5);
		\end{tikzpicture}
	}
	\caption{Illustration of a \textit{dangling node 2} (a)-(b) and \textit{dangling node 3} (c)-(d) with $u:=u(\mathbf{x})$. For (a): $u_\text{E}:=u(\mathbf{x_\text{E}})$, $u_\text{N}:=u(\mathbf{x_\text{N}})$ and $u':=2u_\text{E}-u_\text{N}$. For (b): $u_\text{S}:=u(\mathbf{x_\text{S}})$, $u_\text{E}:=u(\mathbf{x_\text{E}})$ and $u':=2u_\text{S}-u_\text{E}$. For (c): $u_\text{E}:=u(\mathbf{x_\text{E}})$, $u_\text{S}:=u(\mathbf{x_\text{S}})$ and $u':=2u_\text{E}-u_\text{S}$. For (d): $u_\text{S}:=u(\mathbf{x_\text{S}})$, $u_\text{W}:=u(\mathbf{x_\text{W}})$ and $u':=2u_\text{S}-u_\text{W}$.}
	\label{fig:afd-dangling-2-3}
\end{figure}
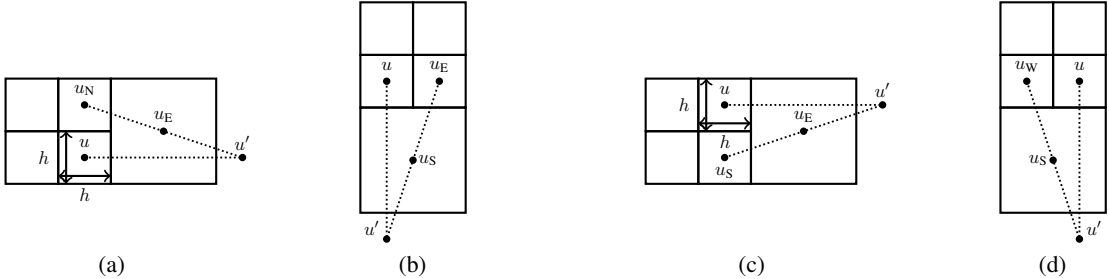

We refer to configurations as depicted in Figure \ref{fig:afd-dangling-2-3}(c) and (d) as \textit{dangling node 3}. To calculate a finite difference derivative for the scenario shown in Figure \ref{fig:afd-dangling-2-3}(c), we set, for $(x,y)\in\Omega_h$, $\mathbf{x_\text{S}} := (x,y+h)\in\Omega_h$ and $\mathbf{x_\text{E}} := (x+\frac{3}{2}h, y+\frac{1}{2}h)\in\Omega_h$. We use Taylor expansion, and 
we obtain the forward $x$-derivative of $u$ at $\mathbf{x}$, given by
$$ \partial_x u(\mathbf{x}) = \frac{2 u(\mathbf{x_\text{E}}) - u(\mathbf{x_\text{S}}) - u(\mathbf{x})}{3\,h} + O(h). $$
Likewise, setting $\mathbf{x_\text{S}} := (x-\frac{1}{2}h, x+\frac{3}{2}h)\in\Omega_h$ and $\mathbf{x_\text{W}} := (x-h,y)\in\Omega_h$ (Figure \ref{fig:afd-dangling-2-3}(d)), we have the forward derivative of $u$ with respect to $y$ at $\mathbf{x}$
$$ \partial_y u(\mathbf{x}) = \frac{2 u(\mathbf{x_\text{S}}) - u(\mathbf{x_\text{W}}) - u(\mathbf{x})}{3\,h} + O(h). $$
To sum up, if we denote $u_\square := u(\mathbf{x_\square})$ where $\square\in\{\text{N},\text{E},\text{S},\text{W},\text{NE},\text{NW},\text{SE},\text{SW}\}$ and $h$ is the length of the element with center $\mathbf{x}\in\Omega_h$, we define $\partial_x^+ u$ as the discrete forward $x$-derivative of $u$ at $\mathbf{x}$, $\partial_y^+ u$ as the discrete forward $y$-derivative of $u$ at $\mathbf{x}$, $\partial_x^- u$ as the discrete backward $x$-derivative of $u$ at $\mathbf{x}$ and $\partial_y^- u$ as the discrete backward $y$-derivative of $u$ at $\mathbf{x}$ by:

\textbullet ~ on \textit{regular node} (Figure \ref{fig:afd-regular}):
$$ \partial_x^+ u = \frac{u_\text{E} - u}{h}, \quad \partial_y^+ u = \frac{u_\text{S} - u}{h}, \quad\text{and}\qquad \partial_x^- u = \frac{u - u_\text{W}}{h}, \quad \partial_y^- u = \frac{u - u_\text{N}}{h}. $$
\textbullet ~ on \textit{dangling node 1} (Figure \ref{fig:afd-dangling-1}):
$$ \partial_x^+ u = \frac{\frac{u_\text{NE} + u_\text{SE}}{2} - u}{(3/4)\,h}, \quad \partial_y^+ u = \frac{\frac{u_\text{SW} + u_\text{SE}}{2} - u}{(3/4)\,h}, \qquad\text{and}\quad \partial_x^- u = \frac{u - \frac{u_\text{NW} + u_\text{SW}}{2}}{(3/4)\,h}, \quad \partial_y^- u = \frac{u - \frac{u_\text{NW} + u_\text{NE}}{2}}{(3/4)\,h}. $$
\textbullet ~ on \textit{dangling node 2} (Figure \ref{fig:afd-dangling-2-3}(a) and (b)):
$$ \partial_x^+ u = \frac{2 u_\text{E} - u_\text{N} - u}{3\,h}, \quad \partial_y^+ u = \frac{2 u_\text{S} - u_\text{E} - u}{3\,h}, \qquad\text{and}\quad \partial_x^- u = \frac{u -2 u_\text{W} + u_\text{N}}{3\,h}, \quad \partial_y^- u = \frac{u -2 u_\text{N} + u_\text{E}}{3\,h}. $$
\textbullet ~ on \textit{dangling node 3} (Figure \ref{fig:afd-dangling-2-3}(c) and (d)):
$$ \partial_x^+ u = \frac{2 u_\text{E} - u_\text{S} - u}{3\,h}, \quad \partial_y^+ u = \frac{2 u_\text{S} - u_\text{W} - u}{3\,h}, \quad\text{and}\qquad \partial_x^- u = \frac{u - 2 u_\text{W} + u_\text{S}}{3\,h}, \quad \partial_y^- u = \frac{u - 2 u_\text{N} + u_\text{W}}{3\,h}. $$
We define the discrete gradient operator $\nabla_h$ by the concatenation of the forward $x$-derivative and the $y$-derivative described above for all $\mathbf{x}\in\Omega_h$, i.e.\ $\nabla_h := (\partial_x^+\ \ \partial_y^+)^T$ . Similarly, we define the discrete divergence operator $\divop_h$ by $\divop_h := (\partial_x^-\ \ \partial_y^-)$.

On a uniform mesh in our discrete setting, the gradient and divergence operators are adjoint to each other, just as they are in the continuous case, i.e.\ we have $(\nabla_h)^* = -\divop_h$ when adding homogeneous Neumann boundaries condition on the discrete gradient and homogeneous Dirichlet boundaries condition on the discrete divergence. However, on non-uniform meshes, the discrete gradient and divergence operators are no longer adjoint. Indeed, let us consider the grid in Figure \ref{fig:counter-example}. In this particular example
one can compute the discrete operators
for $m=1$, and show that $(\nabla_h)^* = (\nabla_h)^T \neq -\divop_h = (\nabla^*)_h$. Note that the order of the operations matters: discretizing and computing the adjoint is different from taking the adjoint and then discretize.

Based on the above derived forward and backward finite differences on non-uniform grids, we additionally propose centered finite difference schemes for non-uniform grids. In particular, these centered finite difference schemes $\partial_x^\text{c}$ and $\partial_y^\text{c}$ are defined as the average of the forward and backward difference scheme, i.e.
$$ \partial_x^\text{c} := \frac{1}{2}(\partial_x^+ + \partial_x^-)\qquad\text{and}\qquad \partial_y^\text{c} := \frac{1}{2}(\partial_y^+ + \partial_y^-). $$

\section{Impact of Refinement on the Total Variation} \label{sec:analysis-adaptive}

In \cite{Condat2017} it has been shown that the discrete total variation is dependent on the number of points in the discrete domain $\Omega_h$. Here, we aim to quantify the change of the discrete total variation with respect to the grid. Let $\Omega_h$ be any quad-mesh and let us consider its one-time refinement (we refine at most once every element) $\widetilde{\Omega_h}$ (see for instance Figure \ref{fig:counter-example-weight}). We denote by $I$ (resp. $\widetilde I$) the set of indices $\{1,\hdots,N\}$ (resp. $\{1,\hdots,\widetilde N\}$), where $N$ (resp. $\widetilde{N}$) is the number of degree of freedom of $\Omega_h$ (resp. $\widetilde{\Omega_h}$). For vectors $\mathbf{u_h}\in\R^{mN}$ and $\widetilde{ \mathbf{u_h}}\in\R^{m\widetilde{N}}$, we define respective discrete (weighted) total variations by
$$ \TV(\mathbf{u_h}) := \sum_{i\in I} h_i^2 (|\nabla_h \mathbf{u_h}|_{F,r})_i, \qquad\text{and}\qquad \widetilde\TV_{\mu}(\widetilde{\mathbf{u_h}}) := \sum_{i\in\widetilde I} \widetilde{h_i}^2 \mu_i (|\widetilde{\nabla_h} \widetilde{\mathbf{u_h}}|_{F,r})_i, $$
where $\mu\in\R^{\widetilde{N}}$ and $\nabla_h,\widetilde{\nabla_h}$ are any suitable discrete gradient on $\Omega_h$ and on $\widetilde{\Omega_h}$.
We denote by $\mathcal{R}$ the set of indices of elements of $\Omega_h$ that have been refined to create the mesh $\widetilde{\Omega_h}$
and we denote by $\mathcal{N}$ the indices of the refined elements and their neighbors (north, east, south and west), i.e. $ \mathcal{N} := \{ i\in I\ |\ \exists j\in\mathcal{R},\ E_i\ \text{share an edge with}\ E_j \} $.
Thus we can write
$$ \TV(\mathbf{u_h}) = \sum_{i\in I\setminus\mathcal{N}} h_i^2 (|\nabla_h \mathbf{u_h}|_{F,r})_i + \sum_{i\in \mathcal{N}\setminus\mathcal{R}} h_i^2 (|\nabla_h \mathbf{u_h}|_{F,r})_i + \sum_{i\in \mathcal{R}} h_i^2 (|\nabla_h \mathbf{u_h}|_{F,r})_i. $$
We define the mapping $\widetilde\sigma$ which maps the indices of an element of $\Omega_h$ to the indices of the corresponding (or the $4$ corresponding) element(s) of $\widetilde{\Omega_h}$, i.e.\ for $i\in I$
 $$ \widetilde\sigma(i)=\begin{cases}
	\{i\}, & \text{if}\ i\not\in\mathcal{R}, \\
	\{i_1,i_2,i_3,i_4\}, & \text{if}\ i\in\mathcal{R},
\end{cases} 
$$
where $i_k\in\widetilde{I}$, $k=1,\hdots,4$, are the indices of the sub-elements.
For $J\subseteq I$, let us define $ \widetilde\sigma(J) = \bigcup_{j\in J}\widetilde\sigma(j) $.
Note that $\widetilde{I} = \widetilde\sigma(I) = \widetilde\sigma\big((I\setminus\mathcal{N})\cup (\mathcal{N}\setminus\mathcal{R})\cup \mathcal{R}\big) = \widetilde\sigma(I\setminus\mathcal{N})\cup \widetilde\sigma(\mathcal{N}\setminus\mathcal{R})\cup\widetilde\sigma( \mathcal{R}) = (I\setminus\mathcal{N}) \cup (\mathcal{N}\setminus\mathcal{R})\cup\widetilde\sigma( \mathcal{R})$.
Since $\widetilde{h_i} = h_i$ for $i\in I\setminus\mathcal{R}$ and $\widetilde{h_i} = \frac{h_i}{2}$ for $i\in\mathcal{R}$, we obtain
\begin{align*}
	\widetilde\TV_{\mu}(\widetilde{\mathbf{u_h}}) & = \sum_{i\in I\setminus\mathcal{N}} \widetilde{h_i}^2 \mu_i (|\widetilde{\nabla_h} \widetilde{\mathbf{u_h}}|_{F,r})_i + \sum_{i\in \mathcal{N}\setminus\mathcal{R}} \widetilde{h_i}^2 \mu_i (|\widetilde{\nabla_h} \widetilde{\mathbf{u_h}}|_{F,r})_i + \sum_{i\in \widetilde\sigma(\mathcal{R})} \widetilde{h_i}^2 \mu_i (|\widetilde{\nabla_h}\widetilde{\mathbf{u_h}}|_{F,r})_i \\
	%
	%
	& = \sum_{i\in I\setminus\mathcal{N}} h_i^2 \mu_i (|\widetilde{\nabla_h} \widetilde{\mathbf{u_h}}|_{F,r})_i + \sum_{i\in \mathcal{N}\setminus\mathcal{R}} h_i^2 \mu_i (|\widetilde{\nabla_h} \widetilde{\mathbf{u_h}}|_{F,r})_i + \sum_{i\in \mathcal{R}} \frac{h_i^2}{4} \sum_{k=1}^4 \mu_{i_k} (|\widetilde{\nabla_h} \widetilde{\mathbf{u_h}}|_{F,r})_{i_k}.
\end{align*}
We denote by $\pi: \R^N\to\R^{\widetilde{N}}$ the projection of a discrete function on $\Omega_h$ into $\widetilde{\Omega_h}$ and by $\widetilde\pi: \R^{\widetilde{N}}\to\R^N$ the projection of a discrete function on $\widetilde{\Omega_h}$ into $\Omega_h$. More precisely, for $i\in I$
$$ (\pi \mathbf{u_h})_j = (\mathbf{u_h})_i,\ \text{for}\ j\in\widetilde\sigma(i) \qquad\text{and}\qquad (\widetilde\pi \widetilde{\mathbf{u_h}})_i = \frac{1}{\sharp \widetilde\sigma(i)}\sum_{j\in\widetilde\sigma(i)} (\widetilde{\mathbf{u_h}})_j, $$
where $\sharp \widetilde\sigma(i)$ is the cardinality of $\widetilde\sigma(i)$. Then, using that $(|\widetilde{\nabla_h} \pi \mathbf{u_h}|_{F,r})_i = (|\nabla_h \mathbf{u_h}|_{F,r})_i$, for $i\in I\setminus\mathcal{N}$, it yields
\begin{align*}
	\TV(\mathbf{u_h}) - \widetilde\TV_{\mu}(\pi \mathbf{u_h}) & = \sum_{i\in I\setminus\mathcal{N}} h_i^2(|\nabla_h \mathbf{u_h}|_{F,r})_i + \sum_{i\in \mathcal{N}\setminus\mathcal{R}} h_i^2 (|\nabla_h \mathbf{u_h}|_{F,r})_i + \sum_{i\in \mathcal{R}} h_i^2 (|\nabla_h \mathbf{u_h}|_{F,r})_i \\ & \hspace{1.5cm} - \sum_{i\in I\setminus\mathcal{N}} h_i^2 \mu_i (|\widetilde{\nabla_h} \pi \mathbf{u_h}|_{F,r})_i - \sum_{i\in \mathcal{N}\setminus\mathcal{R}} h_i^2 \mu_i (|\widetilde{\nabla_h} \pi \mathbf{u_h}|_{F,r})_i \\ & \hspace{1.5cm} - \sum_{i\in \mathcal{R}} \frac{h_i^2}{4} \sum_{k=1}^4 \mu_{i_k} (|\widetilde{\nabla_h} \pi \mathbf{u_h}|_{F,r})_{i_k} \\
	%
	%
	& = \sum_{i\in I\setminus\mathcal{N}} h_i^2 (1-\mu_i) (|\nabla_h \mathbf{u_h}|_{F,r})_i + \sum_{i\in \mathcal{N}\setminus\mathcal{R}} h_i^2 \big((|\nabla_h \mathbf{u_h}|_{F,r})_i - \mu_i (|\widetilde{\nabla_h} \pi \mathbf{u_h}|)_i\big) \\ & \hspace{1.5cm} + \sum_{i\in \mathcal{R}} h_i^2 \left( (|\nabla_h \mathbf{u_h}|_{F,r})_i - \frac{1}{4} \sum_{k=1}^4 \mu_{i_k} (|\widetilde{\nabla_h} \pi \mathbf{u_h}|_{F,r})_{i_k}\right).
\end{align*}
In order to see the quantitative change of the discrete total variation we need to seek for $\mu$ such that $\TV(\mathbf{u_h}) - \widetilde\TV_{\mu}(\pi \mathbf{u_h}) = 0$. For that, we suppose that $\mu$ is constant on every element of $\Omega_h$, in particular, we have $\mu_{i_k} =: \mu_i$ for $i\in\mathcal{R}$ and $k=1,2,3,4$. Then
\begin{multline*}
	\TV(\mathbf{u_h}) - \widetilde\TV_{\mu}(\pi \mathbf{u_h}) = \sum_{i\in I\setminus\mathcal{N}} h_i^2 (1-\mu_i) (|\nabla_h \mathbf{u_h}|_{F,r})_i + \sum_{i\in \mathcal{N}\setminus\mathcal{R}} h_i^2 \big( (|\nabla_h \mathbf{u_h}|_{F,r})_i - \mu_i (|\widetilde{\nabla_h} \pi \mathbf{u_h}|_{F,r})_i \big) \\ + \sum_{i\in \mathcal{R}} h_i^2 \left( (|\nabla_h \mathbf{u_h}|_{F,r})_i - \mu_i  \frac{1}{4}\sum_{k=1}^4 (|\widetilde{\nabla_h} \pi \mathbf{u_h}|_{F,r})_{i_k}\right).
\end{multline*}
The above expression vanishes if the following system of equations holds: 
$$ \left\{\begin{array}{ll}
	(1-\mu_i) (|\nabla_h \mathbf{u_h}|_{F,r})_i = 0, & i\in I\setminus\mathcal{N}, \\
	(|\nabla_h \mathbf{u_h}|_{F,r})_i - \mu_i (|\widetilde{\nabla_h} \pi \mathbf{u_h}|_{F,r})_i = 0, & i\in\mathcal{N}\setminus\mathcal{R}, \\
	(|\nabla_h \mathbf{u_h}|_{F,r})_i - \frac{1}{4} \mu_i \sum_{k=1}^4 (|\widetilde{\nabla_h} \pi \mathbf{u_h}|_{F,r})_{i_k} = 0, & i\in\mathcal{R}.
\end{array}\right . $$
These equations are satisfied if (note that $(\widetilde\pi|\widetilde{\nabla_h} \pi \mathbf{u_h}|_{F,r})_i = 0$ if and only if $(|\nabla_h \mathbf{u_h}|_{F,r})_i = 0$)
\begin{equation} \label{eq:lambda}
	 \mu_i := \begin{cases}
		\frac{(|\nabla_h \mathbf{u_h}|_{F,r})_i}{(\widetilde\pi|\widetilde{\nabla_h} \pi \mathbf{u_h}|_{F,r})_i}, & \text{if}\ (\widetilde\pi|\widetilde{\nabla_h} \pi \mathbf{u_h}|_{F,r})_i \neq 0, \\
		1, & \text{otherwise}.
	\end{cases}
\end{equation}
for all $i\in I$. Indeed, if $i\in I\setminus\mathcal{N}$, then $(|\nabla_h \mathbf{u_h}|_{F,r})_i=(\widetilde\pi|\widetilde{\nabla_h} \pi \mathbf{u_h}|_{F,r})_i$ and hence $\mu_i=1$, otherwise, if $i\in\mathcal{N}\setminus\mathcal{R}$, then we have $(\widetilde\pi|\widetilde{\nabla_h} \pi \mathbf{u_h}|_{F,r})_i = (|\widetilde{\nabla_h} \pi \mathbf{u_h}|_{F,r})_i$. Finally, $(\widetilde\pi|\widetilde{\nabla_h} \pi \mathbf{u_h}|_{F,r})_i = \frac{1}{4} \sum_{k=1}^4 (|\widetilde{\nabla_h} \pi \mathbf{u_h}|_{F,r})_{i_k}$ if $i\in\mathcal{R}$.
Therefore, for an adaptive grid, refining elements of the mesh may impact the value of the total variation.
In fact, the continuous total variation of a function $f$ differs in general from the discrete total variation of the discretized $f$.

\begin{example}
	Let us study the particular case where $m=1$. \\
	\textbf{(a)} We suppose that $\widetilde{\Omega_h}$ is a global refinement of the uniform mesh $\Omega_h$. On the uniform mesh, we use the multi-indices $i=(a,b)$ to identify an element of $\Omega_h$ and the multi-index $i_k=(a,b,k)$ with $k=1,2,3,4$, to identify a sub-element of $\widetilde{\Omega_h}$. The forward finite difference scheme writes (we write $u_{a,b}$ instead of $(\mathbf{u_h})_{a,b}$)
	$$ (|\nabla_h \mathbf{u_h}|_r)_{a,b} = \frac{1}{h} \big( |u_{a+1,b} - u_{a,b}|^r + |u_{a,b+1} - u_{a,b}|^r \big)^{1/r}, $$
	and
	$$ \begin{array}{ll}
		(|\widetilde{\nabla_h} \pi \mathbf{u_h}|_r)_{a,b,1} = 0, &
		(|\widetilde{\nabla_h} \pi \mathbf{u_h}|_r)_{a,b,2} = \frac{2}{h} |u_{a+1,b}-u_{a,b}|, \\
		(|\widetilde{\nabla_h} \pi \mathbf{u_h}|_r)_{a,b,3} = \frac{2}{h} |u_{a,b+1}-u_{a,b}|, &
		(|\widetilde{\nabla_h} \pi \mathbf{u_h}|_r)_{a,b,4} = \frac{2}{h} \big( |u_{a+1,b} - u_{a,b}|^r + |u_{a,b+1} - u_{a,b}|^r \big)^{1/r}. \\
	\end{array} $$
	Then
	$$ (\widetilde{\pi}|\widetilde{\nabla_h} \pi \mathbf{u_h}|_r )_{a,b} = \frac{1}{2h} |u_{a+1,b}-u_{a,b}| + \frac{1}{2h} |u_{a,b+1}-u_{a,b}| + \frac{1}{2h} \big( |u_{a+1,b} - u_{a,b}|^r + |u_{a,b+1} - u_{a,b}|^r \big)^{1/r}. $$
	
	If $r=1$ (anisotropic total variation), we have $ \mu_{a,b} = 1 $. Therefore, for a uniform grid, the size of the elements of the mesh have no impact on the value of the anisotropic total variation. If $r=2$ (isotropic total variation), using the above computations, we obtain if $\sqrt{(u_{a+1,b}-u_{a,b})^2} + \sqrt{(u_{a,b+1}-u_{a,b})^2} + \sqrt{(u_{a+1,b} - u_{a,b})^2 + (u_{a,b+1} - u_{a,b})^2} \neq 0$
	$$ 0 < \mu_{a,b} = \frac{ 2\sqrt{(u_{a+1,b} - u_{a,b})^2 + (u_{a,b+1} - u_{a,b})^2}}{ \sqrt{(u_{a+1,b}-u_{a,b})^2} + \sqrt{(u_{a,b+1}-u_{a,b})^2} + \sqrt{(u_{a+1,b} - u_{a,b})^2 + (u_{a,b+1} - u_{a,b})^2}} \leq 1, $$
	and the size of the elements of the mesh indeed have impact on the value of the isotropic total variation, in particular, it increases its value. Hence, the solution of the discrete problem \eqref{eq:primal_h} depends on the mesh-size. \\
	\textbf{(b)} Now we consider any grid $\Omega_h$ and its refinement $\widetilde{\Omega_h}$, not necessarily uniform. If $r=1$, let us consider the grid $\Omega_h$ on Figure \ref{fig:counter-example-weight}, and an associated one-time refined grid $\widetilde{\Omega_h}$
	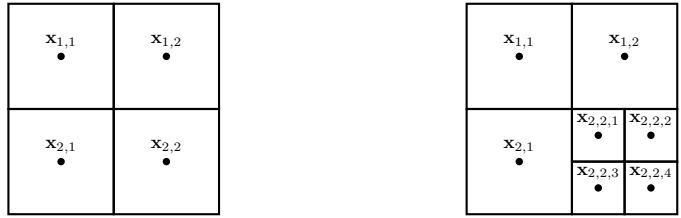
\begin{figure}[h]
		\centering
		\subfloat{
			\begin{tikzpicture}[thick,scale=0.7, every node/.style={scale=0.7}]
				\draw[draw=black] (1,3) rectangle ++(2,2);
				\draw[draw=black] (1,5) rectangle ++(2,2);
				\draw[draw=black] (3,3) rectangle ++(2,2);
				\draw[draw=black] (3,5) rectangle ++(2,2);
				\fill (2,6) circle[radius=2pt] node[yshift=8pt] {$\mathbf{x}_{1,1}$};
				\fill (4,6) circle[radius=2pt] node[yshift=8pt] {$\mathbf{x}_{1,2}$};
				\fill (2,4) circle[radius=2pt] node[yshift=8pt] {$\mathbf{x}_{2,1}$};
				\fill (4,4) circle[radius=2pt] node[yshift=8pt] {$\mathbf{x}_{2,2}$};
			\end{tikzpicture}
		}
		\hspace{3cm}
		\subfloat{
			\begin{tikzpicture}[thick,scale=0.7, every node/.style={scale=0.7}]
				\draw[draw=black] (1,3) rectangle ++(2,2);
				\draw[draw=black] (1,5) rectangle ++(2,2);
				\draw[draw=black] (3,3) rectangle ++(1,1);
				\draw[draw=black] (4,3) rectangle ++(1,1);
				\draw[draw=black] (3,4) rectangle ++(1,1);
				\draw[draw=black] (4,4) rectangle ++(1,1);
				\draw[draw=black] (3,5) rectangle ++(2,2);
				\fill (2,6) circle[radius=2pt] node[yshift=8pt] {$\mathbf{x}_{1,1}$};
				\fill (4,6) circle[radius=2pt] node[yshift=8pt] {$\mathbf{x}_{1,2}$};
				\fill (2,4) circle[radius=2pt] node[yshift=8pt] {$\mathbf{x}_{2,1}$};
				\fill (3.5,4.5) circle[radius=2pt] node[yshift=8pt] {$\mathbf{x}_{2,2,1}$};
				\fill (4.5,4.5) circle[radius=2pt] node[yshift=8pt] {$\mathbf{x}_{2,2,2}$};
				\fill (3.5,3.5) circle[radius=2pt] node[yshift=8pt] {$\mathbf{x}_{2,2,3}$};
				\fill (4.5,3.5) circle[radius=2pt] node[yshift=8pt] {$\mathbf{x}_{2,2,4}$};
			\end{tikzpicture}
		}
		\caption{One refinement example with $\Omega_h$ on the left and a one-refinement $\widetilde{\Omega_h}$ on the right.}
		\label{fig:counter-example-weight}
	\end{figure} 
	Simple computations give $ \mu_{1,1} = 1$, $\mu_{1,2} = \frac{3}{4}$ if $u_{2,2}\neq u_{1,2}$, $\mu_{2,1} = \frac{3}{4}$ if $u_{2,2}\neq u_{2,1}$ and $\mu_{2,2} = 1 $,
	which means that the total variation is changing after refinement. If $r=2$, this result is also true for any (non)-uniform grid since it always contains some uniform parts (if the function is not constant on the uniform part of the grid).
\end{example}

In practice, we look at a fixed discrete data $g$ and refining the grid alters the balance between the total variation and the other terms in the discrete problem. In order to compensate this change of energy, one could change the value of the parameter according to \eqref{eq:lambda}. However, we note that while \eqref{eq:lambda} is compensating the change of the discrete total variation, it leads to a locally weighted total variation, where the weights are chosen considering only very local (one pixel) information. Hence, compensating with \eqref{eq:lambda} in total variation minimization should only be used with caution as a very local criterion might lead to a very dramatic change of the weight from one pixel to the next introducing artifacts in the restoration, as we observed it in numerical experiments. Note that locally adaptive parameter strategies for total variation minimization, as in \cite{Dong2011, Hintermuller2018, Langer2017Feb, Langer2019}, consider overlapping windows of reasonable size to establish a local parameter for total variation minimization. This avoids a dramatic change of the local weights and helps to avoid artifacts.

\section{Primal-Dual Gap Error Estimator} \label{sec:error}

In this section we derive an indicator for the mesh adaptivity, namely the primal-dual gap error estimator.
For the continuous problem we have strong duality \cite[Theorem 4.2, Chapter 3]{Ekeland1999}, namely, if $\mathbf{u}$ is a solution of \eqref{eq:primal} and $(p_1,\mathbf{p_2})$ is a solution of \eqref{eq:dual}, we have 
$ E(\mathbf{u}) = D(p_1,\mathbf{p_2}) $.
However, for non-uniform grids, this last statement might not be true as the discrete gradient $\nabla_h$ is not the adjoint of the discrete divergence $\divop_h$. We have the following primal-dual gap error, whose proof is provided in Appendix \ref{app:pd-gap-proof}: \\

\begin{proposition} \label{prop:pd-gap-error} Let $\mathbf{u}\in H^1(\Omega)^m$ be a solution of \eqref{eq:primal} and $\mathbf{u_h}\in\R^{mN}$ be a solution of \eqref{eq:primal_h}. Then, for all $\mathbf{v_h}\in\R^{mN}$ and $(q_{h,1},\mathbf{q_{h,2}})\in\R^{N}\times\R^{2mN}$, we have
	$$ \frac{1}{2}\|I_h \mathbf{u_h}-\mathbf{u}\|_B^2 \leq \eta_h(\mathbf{v_h}, q_{h,1},\mathbf{q_{h,2}}) + c \| I_h g_h - g \|_{L^2(\Omega)} + R_h(\mathbf{u_h}) +  Q_h(p_{h,1}, \mathbf{p_{h,2}}), $$
	where $c\geq 0$ and
	\begin{align*}
		& \eta_h(\mathbf{v_h}, q_{h,1},\mathbf{q_{h,2}}) := E_h(\mathbf{v_h}) - D_h(q_{h,1},\mathbf{q_{h,2}}), \\
		& R_h(\mathbf{u_h}) := E(I_h \mathbf{u_h}) - E_h(\mathbf{u_h}), \\
		& Q_h(p_{h,1}, \mathbf{p_{h,2}}) := \frac{1}{2}\Big(\|T^* I_h p_{h,1} + \nabla^*I_h \mathbf{p_{h,2}} + \alpha_2 T^* g\|_{B^{-1}}^2 - \| T_h^* q_{h,1} + \nabla_h^*\mathbf{q_{h,2}} - \alpha_2 T_h^* g_h \|_{B_h^{-1}}^2 \Big). \\
	\end{align*}
\end{proposition}

We notice that
\begin{align*}
	E(I_h \mathbf{u_h}) - E_h(\mathbf{u_h}) & = \alpha_1\Big(\int_\Omega |T I_h \mathbf{u_h}-g|\ dx - \sum_{i=1}^N h_i^2 (|T_h\mathbf{u_h}-g_h|)_i\big)\Big) \\ & \hspace{1cm} + \frac{\alpha_2}{2}\Big( \|T I_h \mathbf{u_h} - g\|_{L^2(\Omega)}^2 - \|T_h\mathbf{u_h} - g_h\|_{L^2(\Omega_h)}^2 \Big) + \frac{\beta}{2}\Big( \|S I_h \mathbf{u_h}\|_{V_S}^2  - \|S_h\mathbf{u_h}\|_{V_{S,h}}^2 \Big) \\ & \hspace{1cm} + \lambda\Big( \int_\Omega |\nabla I_h \mathbf{u_h}|_{\text{F},r}\ dx - \sum_{i=1}^N h_i^2 (|\nabla_h \mathbf{u_h}|_{\text{F},r})_i \Big).
\end{align*}
In general, $T I_h u_h$ and $S I_h u_h$ are not piecewise constant despite $I_h u_h$ is. That is why, for example, the second term becomes
$$ \|T I_h \mathbf{u_h} - g\|_{L^2(\Omega)}^2 - \|T_h\mathbf{u_h} - g_h\|_{L^2(\Omega_h)}^2 = \sum_{i=1}^N \int_{E_i} (T I_h \mathbf{u_h} - g)^2 - \big( (T I_h \mathbf{u_h})(\mathbf{x}_i) - I_h g_h(x) \big)^2\ dx, $$
where $E_i$ is the $i$-th element of the mesh, $h_i$ is its length and $\mathbf{x}_i$ is its center. A similar behavior occurs for the first and third terms. In fact, these terms represent the discretization error of the operators $T$ and $S$ and of the data $g$. The term $ \int_\Omega |\nabla I_h \mathbf{u_h}|_{\text{F},r}\ dx - \sum_{i=1}^N h_i^2 (|\nabla_h \mathbf{u_h}|_{\text{F},r})_i $ corresponds to the discretization error of the total variation. In the piecewise constant setting, it has been proved that this term does not go to $0$ when $h$ tends to $0$ \cite{Bartels2012}. It is the reason why we do not expect the term $E(I_h \mathbf{u_h}) - E_h(\mathbf{u_h})$ to go to $0$ when $h$ goes to $0$ and why we do not include this term in our mesh refinement indicator.

The term $Q_h$ is the discretization error of the operators and of the data term $g$. In general, we do not have $(B^{-1})_h = (B_h)^{-1}$ and $(B^{-1} T^*)_h = (B^{-1})_h (T^*)_h$. Moreover, we may have to assume the convergence of the discrete operators to the continuous one in suitable senses. That is why the study of $Q_h$ is difficult and may be the topic of further research. Consequently, we chose not to use $D(I_h p_{h,1},I_h \mathbf{p_{h,2}}) - D_h(p_{h,1},\mathbf{p_{h,2}})$ in our refinement indicator.

As a consequence of the previous proposition, we use a local version $(\eta_h)_i$ of $\eta_h$, i.e.\ such that $\eta_h = \sum_{i=1}^N (\eta_h)_i$, as indicator for the grid adaptivity defined as: for all $\mathbf{v_h}\in\R^{mN}$ and $(q_{h,1},\mathbf{q_{h,2}})\in\R^{N}\times\R^{2mN}$ such that $|q_{h,1}|\leq \alpha_1$ and $|\mathbf{q_{h,2}}|_{F,s}\leq \lambda$
\begin{multline} \label{eq:grid-ref-indicator}
	h_i^{-2}(\eta_h)_i = \alpha_1 |(T_h\mathbf{v_h}-g_h)_i| + \frac{\alpha_2}{2} (T_h\mathbf{v_h} - g_h)_i^2 + \frac{\beta}{2} \Phi_i(\mathbf{v_h}) + \lambda (|\nabla_h \mathbf{v_h}|_{\text{F},r})_i \\ + \frac{1}{2} \big(T_h^* q_{h,1} + \nabla_h^*\mathbf{q_{h,2}} - \alpha_2 T_h^* g_h\big)_i \big(B_h^{-1}(T_h^* q_{h,1} + \nabla_h^*\mathbf{q_{h,2}} - \alpha_2 T_h^* g_h) \big)_i - \frac{\alpha_2}{2} (g_h)_i^2 + (g_h)_i (q_{h,1})_i,
\end{multline}
where
\begin{equation} \label{eq:grid-ref-indicator:phi}
	\Phi_i(\mathbf{v_h}) := \begin{cases}
		\sum_{k=1}^m (\mathbf{v_h})_{(k-1)N+i}^2 & \text{if}\ S_h=Id,\\
		\sum_{k=1}^m \sum_{k=1}^2 (\nabla_h\mathbf{v_h})_{2(n-1)N + (k-1)N + i}^2 & \text{if}\ S_h=\nabla_h,
	\end{cases}
\end{equation}
for $i=1,\hdots,N$.

\section{Semi-Smooth Newton Method}
\label{sec:algorithms}

In this section we derive a semi-smooth Newton iteration in order to solve the $L^1$-$L^2$-$\TV$ problem. While the dual problem \eqref{eq:dual} is concave, its solution may not be unique due to the non-trivial kernel of $\nabla^*$. To ensure a unique solution, we introduce regularization terms for $p_1$ and $\mathbf{p_2}$ with corresponding parameters $\gamma_1,\gamma_2\geq 0$. By setting $\gamma_1,\gamma_2 > 0$, we guarantee the strong concavity of the dual energy \cite{Hilb2023Jul}. The modified dual problem then becomes
\begin{equation} \label{eq:dual-reg}
	\sup_{\mathbf{p}=(p_1,\mathbf{p_2})\in L^2(\Omega)\times \big(L^2(\Omega)^d\big)^m} D(p_1,\mathbf{p_2}) - \frac{\gamma_1}{2\alpha_1}\|p_1\|_{L^2(\Omega)}^2 - \frac{\gamma_2}{2\lambda}\|\mathbf{p_2}\|_{\big(L^2(\Omega)^d\big)^m}^2.
\end{equation}
As in \cite{Hilb2023Jul}, we use the convention that the terms $\frac{\gamma_1}{2\alpha_1}\|p_1\|_{L^2(\Omega)}^2$ and $\frac{\gamma_2}{2\lambda}\|\mathbf{p_2}\|_{\big(L^2(\Omega)^d\big)^m}^2$ vanish respectively for $\alpha_1=0$ or $\lambda=0$, as these immediately implie $p_1=0$ or $\mathbf{p_2}=0$ respectively. By \cite[Theorem 3.1 and Proposition 3.2]{Hilb2023Jul} the dual problem to \eqref{eq:dual-reg} is given by
\begin{equation} \label{eq:primal-reg}
	\inf_{\mathbf{u}\in H^1(\Omega)^m} \alpha_1\int_\Omega \varphi_{\gamma_1}\big(|T\mathbf{u}-g|\big)\ dx + \frac{\alpha_2}{2} \|T\mathbf{u} - g\|_{L^2(\Omega)}^2 + \frac{\beta}{2} \|S\mathbf{u}\|_{V_S}^2 + \lambda\int_\Omega \varphi_{\gamma_2}\big(|\nabla\mathbf{u}|_{\text{F},r}\big)\ dx,
\end{equation}
where $\varphi_\gamma:\R\to [0,+\infty[$ is the Hubert function defined for $\gamma\geq 0$ as
$$ \varphi_\gamma(x) := \begin{cases}
	\frac{1}{2\gamma} x^2, & \text{if}\ |x|\leq \gamma, \\
	|x|-\frac{\gamma}{2}, & \text{otherwise}.
\end{cases} $$
The optimality conditions for a solution $\mathbf{u}$ of \eqref{eq:primal-reg} and a solution $(p_1,\mathbf{p_2})$ of \eqref{eq:dual-reg} read
\begin{equation} \label{eq:optimality-conditions}
	\left\{\begin{array}{ll}
		0 = T^* p_1 + \nabla^* \mathbf{p_2} - \alpha_2 T^* g + B\mathbf{u}, &  \\
		0 = p_1 \max\{ \gamma_1, |T\mathbf{u}-g| \} - \alpha_1 (T\mathbf{u}-g), & \text{if}\ |p_1|\leq\alpha_1, \\
		0 = \mathbf{p_2} \max\{ \gamma_2, |\nabla \mathbf{u}|_{F,r} \} - \lambda \nabla \mathbf{u}, & \text{if}\ |\mathbf{p_2}|_{F,r}\leq\lambda.
	\end{array}\right.
\end{equation}
We restrict ourselves to $r=2$ until the end of the article so that the $\max$-operator and the $r$-Frobenius norm are generalized differentiable and semi-smooth, justifying the application of a Newton step \cite{Hintermuller2006}. When $|p_1|\leq\alpha_1$ and $|\mathbf{p_2}|_{F,r}\leq\lambda$, we have by using the optimality conditions \eqref{eq:optimality-conditions}
\begin{equation} \label{eq:opti-equ}
	F(\mathbf{u},p_1,\mathbf{p_2}) := T^* p_1 + \nabla^* \mathbf{p_2} - \alpha_2 T^* g + B\mathbf{u} + m_1(\mathbf{u})\,p_1 - \alpha_1 (T\mathbf{u}-g) + m_2(\mathbf{u})\,\mathbf{p_2} - \lambda \nabla \mathbf{u} = 0,
\end{equation}
where $m_1(\mathbf{u}) := \max\{ \gamma_1, |T\mathbf{u}-g| \}$ and $m_2(\mathbf{u}) := \max\{ \gamma_2, |\nabla \mathbf{u}|_{F,r} \}$.
The resulting Newton system $(0, 0, 0) = DF(\mathbf{u}, p_1, \mathbf{p_2})(\mathbf{d_u}, d_{p_1}, \mathbf{d_{p_2}})$ reads as follows:
\begin{equation*} \label{eq:newton-system}
    \left\{\begin{array}{ccccl}
       B\mathbf{d_u} & + T^* d_{p_1} & + \nabla^*\mathbf{d_{p_2}} & = & - \Big( \nabla^*\mathbf{p_2} + T^*p_1 + B\mathbf{u} - \alpha_2 T^* g \Big), \\
       \chi_1 \frac{(T\mathbf{u}-g)\cdot T\mathbf{d_u}}{|T\mathbf{u}-g|}p_1-\alpha_1 T\mathbf{d_u} & + m_1 d_{p_1} & & = & -\Big( m_1p_1-\alpha_1(T\mathbf{u}-g) \Big), \\
       \chi_2\frac{\nabla\mathbf{u}\cdot\nabla\mathbf{d_u}}{|\nabla\mathbf{u}|_{F,r}}\mathbf{p_2} - \lambda\nabla\mathbf{d_u} & & + m_2\mathbf{d_{p_2}} & = & - \Big( m_2\mathbf{p_2}-\lambda\nabla\mathbf{u} \Big),
    \end{array}\right.
\end{equation*}
where $\mathbf{u}\in H^1(\Omega)^m$, $p_1\in L^2(\Omega)$, $\mathbf{p_2}\in \big(L^2(\Omega)^2\big)^d$ represent the variables from the previous Newton step and $(\mathbf{d_u}, d_{p_1}, \mathbf{d_{p_2}}) \in H^1(\Omega)^m\times L^2(\Omega)\times \big(L^2(\Omega)^2\big)^d$ is the solution of the Newton system. Using the previous equations and by using the definitions of $m_1$ and $m_2$ we have
\begin{equation} \label{eq:dp1}
    d_{p_1} = -p_1  + \frac{\alpha_1}{m_1(\mathbf{u})}\big(T(\mathbf{u}+\mathbf{d_u})-g\big) - \chi_1 \frac{(T\mathbf{u}-g)\cdot T\mathbf{d_u}}{m_1(\mathbf{u})^2}p_1,
\end{equation}
\begin{equation} \label{eq:dp2}
    \mathbf{d_{p_2}} = - \mathbf{p_2} + \frac{\lambda}{m_2(\mathbf{u})} \nabla(\mathbf{u}+\mathbf{d_u}) - \chi_2\frac{\nabla\mathbf{u}\cdot\nabla\mathbf{d_u}}{m_2(\mathbf{u})^2}\mathbf{p_2},
\end{equation}
where 
$$ \chi_1(\mathbf{u}) := \begin{cases}
	1 & \text{if}\ |T\mathbf{u}-g|\geq\gamma_1, \\
	0 & \text{otherwise},
\end{cases} \quad\text{and}\quad \chi_2(\mathbf{u}) := \begin{cases}
	1 & \text{if}\ |\nabla \mathbf{u}|_{F,r}\geq\gamma_2, \\
	0 & \text{otherwise}.
\end{cases}$$
Replacing these two equations into the Newton system gives us
\begin{equation}\label{eq:du}
H \mathbf{d_u} = G,
\end{equation}
where
$$ H := T^* \left( \frac{\alpha_1}{m_1(\mathbf{u})}T - \chi_1 p_1 \frac{(T\mathbf{u}-g)\cdot T}{m_1(\mathbf{u})^2} \right) + \nabla^* \left( \frac{\lambda}{m_2(\mathbf{u})} \nabla - \chi_2 \mathbf{p_2} \frac{\nabla \mathbf{u}\cdot\nabla}{m_2(\mathbf{u})^2} \right) + B, $$
and
$$ G := -\nabla^* \frac{\lambda}{m_2(\mathbf{u})}\nabla \mathbf{u} - T^* \frac{\alpha_1}{m_1(\mathbf{u})} (T\mathbf{u}-g) - B\mathbf{u} + \alpha_2 T^* g. $$

Now we discretize \eqref{eq:dp1}, \eqref{eq:dp2} and \eqref{eq:du} using the adaptive finite difference scheme introduced in Section \ref{sec:afdm}. 
Let $m_{h,1}, m_{h,2} \in \R^N$ represent the approximations of $m_1$ and $m_2$ respectively, and let $\chi_{h,1}, \chi_{h,2} \in \R^N$ be the corresponding approximations of $\chi_1$ and $\chi_2$ respectively, as described in Section \ref{sec:discretization}. 
Then the discrete version of \eqref{eq:du} is
\begin{equation} \label{eq:newton-du}
    H_h \mathbf{d_{u_h}} = G_h,
\end{equation}
where
\begin{multline*}
    H_h := B_h + T_h^*D_N(m_{h,1})^{-1}\Big( D_N(\alpha_1 \mathbf{1}_N) - D_N(\chi_{h,1})D_N(m_{h,1})^{-1}D_N(p_{h,1})D_N(T_h \mathbf{u_h}-g_h)\Big) T_h \\ + \nabla_h^*D_{2N}(m_{h,2})^{-1}\Big( D_{2N}(\lambda \mathbf{1}_{2N}) - D_{2N}(\chi_{h,2}) D_{2N}(m_{h,2})^{-1} D_{2N}(\mathbf{p_{h,2}}) \mathcal{N}(\nabla_h \mathbf{u_h}) \Big)\nabla_h,
\end{multline*}
$$ G_h := -\nabla_h^* D_{2N}(\lambda \mathbf{1}_{2N}) D_{2N}(m_{h,2})^{-1} \nabla_h\mathbf{u_h} - T_h^* D_N(\alpha_1 \mathbf{1}_N) D_N(m_{h,1})^{-1} (T_h \mathbf{u_h} - g_h) - B_h\mathbf{u_h} + D_N(\alpha_2 \mathbf{1}_N) T_h^* g_h, $$
$D_M(v_h)$ denotes the diagonal matrix with the values of $v_h\in\R^M$, $M\geq 1$ on the diagonal and $\mathbf{1}_M := (1,\hdots,1)^T\in\R^M$. Here $\mathcal{N}(\nabla_h \mathbf{u_h})\nabla_h\mathbf{d_{u_h}}$ is the discretization of the term $\nabla\mathbf{u}\cdot\nabla\mathbf{d_u}$ in $\eqref{eq:du}$. Moreover, we have
\begin{multline} \label{eq:newton-dp1}
    d_{p_{h,1}} = - p_{h,1} + D_N(m_{h,1})^{-1} \Big( D_N(\alpha_{1}\mathbf{1}_N)\big(T_h (\mathbf{u_h} + \mathbf{d_{u_h}})-g_h\big) \\ - D_N(\chi_{h,1})D_N(m_{h,1})^{-1}D_N(p_{h,1})D_N(T_h \mathbf{u_h}-g_h) T_h \mathbf{d_{u_h}} \Big),
\end{multline}
\begin{multline}  \label{eq:newton-dp2}
    \mathbf{d_{p_{h,2}}} = - \mathbf{p_{h,2}} + D_{2N}(m_{h,2})^{-1} \Big( D_{2N}(\lambda \mathbf{1}_{2N})\nabla_h (\mathbf{u_h}+\mathbf{d_{u_h}}) \\ - D_{2N}(\chi_{h,2}) D_{2N}(m_{h,2})^{-1} D_{2N}(\mathbf{p_{h,2}}) \mathcal{N}(\nabla_h \mathbf{u_h})\nabla_h \mathbf{d_{u_h}} \Big).
\end{multline}
We note that $H_h$ is even in the uniform case not symmetric because
$ \mathcal{N}(\nabla_h \mathbf{u_h}) $
is not. Indeed, if $m=1$ (image reconstruction) for $v=(v_x, v_y)\in\R^{2N}$, we have
$$ \mathcal{N}(v) := \begin{pmatrix} D_N(v_x) & D_N(v_y) \\ D_N(v_x) & D_N(v_y) \end{pmatrix}, $$
and if $m=2$ (optical flow estimation), for $v=(v^1_x, v^1_y, v^2_x, v^2_y)\in\R^{4N}$, it follows
$$ \mathcal{N}(v) := \begin{pmatrix} D_N(v^1_x) & D_N(v^1_y) & D_N(v^2_x) & D_N(v^2_y) \\ D_N(v^1_x) & D_N(v^1_y) & D_N(v^2_x) & D_N(v^2_y) \\ D_N(v^1_x) & D_N(v^1_y) & D_N(v^2_x) & D_N(v^2_y) \\ D_N(v^1_x) & D_N(v^1_y) & D_N(v^2_x) & D_N(v^2_y) \end{pmatrix}. $$
In order to satisfy $|p_{h,1}|\leq\alpha_{1}$ and $|\mathbf{p_{h,2}}|_{F,s}\leq\lambda$, we propose to project $(p_{h,1})_i$ (resp. $(\mathbf{p_{h,2}})_i$) on the disk of radius $\alpha_{1}$ (resp. $\lambda$) for the corresponding norms. We give in Algorithm \ref{algo:semismooth-newton} the semi-smooth Newton algorithm.

\IncMargin{1.5em}
\begin{algorithm}[h]
    \SetAlgoLined
    \KwData{parameters $\alpha_{1}, \alpha_{2}, \lambda, \beta, \gamma_{1}, \gamma_{2} \geq 0$, data $g_h\in \R^N$, initial values $(\mathbf{u}_\mathbf{h}^0, p_{h,1}^0, \mathbf{p}_{\mathbf{h,2}}^0) \in \R^{mN}\times\R^N\times\R^{2mN}$.}
    \KwResult{sequence $(\mathbf{u}_\mathbf{h}^n, p_{h,1}^n, \mathbf{p}_\mathbf{h,2}^n)_{n\in\N}$ approximating the solution to \eqref{eq:optimality-conditions}.}
    $n = 0$\;
    \While{the stopping criterion is not fulfilled}{
        Solve \eqref{eq:newton-du} and save the result in $\mathbf{d_{u_h}}$\;
        Compute $d_{p_{h,1}}$ and $\mathbf{d_{p_{h,2}}}$ using \eqref{eq:newton-dp1} and \eqref{eq:newton-dp2} respectively\;
        $\mathbf{u}_\mathbf{h}^{n+1} = \mathbf{u}_\mathbf{h}^n + \mathbf{d_{u_h}}$\;
        $\mathbf{p}_\mathbf{h}^{n+1} = \mathbf{p}_\mathbf{h}^n + (d_{p_{h,1}},  \mathbf{d_{p_{h,2}}})$\;
        Project $\mathbf{p}_\mathbf{h}^{n+1}$ to satisfy $|p_{h,1}^{n+1}|\leq\alpha_{1}$ and $|\mathbf{p}_\mathbf{h,2}^{n+1}|_{F,s}\leq\lambda$\;
        $n=n+1$\;
    }
    
    \caption{Discrete semi-smooth Newton algorithm.}
    \label{algo:semismooth-newton}
\end{algorithm}
\DecMargin{1.5em}
\color{black}
We propose to use as stopping criterion in Algorithm \ref{algo:semismooth-newton} the residual, i.e.\ $ \| F_h(\mathbf{u}_\mathbf{h}^n, p_{h,1}^n, \mathbf{p}_\mathbf{h,2}^n) \|_{L^2(\Omega_h)} \leq \varepsilon$, for some given tolerance $\varepsilon > 0$, where $F_h$ is the discretization of $F$ defined in \eqref{eq:opti-equ}. Additionally, we stop the algorithm when the maximal number of iteration $\text{max\_it}\in\N$ is reached. 
\begin{note}
	In order to guaranty the convergence of the algorithm to a minimizer of \eqref{eq:opti-equ}, we need $H_h$ to be positive definite. Unfortunately, this is not the case in general. In order to overcome this, one could replace $H_h$ by a positive definite matrix $H_h^+$ by projecting $H_h$ into the space of positive definite matrices $S_+^{2mN}$, i.e. by solving
	$$ H_h^+ := \argmin_{A \in S_+^{2mN}} |A - H_h|_{F,2}. $$
	The solution of this projection has the following form \cite[Section 8.1.1]{Boyd2004}
	\begin{equation} \label{eq:matrix-projection}
		H_h^+ = \sum_{i=1}^{2mN} \max(0,\mu_i)\xi_i\xi_i^T,
	\end{equation}
	where $(\mu_i)_i$ are the eigenvalues of the corresponding eigenvectors $(\xi_i)_i$ of $H$. In practice, due to the large size of the matrix $H$, we use sparse matrices. To compute the projection in equation \eqref{eq:matrix-projection}, we need to determine all $2mN$ eigenvalues. The only available implementation for handling sparse matrices is the \textit{implicitly restarted Arnoldi method} \cite{Lehoucq1998}, which allows us to numerically compute at most $2mN-1$ eigenvalues/eigenvectors. Consequently, we cannot directly use equation \eqref{eq:matrix-projection}, as it requires all eigenvalues and eigenvectors. This is the reason why we use the original matrix $H_h$ instead of $H_h^+$.
\end{note}

\section{Numerical Experiments}
\label{sec:experiments}

In this section, we solve the regularized primal problem \eqref{eq:primal-reg} and its associated pre-dual problem \eqref{eq:dual-reg} using the semi-smooth Newton algorithm in Algorithm \ref{algo:semismooth-newton}. The energies employed here differ from those used to derive the grid refinement indicator \eqref{eq:grid-ref-indicator}. 
Nevertheless, a similar result as Proposition \ref{prop:pd-gap-error} can be shown changing the indicator  \eqref{eq:grid-ref-indicator} in the following way: for all $\mathbf{v_h}\in\R^{mN}$ and $(q_{h,1},\mathbf{q_{h,2}})\in\R^{N}\times\R^{2mN}$ such that $|q_{h,1}|\leq \alpha_1$ and $|\mathbf{q_{h,2}}|_{F,s}\leq \lambda$
\begin{multline} \label{eq:grid-ref-indicator-reg}
	h_i^{-2}(\eta_h)_i = \alpha_1 \varphi_{\gamma_{h,1}}\big((T_h\mathbf{v_h}-g_h)_i\big) + \frac{\alpha_2}{2} (T_h\mathbf{v_h} - g_h)_i^2 + \frac{\beta}{2} \Phi_i(\mathbf{v_h}) + \lambda (|\nabla_h \mathbf{v_h}|_{\text{F},r})_i \\ + \frac{1}{2} \sum_{k=1}^m \big(T_h^* q_{h,1} + \nabla_h^*\mathbf{q_{h,2}} - \alpha_2 T_h^* g_h\big)_{(k-1)N+i} \big(B_h^{-1}(T_h^* q_{h,1} + \nabla_h^*\mathbf{q_{h,2}} - \alpha_2 T_h^* g_h) \big)_{(k-1)N+i} \\ - \frac{\alpha_2}{2} (g_h)_i^2 + (g_h)_i (q_{h,1})_i + \frac{\gamma_1}{2\alpha_1}(q_{h,1})_i^2 + \frac{\gamma_2}{2\lambda}\sum_{n=1}^m\sum_{k=1}^2 (\mathbf{q_{h,2}})_{2(n-1)N + (k-1)N + i}^2,
\end{multline}
where $\Phi_i$ is defined in \eqref{eq:grid-ref-indicator:phi}, for $i=1,\hdots,N$. Observe that, because we discard many terms in Proposition \ref{prop:pd-gap-error}, it is possible that certain components of the indicator could turn out to be negative. More precisely, the fifth and seventh term in \eqref{eq:grid-ref-indicator-reg} may take on negative values, whereas the sixth term is always negative. To address this, we ensure the indicator is always positive by replacing $\eta_h$ in \eqref{eq:dorfler} by
\begin{equation} \label{eq:grid-ref-indicator-reg-shifted}
	\widetilde{\eta_h} := \eta_h - \min_n (\eta_h)_n.
\end{equation}
The initial values for the Newton algorithm are chosen to be $\mathbf{u}_{\mathbf{h}}^0 = T_h^* g_h$, $p_{h,1}^0 = T_h \mathbf{u}_{\mathbf{h}}^0$ and $\mathbf{p}_{\mathbf{h,2}}^0 = \nabla_h \mathbf{u}_{\mathbf{h}}^0$. Experimentally, this choice leads to a faster convergence, in terms of number of iterations, of the algorithm than setting $\mathbf{u}_{\mathbf{h}}^0$, $p_{h,1}^0$ and $\mathbf{p}_{\mathbf{h,2}}^0$ identically to zero. In practice, we set $B_h := \alpha_{2} T_h^*T_h + \beta S_h^* S_h$ and $B_h^{-1}$ is taken as the inverse of $B_h$. These choices are different from Section \ref{sec:discretization} as it is not clear that $\alpha_{2} T_h^*T_h + \beta S_h^* S_h$ is the discretization of the continuous $B$ from Section \ref{sec:formulation} and that $(\alpha_{2} T_h^*T_h + \beta S_h^* S_h)^{-1}$ is the discretization of the inverse of the continuous $B$. In fact, the invertibility of $\alpha_{2} T_h^*T_h + \beta S_h^* S_h$ remains uncertain as in the case $S_h = \nabla_h$ we do not have that $(\nabla^*)_h = (\nabla_h)^*$. Furthermore, the semi-smooth Newton algorithm occasionally reaches the maximum allowed number of iterations, $\text{max\_it}=100$, which can result in the residual $ \| F_h(\mathbf{u}_\mathbf{h}^n, p_{h,1}^n, \mathbf{p}_\mathbf{h,2}^n) \|_{L^2(\Omega_h)} $ remaining large. The implementation of the algorithms is done in the programming language Python and can be found at \cite{jacumin_thomasjacuminadaptive-l1-l2-tv_2024}.

\subsection{Application to Image Reconstruction}

A grayscale image is represented by a function $f:\Omega\subseteq\R^2\to[0,1]$.
Here $m=1$ and we choose
$ T = Id$ and $g = f,$
which leads to
$ T_h = Id \ \text{and}\ g_h = f_h,$
where $f_h$ is the discretization of $f$. For the grid refinement strategy, we utilize the greedy \textit{Dörfler} marking strategy \cite{Nochetto2009} with $0 \leq \theta_\text{mark} \leq 1$, i.e.\ denoting by $N\in\N$ the number of elements of a given grid $\Omega_h$, for given error indicators in descending order $(\widetilde{\eta_h})_n$, the algorithm refines the first $n_\text{mark}\in\N$ elements \color{black} that satisfy
\begin{equation} \label{eq:dorfler}
	\sum_{n=1}^{n_\text{mark}} (\widetilde{\eta_h})_n \geq \theta_\text{mark} \sum_{n=1}^N (\widetilde{\eta_h})_n.
\end{equation}
We give in Algorithm \ref{algo:noiseless} the coarse-to-fine algorithm to solve the TV problem on an adaptive mesh.

\IncMargin{1.5em}
\begin{algorithm}[h]
	\SetAlgoLined
	\KwData{input image $f_h$ of size $w\times h$, parameters ($\alpha_1$, $\alpha_2$, $\lambda$, $\beta$, $\gamma_1$, $\gamma_2$), maximal number of refinement $N_\text{ref\_max}\in\N$.}
	\KwResult{output image $\mathbf{u_h}$.}
	
	$\Omega_h^1$ uniform coarse grid of size $\lfloor\frac{w}{2^{N_\text{ref\_max}}}\rfloor \times \lfloor\frac{h}{2^{N_\text{ref\_max}}}\rfloor$\;
	$k = 1$\;
	\While{$\text{mesh level} < N_\text{max\_ref}$}{
		compute $\mathbf{u_h}$ using the semi-smooth Newton algorithm (Algorithm \ref{algo:semismooth-newton})\;
		compute the local primal-dual gap error $\widetilde{\eta_h}$ using \eqref{eq:grid-ref-indicator-reg-shifted}\;
		adapt (as described in Section \ref{sec:discretization}) the grid using \eqref{eq:dorfler} and save the resulting grid in $\Omega_h^{k+1}$\;
		$L^2$-projection of the data on the new grid $\Omega_h^{k+1}$\;
		$k = k+1$\;
	}
	
	\caption{Adaptive $L^1$-$L^2$-$\TV$ algorithm for images ($m=1$).}
	\label{algo:noiseless}
\end{algorithm}
\DecMargin{1.5em}

\subsubsection{Noiseless Case}
In order to demonstrate that using the indicator $\tilde \eta_h$ derived from the primal-dual gap error \eqref{eq:grid-ref-indicator-reg-shifted} as a mesh refinement criterion may lead to a faster convergence with respect to the number of degrees of freedom (dofs) than uniform refinement we consider a simple 2D example.
Let the image $f$ be the characteristic function of the disk centered at the origin and of radius $R>0$, i.e.\ $ f = \mathbf{1}_{B(0,R)}$. We recall the explicit solution $u_{\alpha_1,\alpha_2}$ of the problem in \eqref{eq:primal} when $\Omega=\R^2$, $\beta = \gamma_1 = \gamma_2 = 0$ and $\alpha_1 > 0$, $\alpha_2 > 0$ found in \cite{Hintermuller2013}:
	$$ u_{\alpha_1,\alpha_2} = \begin{cases}
		0, & \text{if}\ 0\leq R < \frac{2}{\alpha_2 + \alpha_1}, \\
		\left(\frac{\alpha_2+\alpha_1}{\alpha_2} - \frac{2}{\alpha_2 R}\right)\mathbf{1}_{B(0,R)}, & \text{if}\ \frac{2}{\alpha_2 + \alpha_1}\leq R \leq \frac{2}{\alpha_1}, \\
		\mathbf{1}_{B(0,R)}, & \text{if}\ R > \frac{2}{\alpha_1}.
	\end{cases} $$
In particular, if we choose $\alpha_1 = \alpha_2 = 1$ and $R=\frac{3}{2}$, we have $ u_\text{exact} := \frac{2}{3}\mathbf{1}_{B(0,R)} $.
\begin{figure}[htb]
	\centering
	\includegraphics[width=6cm]{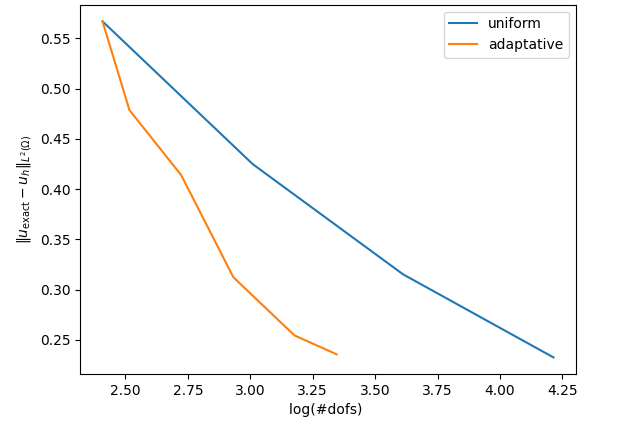}	
	\caption{Convergence for the disk example.}
	\label{fig:convergence-plot}
\end{figure}
Figure \ref{fig:convergence-plot} shows the error between the exact solution $u_{\text{exact}}$ and the computed discrete solution $\mathbf{u_h}$, i.e.\ $\|u_\text{exact} - I_h \mathbf{u_h}\|_{L^2(\Omega)}$, with respect to the number of degrees of freedom when using Algorithm \ref{algo:noiseless} with a uniform mesh and an adaptive mesh on the disk. As we use the semi-smooth Newton algorithm from Section \ref{sec:algorithms}, we must set $\gamma_{1},\gamma_{2}$ non-zero and choose $\gamma_1=\gamma_2=2\times 10^{-4}$, $\varepsilon=1\times 10^{-3}$ and $\theta_\text{mark}=0.2$. As expected, for a given number of dofs, the adaptive algorithm gives a more accurate solution, i.e.\ the error is smaller, than with a uniform mesh.

In Figure \ref{fig:meshes-convergence}, we show the sequence of adapted meshes corresponding to the disk example. We see that the indicator $\tilde \eta_h$ is well suited for mesh adaptivity since it marks elements for refinement which are located near the discontinuities (edges) of the original image.

\def\convwidth{2.5cm}
\begin{figure}[h]
	\centering
	\subfloat[dofs: 256]{
		\includegraphics[width=\convwidth]{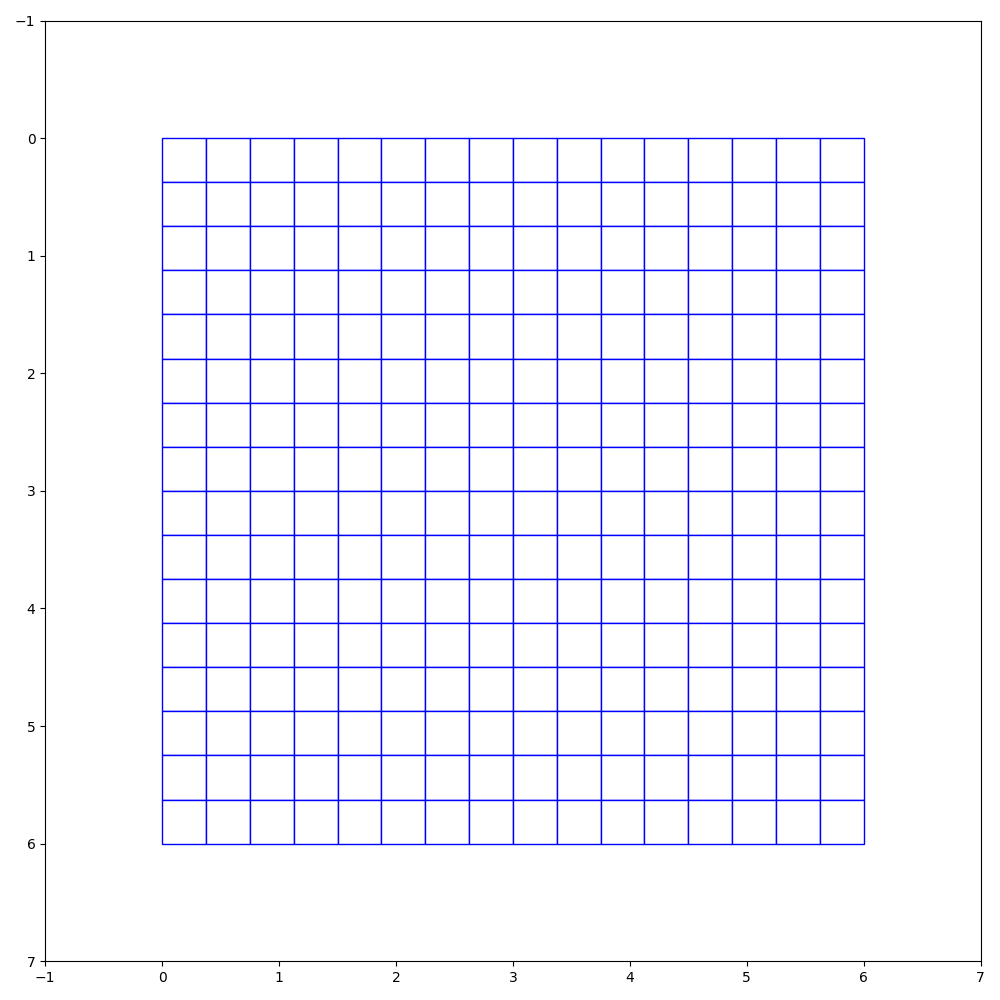}
	}
	\subfloat[dofs: 328]{
		\includegraphics[width=\convwidth]{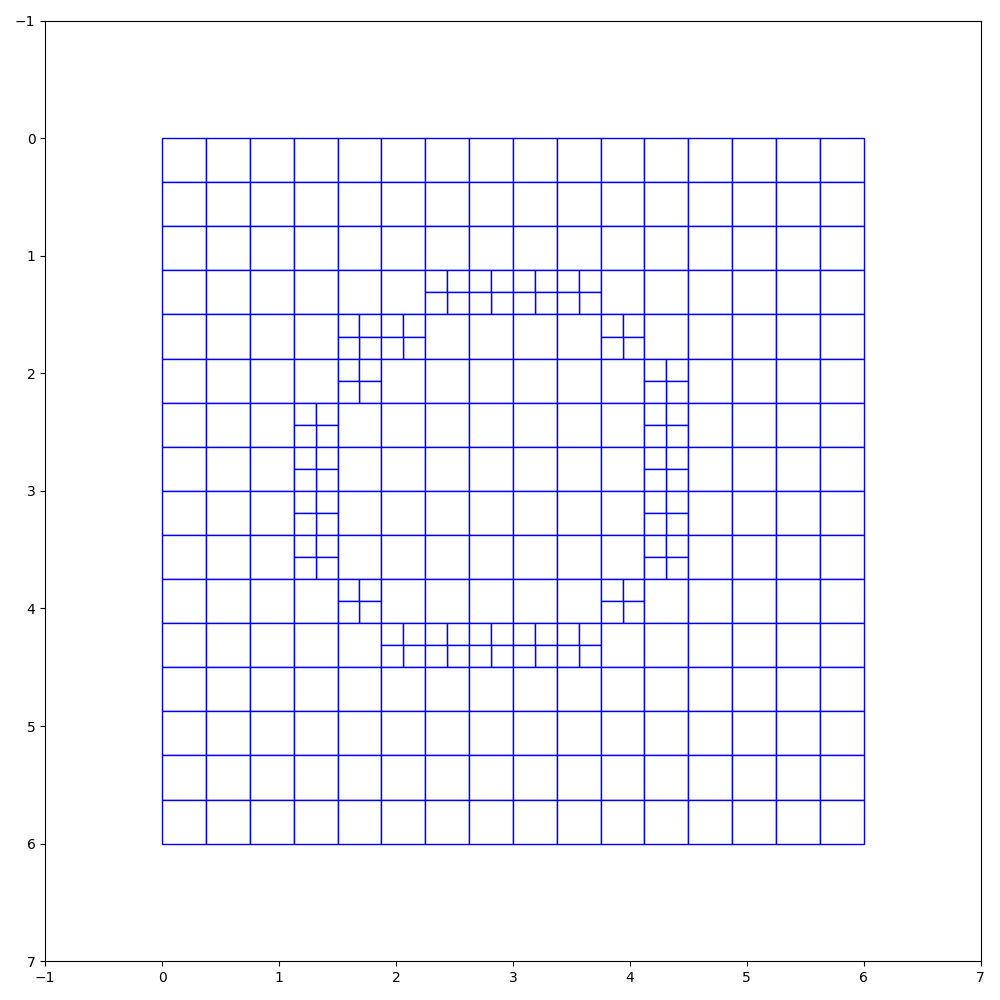}
	}
	\subfloat[dofs: 529]{
		\includegraphics[width=\convwidth]{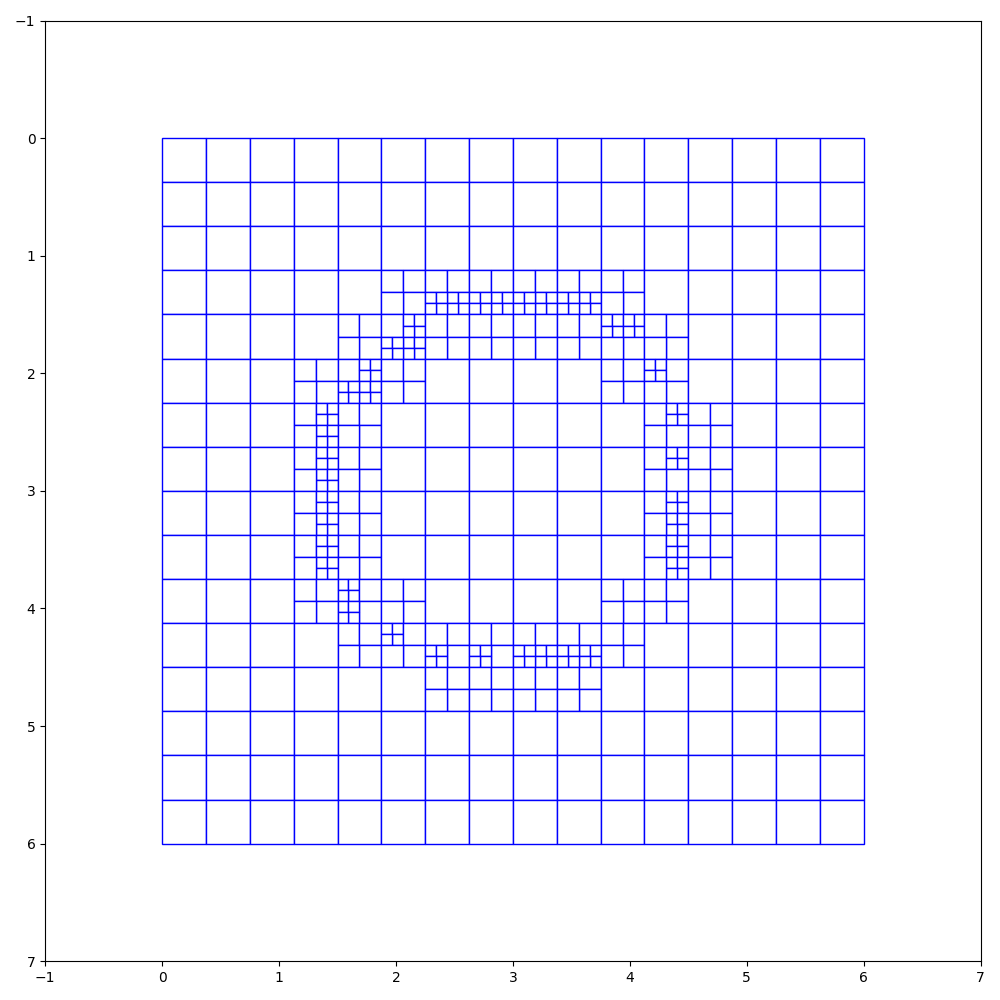}
	}
	\subfloat[dofs: 856]{
		\includegraphics[width=\convwidth]{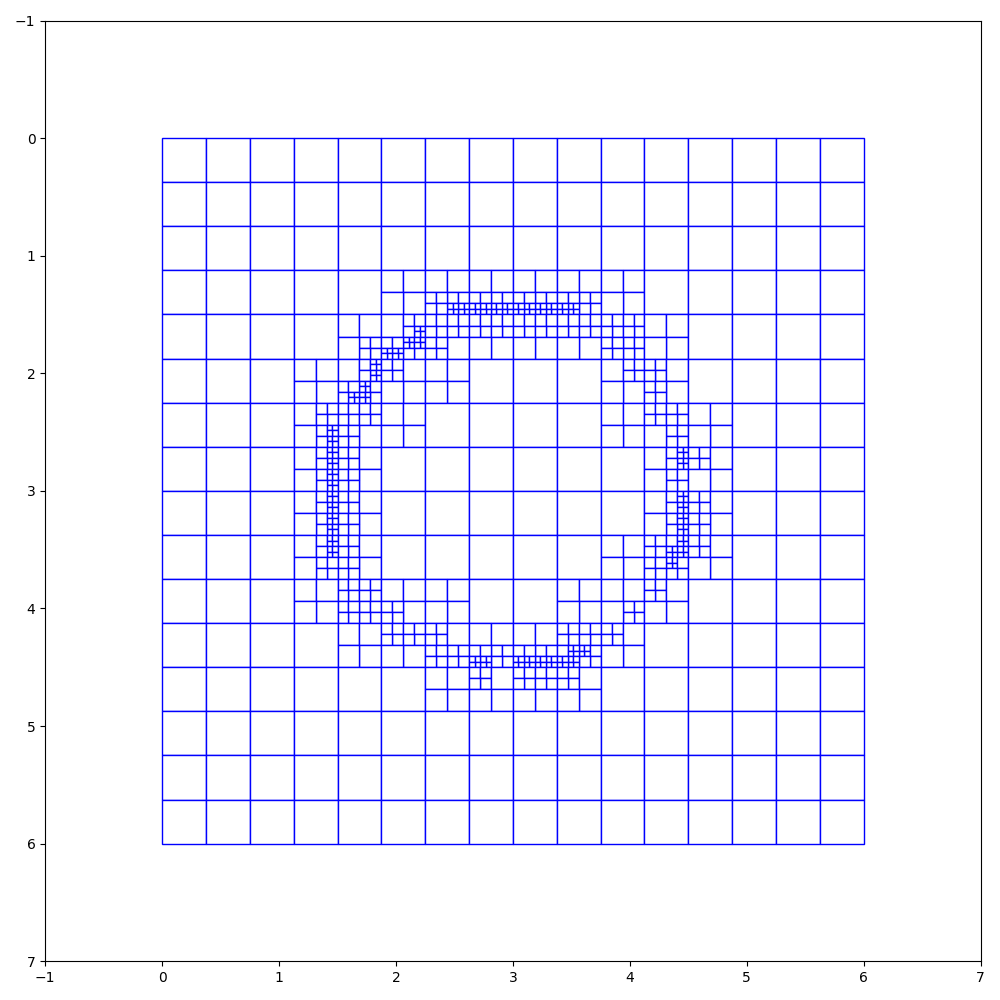}
	}
	\subfloat[dofs: 1,504]{
		\includegraphics[width=\convwidth]{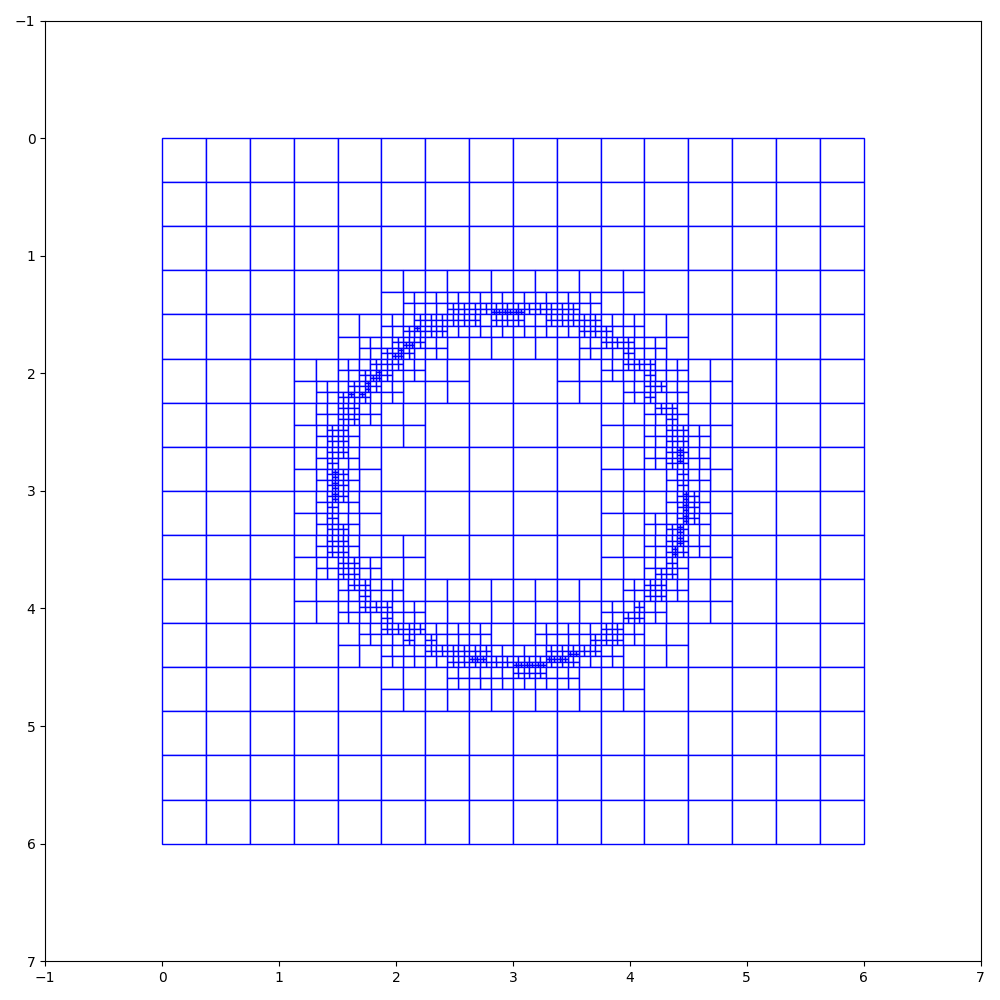}
	}
	\subfloat[dofs: 2,218]{
		\includegraphics[width=\convwidth]{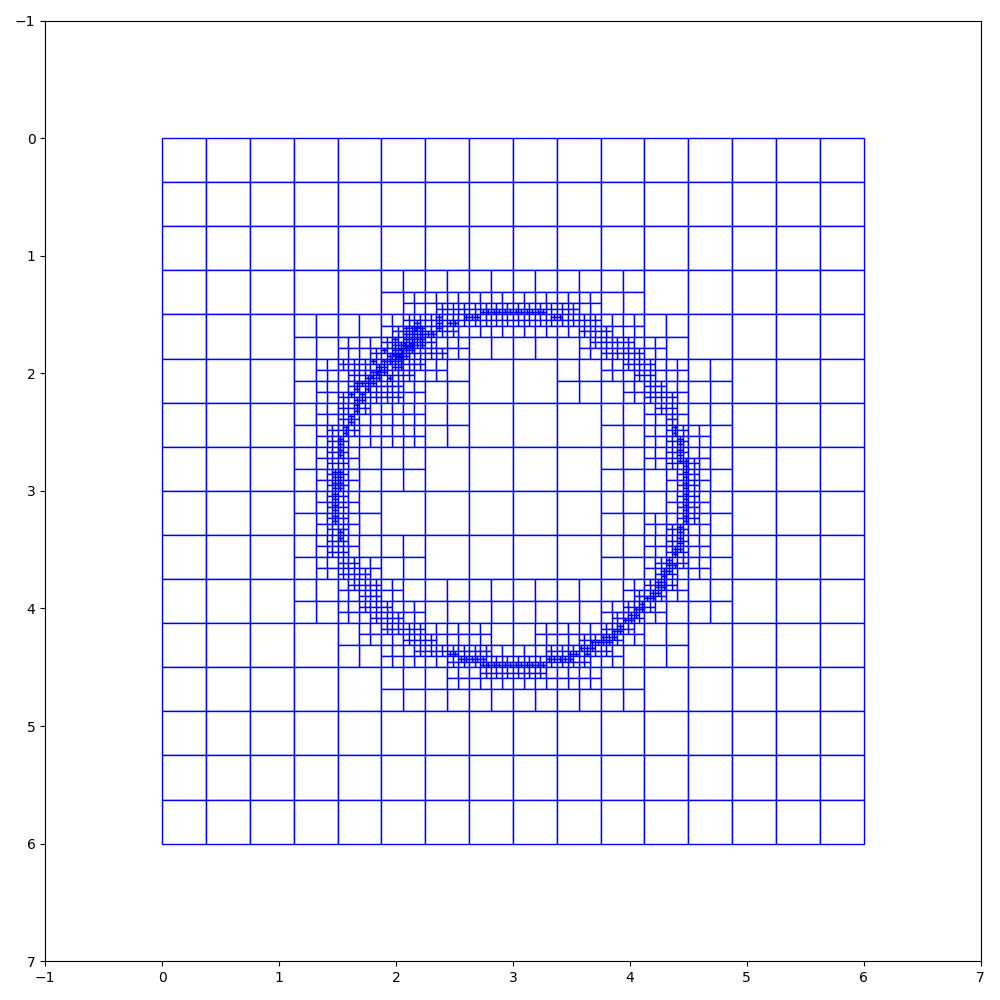}
	}
	\caption{Adapted meshes sequence for the disk example using Algorithm \ref{algo:noiseless}.}
	\label{fig:meshes-convergence}
\end{figure}

\subsubsection{Gaussian Denoising}

For removing additive Gaussian noise with zero mean and standard deviation $\sigma$ from an image ($m=1$), we set $T=Id$ and we use the same algorithm as in the noiseless case, namely Algorithm \ref{algo:noiseless}, with the slight change that a second condition is added to the refinement criterion named the \textit{bulk criterion}. This approach is motivated by \cite{Hintermuller2014}, where an element for refinement is marked when the local residual is bigger than a quantity depending on the noise level and on the size of the element by using a statistical argument in the same spirit as \cite{Dong2011}. Here however, a different \textit{bulk criterion} which do not depend on the size of the element is used: an element $E_i$ with center $\mathbf{x}_i\in\Omega_h$ can be marked for refinement if
\begin{equation} \label{eq:bulk-criterion}
	\frac{1}{2}h_i^2(\mathbf{u_h}-g_h)_i^2 \geq \frac{\sigma^2}{2}.
\end{equation}
Loosely speaking, we do an averaging of the residual on an element $E_i$, and if this quantity is bigger than the noise level, it means that the there is an edge in the element, thus the algorithm refines $E_i$. This condition is similar to the one from \cite{Langer2017Feb}. That is here the marking strategy for refinement runs in two steps: first all elements fulfilling the bulk criterion \eqref{eq:bulk-criterion} are detected and then the Dörfler marking strategy is only applied on these selected elements.

We recall the definition of the \textit{peak signal-to-noise ratio} (PSNR) \cite{Bovik2000}
$$ \text{PSNR} = 10\log_{10} \left( \frac{1}{\frac{1}{N} \sum_{i=0}^N \big((\mathbf{u_h})_i-(g_h)_i\big)^2 } \right). $$
The PSNR is useful to evaluate the proximity of the reconstruction to the original image. In the experiments, we will also use the \textit{structural similarity} (SSIM). The idea of SSIM is to measure the structural similarity between two images, rather than a pixel-to-pixel difference as for example PSNR does. The measure between two windows $x$ and $y$ of common size $\omega\times\omega$ is \cite{wang_multiscale_2003}:
$$ \text{SSIM}(x,y) = \frac{(2\mu_x\mu_y + c_1)(2\sigma_{xy}+c_2)}{(\mu_x^2+\mu_y^2 + c_1)(\sigma_x^2 + \sigma_y^2+c_2)}, $$
where $\mu_x,\mu_y$ is the pixel sample mean of $x,y$ respectively, $\sigma_x^2,\sigma_y^2$ is the variance of $x,y$ respectively, $\sigma_{xy}$ is the covariance of $x$ and $y$, $c_1=0.01^2$ and $c_2=0.03^2$. The constants $c_1$ and $c_2$ are chosen to stabilize the division with weak denominator. Note that when comparing PSNR and SSIM, larger values indicate a better reconstruction than smaller values.

In these experiments, we chose as model parameters $\alpha_1 = 0$, $\alpha_2 = 10$, $\lambda=1$, $\beta = 0$ and $\gamma_1=\gamma_2 = 2\times 10^{-4}$, and we chose $N_\text{max ref} = 6$, $\theta_\text{mark} = 0.6$ and $\varepsilon = 1\times 10^{-3}$.

Table \ref{tab:denoising} shows the PSNR and the mean SSIM (MSSIM) for the denoising with both uniform and adaptive mesh for several noise level $\sigma$. We do the experiment using the \textit{Middlebury} dataset. For image processing tasks like denoising, we aim to align the finely refined grid with the image pixels to preserve discontinuities. Therefore, we crop and resize the input image so that its width and height are powers of two (more precisely to $256\times 256$ pixels). Otherwise, the grid nodes may not align properly, leading to a blurred reconstruction. 
In every example, using a uniform mesh generates results with larger PSNR and MSSIM
and less computation time
than using the coarse-to-fine adaptive refinement strategy.
We observe that in the coarse-to-fine approach, sometimes the PSNR increases as the noise level increases. This is because noise results in a noisy refinement indicator, and since the bulk criterion cannot differentiate edges from noise perfectly, the algorithm tends to refine the grid more uniformly, leading to more accurate reconstructions. Moreover, most of the computation time is due to the projection step of Algorithm \ref{algo:noiseless}, which indicate that projections might not efficiently implemented.

In Figure \ref{fig:denoising}, we give from left to right, the original image from the \textit{Middlebury} dataset, the noisy image used as input for the algorithm ($\sigma=0.1$), the denoised image with a uniform mesh and the denoised image with an adaptive mesh. A sequence of meshes for the adaptive algorithm is given in Figure \ref{fig:meshes-denoising}.

\begin{table}[h]
	\centering
	\begin{tabular}{ll||ccc|ccc}
		\multirow{2}{*}{Benchmark} & \multirow{2}{*}{$\sigma$} & \multicolumn{3}{c|}{Uniform} & \multicolumn{3}{c}{Adaptive} \\
	                               &                           & PSNR  & MSSIM & Time (s)     & PSNR  & MSSIM & Time (s) \\
		\hline\hline
		Dimetrodon                 &          $0.1$            & 29.21 & 0.80  &    5.39    & 27.60 & 0.73  & 23.96 \\
		                           &          $0.05$           & 29.59 & 0.80  &    6.07    & 27.69 & 0.75  & 24.05 \\
		                           &          $0.01$           & 29.68 & 0.80  &    6.05    & 27.07 & 0.74  & 23.92 \\
		\hline
		Grove2                     &          $0.1$            & 25.28 & 0.71  &    7.02    & 23.54 & 0.63  & 23.01 \\
	                               &          $0.05$           & 25.38 & 0.71  &    5.90    & 23.34 & 0.63  & 26.18 \\
	                               &          $0.01$           & 25.39 & 0.71  &    6.36    & 23.26 & 0.62  & 22.98 \\
		\hline
		RubberWhale                &          $0.1$            & 30.53 & 0.79  &    6.35    & 28.90 & 0.75  & 22.07 \\
		                           &          $0.05$           & 31.26 & 0.80  &    5.75    & 29.71 & 0.78  & 65.20 \\
	                               &          $0.01$           & 31.46 & 0.81  &    6.01    & 29.45 & 0.78  & 22.27 \\
		\hline
		Urban2                     &          $0.1$            & 29.75 & 0.76  &    6.20    & 28.32 & 0.70  & 23.76 \\
		                           &          $0.05$           & 30.00 & 0.76  &    5.60    & 28.00 & 0.72  & 23.70 \\
		                           &          $0.01$           & 30.06 & 0.76  &    5.61    & 27.88 & 0.71  & 23.74 \\
		\hline
	\end{tabular}\vspace{0.4cm}
	
	\caption{Quantitative results: peak signal-to-noise ratio (PSNR), structural similarity (SSIM), and computational times in seconds.}
	\label{tab:denoising}
\end{table}
\def\mywidth{1.8cm}
\begin{figure}[h]
	\centering
	\subfloat{
		\includegraphics[width=\mywidth]{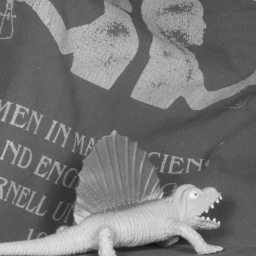}
	}
	\subfloat{
		\includegraphics[width=\mywidth]{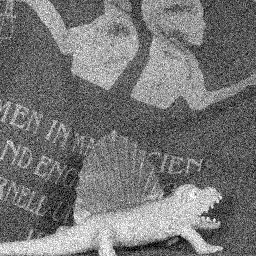}
	}
	\subfloat{
		\includegraphics[width=\mywidth]{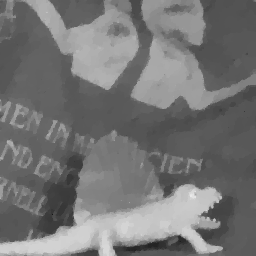}
	}
	\subfloat{
		\includegraphics[width=\mywidth]{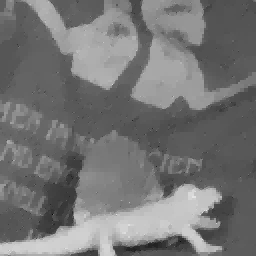}
	}
	\hfill
	\subfloat{
		\includegraphics[width=\mywidth]{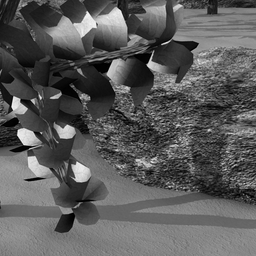}
	}
	\subfloat{
		\includegraphics[width=\mywidth]{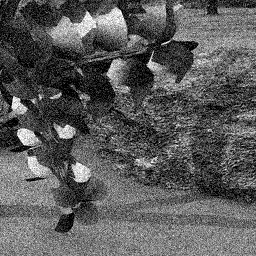}
	}
	\subfloat{
		\includegraphics[width=\mywidth]{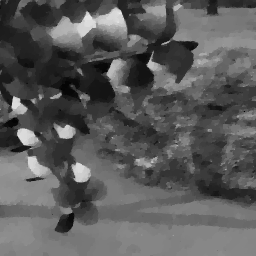}
	}
	\subfloat{
		\includegraphics[width=\mywidth]{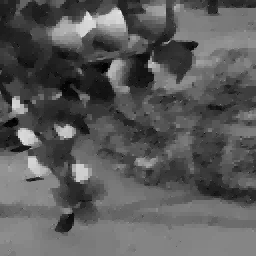}
	}
	%
	\quad
	\subfloat{
		\includegraphics[width=\mywidth]{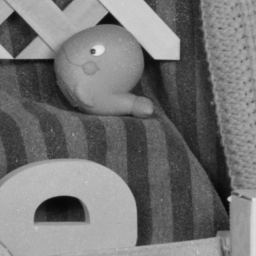}
	}
	\subfloat{
		\includegraphics[width=\mywidth]{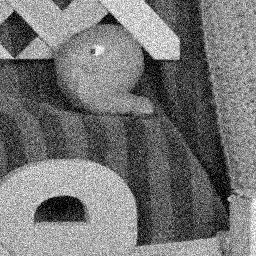}
	}
	\subfloat{
		\includegraphics[width=\mywidth]{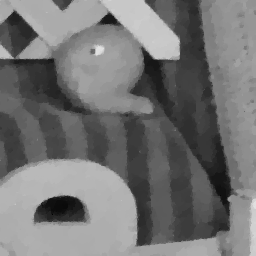}
	}
	\subfloat{
		\includegraphics[width=\mywidth]{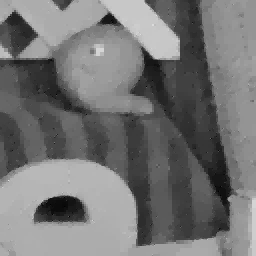}
	}
	\hfill
	\subfloat{
		\includegraphics[width=\mywidth]{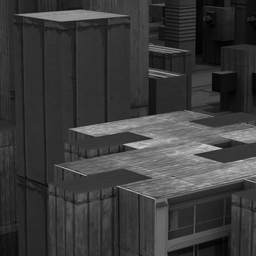}
	}
	\subfloat{
		\includegraphics[width=\mywidth]{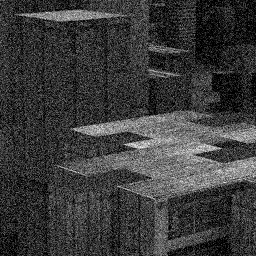}
	}
	\subfloat{
		\includegraphics[width=\mywidth]{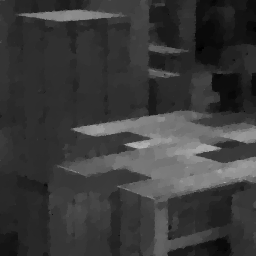}
	}
	\subfloat{
		\includegraphics[width=\mywidth]{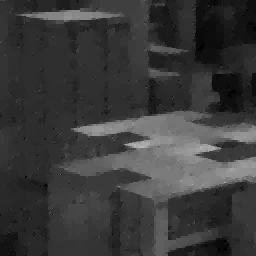}
	}
	
	\caption{Middlebury Denoising Benchmark with $\sigma=0.1$; the original image ($1^\text{st}$ and $5^\text{th}$ column), the noisy image $f$ with $\sigma=0.1$ ($2^\text{nd}$ and $6^\text{th}$ column), the denoised image using Algorithm \ref{algo:semismooth-newton} on a uniform mesh ($3^\text{rd}$ and $7^\text{th}$ column), and the denoised image using Algorithm \ref{algo:noiseless} and the bulk criterion \eqref{eq:bulk-criterion} ($4^\text{th}$ and $8^\text{th}$ column). Benchmarks from top-left to bottom-right: \textit{Dimetrodon}, \textit{Grove2}, \textit{RubberWhale}, \textit{Urban2}.}
	\label{fig:denoising}
\end{figure}

\def\denoisingwidth{3.cm}
\begin{figure}[h]
	\centering
	\subfloat[dofs: 256]{
		\includegraphics[width=\denoisingwidth]{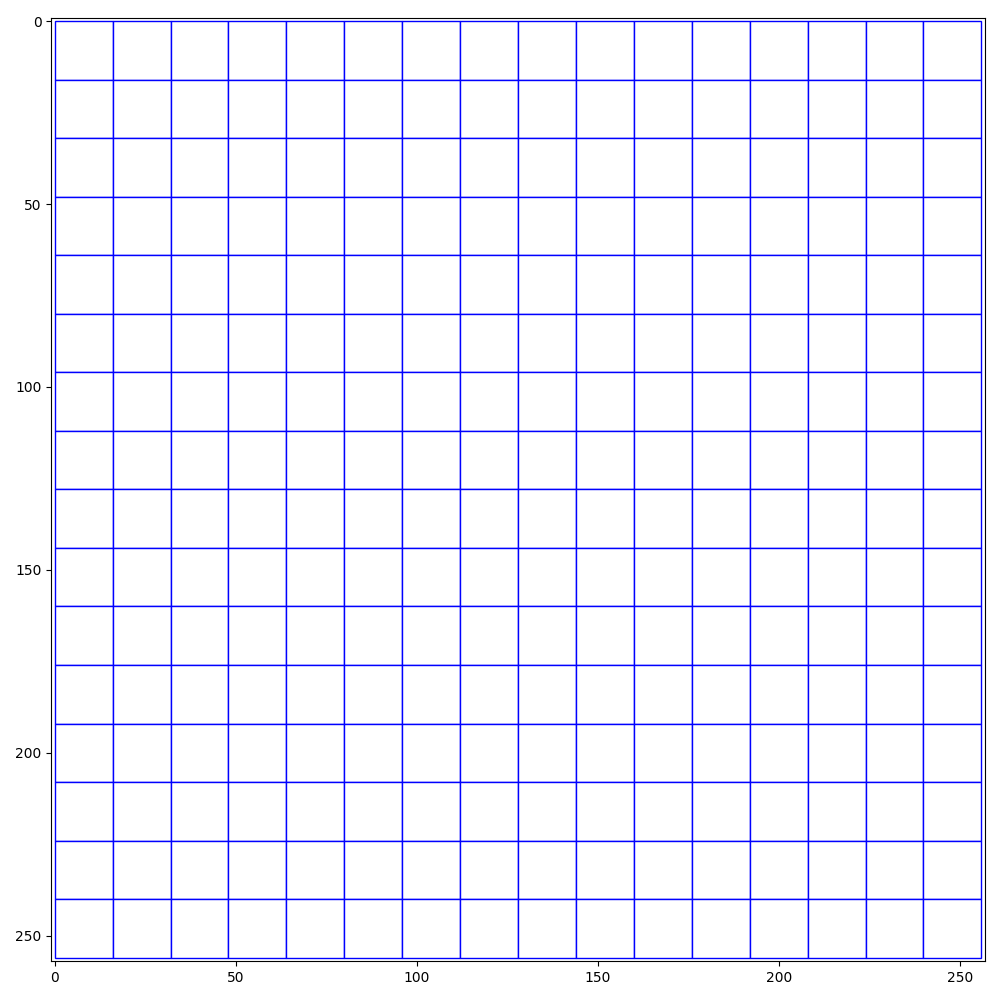}
	}
	\subfloat[dofs: 628]{
		\includegraphics[width=\denoisingwidth]{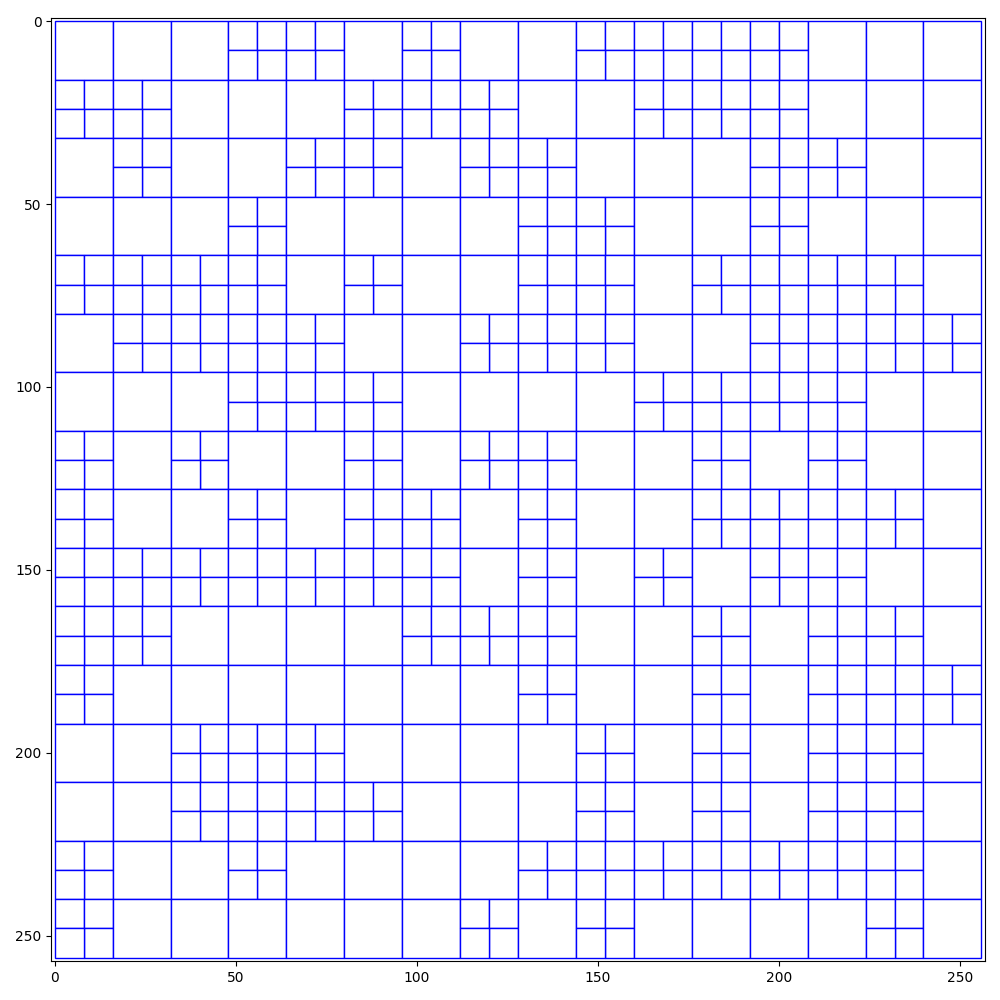}
	}
	\subfloat[dofs: 1,963]{
		\includegraphics[width=\denoisingwidth]{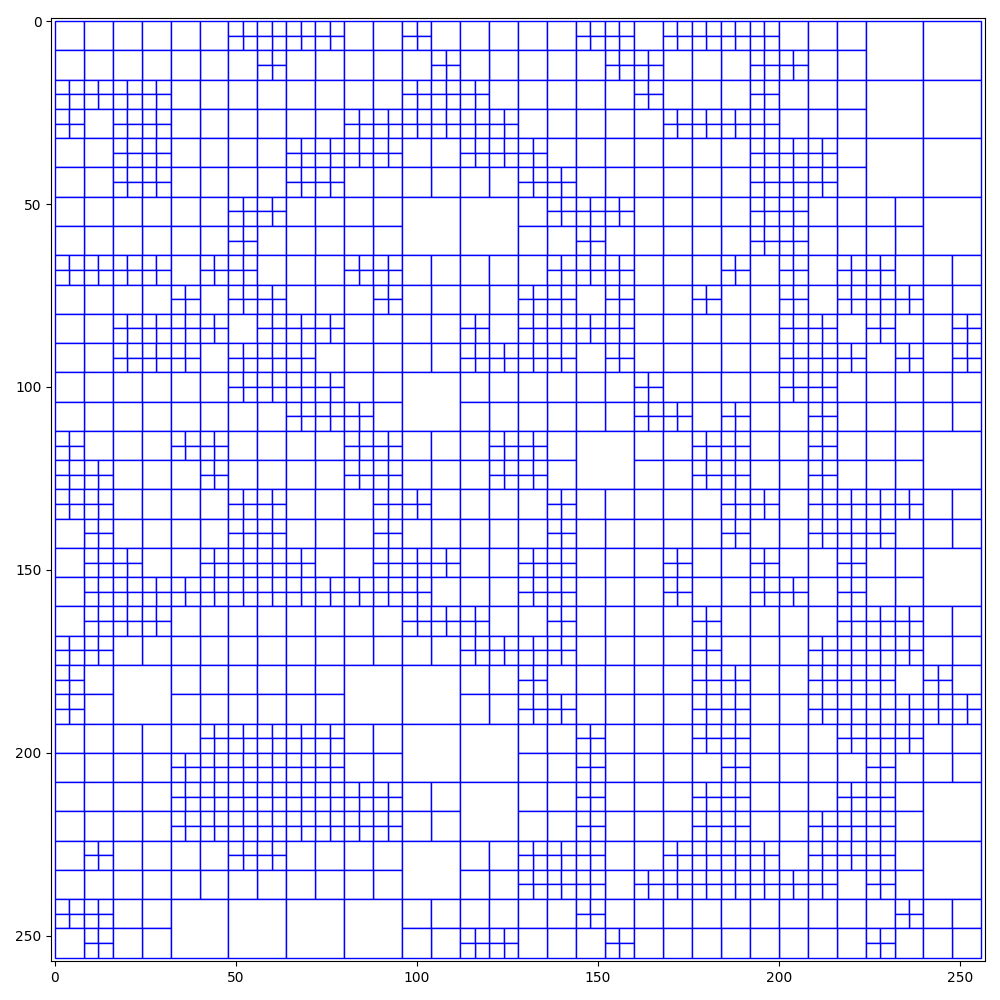}
	}
	\subfloat[dofs: 6,058]{
		\includegraphics[width=\denoisingwidth]{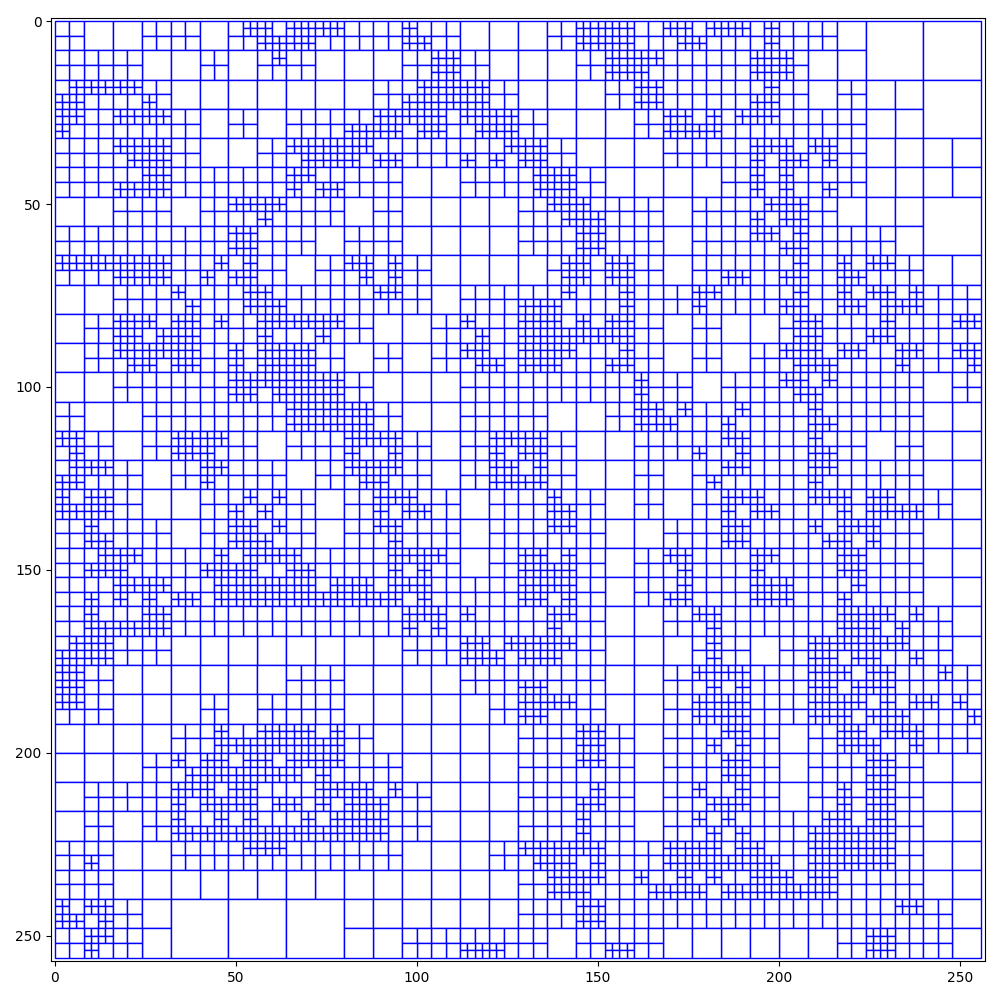}
	}
	\subfloat[dofs: 18,706]{
		\includegraphics[width=\denoisingwidth]{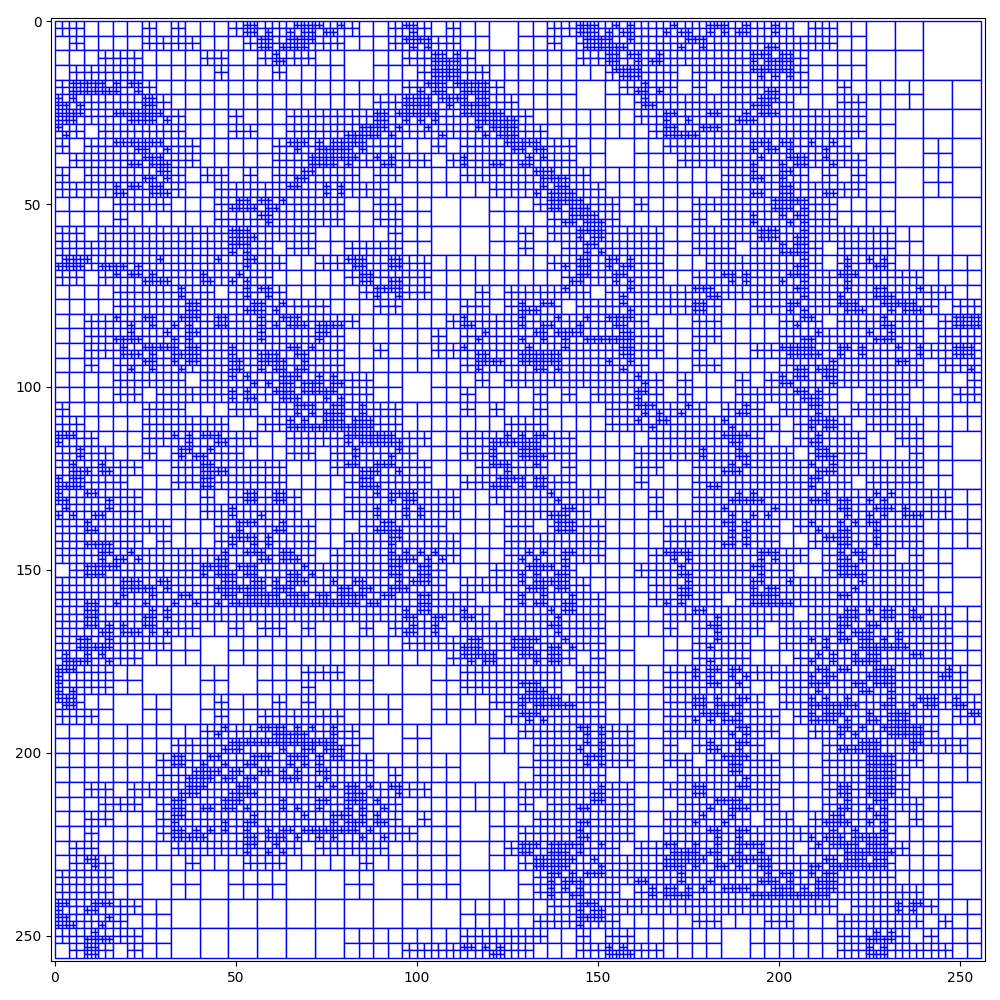}
	}
	\caption{Adapted meshes sequence for Algorithm \ref{algo:noiseless} with the \textit{bulk criterion} \eqref{eq:bulk-criterion} on the \textit{RubberWhale} example ($\sigma=0.05$).}
	\label{fig:meshes-denoising}
\end{figure}

\subsection{Application to Optical Flow Computation}

The estimation of optical flow consists in finding a vector field $\mathbf{u} = (u_x,u_y)$, called the optical flow, describing the motion of each pixel between two frames $f_0,f_1 : \Omega\subset \R^2 \to [0,1]$ of a given sequence. To this end, it is classical to consider the so-called \textit{brightness constancy assumption} which assumes that the intensity of pixels remains constant along their trajectories. More precisely, for $\mathbf{x}\in\Omega$, we assume that
\begin{equation} \label{eq:brightness-constancy}
	f_0(\mathbf{x}) = f_1\big(\mathbf{x}+\mathbf{u}(\mathbf{x})\big).
\end{equation}
Exceeding displacements $\mathbf{x}+\mathbf{u}(\mathbf{x}) \not\in\Omega$ are ignored. We use the $L^1$-$L^2$-$\TV$ model to estimate the optical flow by setting
$$ T = (\partial_x f_1 \quad \partial_y f_1), \quad S=\nabla, \quad\text{and}\quad g := T\mathbf{u_0} - (f_1-f_0), $$
which leads to
$$ T_h = (\partial_x^\text{c} f_1 \quad \partial_y^\text{c} f_1), \quad S_h=\nabla_h, \quad\text{and}\quad g_h = T_h\mathbf{u}_{\mathbf{h},0} - (f_{h,1}-f_{h,0}), $$
with $\mathbf{u}_0$ a given initial guess of the optical flow and $\mathbf{u}_{\mathbf{h},0}$ its discretization, and $f_{h,0},f_{h,1}$ the discretization of $f_0$ and $f_1$. We recall in Algorithm \ref{algo:warping}, the optical flow estimation algorithm using warping \cite[Algorithm 3]{Hilb2023Jul}. The warping technique consists in replacing $f_1$ by a \textit{warped} image $f_\text{w}(\mathbf{x}) := f_1\big(\mathbf{x}+\mathbf{u}_0(\mathbf{x})\big)$, introduced after the linearization of the optical flow equation \eqref{eq:brightness-constancy}. The linearized optical flow equation provides flow information only in the image gradient direction, a phenomenon also known as \textit{aperture problem}. The optical flow estimation algorithm with warping reads: 

\IncMargin{1.5em}
\begin{algorithm}[h]
    \SetAlgoLined
    \KwData{images $f_{h,0},f_{h,1}$, initial optical flow guess $\mathbf{u}_{\mathbf{h},0}$.}
    \KwResult{motion fields $(\mathbf{u}_{\mathbf{h},k})_{k\in\N}$.}
	$\Omega_h$ is a uniform grid where the degrees of freedom are located in the center of the pixels\;
    $f_{\text{w},0}(\mathbf{x}) = f_{h,1}\big(\mathbf{x}+\mathbf{u}_{\mathbf{h},0}(\mathbf{x})\big)$\;
    $k=0$\;
    \While{the stopping criterion is not fulfilled}{
        compute $\mathbf{u}_{\mathbf{h},k}$ using the semi-smooth Newton algorithm (Algorithm \ref{algo:semismooth-newton}) on $\Omega_h$\;
        $f_{\text{w},k}(\mathbf{x}) = f_{h,1}\big(\mathbf{x}+\mathbf{u}_{\mathbf{h},k}(\mathbf{x})\big)$\;
        $k = k + 1$\;
    }   
    
    \caption{Optical flow algorithm with warping \cite[Algorithm 3]{Hilb2023Jul}.}
    \label{algo:warping}
\end{algorithm}
\DecMargin{1.5em}
Here, the line $f_{\text{w},k}(\mathbf{x}) = f_{h,1}\big(\mathbf{x}+\mathbf{u}_{\mathbf{h},k}(\mathbf{x})\big)$ is to be understood as a bilinear interpolation on the finest uniform grid. We use in Algorithm \ref{algo:warping} the stopping criterion
$$ \frac{\|f_{\text{w},k-1}-f_{h,0}\|_{L^2(\Omega_h)} - \|f_{\text{w},k}-f_{h,0}\|_{L^2(\Omega_h)}}{\|f_{\text{w},k-1}-f_{h,0}\|_{L^2(\Omega_h)}} \leq \varepsilon_\text{warp}, $$
for some specified constant $\varepsilon_\text{warp}$, which ensures that warping continues only as long as the remaining image difference $\|f_{\text{w},k}-f_{h,0}\|_{L^2(\Omega_h)}$ is being reduced sufficiently. Now, we recall in Algorithm \ref{algo:warping-while-ctf} the optical flow estimation algorithm with adaptive warping \cite[Algorithm 1]{Alkamper2024}:

\IncMargin{1.5em}
\begin{algorithm}[h]
    \SetAlgoLined
    \KwData{images $f_{h,0},f_{h,1}$ of size $w\times h$, initial optical flow guess $\mathbf{u}_{\mathbf{h},0}$, maximal number of refinement $N_\text{ref\_max}\in\N$.}
    \KwResult{motion fields $(\mathbf{u}_{\mathbf{h},k})_{k\in\N}$.}

    $f_{\text{w},0}(\mathbf{x}) = f_1\big(\mathbf{x}+\mathbf{u}_{\mathbf{h},0}(\mathbf{x})\big)$\;
    $\Omega_h^1$ uniform coarse grid of size $\lfloor\frac{w}{2^{N_\text{ref\_max}}}\rfloor \times \lfloor\frac{h}{2^{N_\text{ref\_max}}}\rfloor$\;
    $k=1$\;
    $N_\text{ref}=0$\;
    \While{$N_\text{ref} < N_\text{ref\_max}$}{
        compute $\mathbf{u}_{\mathbf{h},k}$ using the semi-smooth Newton algorithm (Algorithm \ref{algo:semismooth-newton}) on the mesh $\Omega_h^k$\;
        $f_{\text{w},k}(\mathbf{x}) = f_{h,1}\big(\mathbf{x}+\mathbf{u}_{\mathbf{h},k}(\mathbf{x})\big)$\;
        \If{$\frac{\|f_{\text{w},k-1}-f_{h,0}\|_{L^2(\Omega_h)} - \|f_{\text{w},k}-f_{h,0}\|_{L^2(\Omega_h)}}{\|f_{\text{w},k-1}-f_{h,0}\|_{L^2(\Omega_h)}} \leq \varepsilon_\text{warp}$}{
            refine the grid $\Omega_h^k$, save it in $\Omega_h^{k+1}$ and reproject image data ($L^2$-projection)\;
            $N_\text{ref} = N_\text{ref} + 1$\;
        }
        $k = k + 1$\;
    }   
    
    \caption{Optical flow algorithm with adaptive warping \cite[Algorithm 1]{Alkamper2024}.}
    \label{algo:warping-while-ctf}
\end{algorithm}
\DecMargin{1.5em}

Unlike the image reconstruction case, we choose to refine $75\%$ of the degree of freedom with the highest error $\eta_h$ instead of using the Dörfler marking strategy. In practice, for optical flow, the primal-dual gap error is significant at a few elements, indicating that only a small portion of elements contribute to the overall error. Consequently, the Dörfler strategy results in an excessively sparse grid, leading to inaccurate optical flow estimation.

We recall the definition of the endpoint error \cite{Baker2007}
$$ \text{EE} = \sqrt{(u_x-u_{\text{GT},x})^2 + (u_y-u_{\text{GT},y})^2}, $$
where $\mathbf{u}=(u_x,u_y)$ is an approximation of the optical flow and $\mathbf{u}_\text{GT} = (u_{\text{GT},x}, u_{\text{GT},y})$ is the \textit{ground-truth} optical flow and the definition of the angular error \cite{Baker2007}
$$ \text{AE} = \arccos{\frac{1 + u_x u_{\text{GT},x} + u_y u_{\text{GT},y}}{\sqrt{1 + u_x^2 + u_y^2}\sqrt{1 + u_{\text{GT},x}^2 + u_{\text{GT},y}^2}}}. $$
These two errors are the most commonly used measures of performance for optical flow. The AE between a flow vector $\mathbf{u}=(u_x,u_y)$ and the ground-truth flow $\mathbf{u}_\text{GT} = (u_{\text{GT},x}, u_{\text{GT},y})$ is the angle between the estimated optical flow vectors and the \textit{ground-truth} vectors, while the EE quantify the Euclidean norm of the difference between the estimated optical flow vectors and the \textit{ground-truth} vectors.

In our optical flow experiments, we set $\alpha_1 = 3$, $\alpha_2 = 0$, $\lambda=1$, $\beta = 1\times 10^{-5}$ and $\gamma_1=\gamma_2 = 2\times 10^{-4}$ and as initial optical flow guess $\mathbf{u}_{\mathbf{h},0} = 0$. For Algorithm \ref{algo:warping} and Algorithm \ref{algo:warping-while-ctf}, we chose $\varepsilon_\text{warp} = 5\times 10^{-2}$ and for Algorithm \ref{algo:semismooth-newton}, we chose $\varepsilon = 1\times 10^{-3}$.

In Table \ref{tab:optical-flow}, we give endpoint errors (EE) and angular errors (AE), each given by mean and standard deviation, and computational times in seconds for the \textit{Middlebury} dataset where the \textit{ground-truth} optical flows are known for the computed optical flow $\mathbf{u}$ without warping (first iteration of Algorithm \ref{algo:warping}), with warping (Algorithm \ref{algo:warping}) and the computed optical flow $\mathbf{u}$ using the adaptive warping (Algorithm \ref{algo:warping-while-ctf}). 
The \textit{ground-truth} optical flow is the exact optical flow between the image $f_{h,0}$ and $f_{h,1}$. Unlike denoising, it is uncertain whether the discontinuities in optical flow align with those in the images. Therefore, there is no need to adjust the size of the input images, as the alignment of discontinuities is not as critical for optical flow estimation.

For the examples \textit{Dimetrodon}, \textit{Grove3}, \textit{Hydrangea}, \textit{RubberWhale} and \textit{Venus} the uniform grid with warping, i.e. Algorithm \ref{algo:warping}, gives the lowest endpoint and angular errors. However, this accuracy has a price in terms of number of degrees of freedom (dofs), and as a result, of computation time. On the other hand, the results without warping give the highest errors in all the examples due to the linearization of the optical flow estimation problem. Finally, the adaptive warping, Algorithm \ref{algo:warping-while-ctf}, gives similar results as for the uniform warping, but with a computation cost similar as for the non-warping case. For \textit{Grove2}, \textit{Urban2} and \textit{Urban3}, we notice that the errors with adaptive warping are even smaller than uniform warping.

We give in Figure \ref{fig:optical-flow} the images from the \textit{Middlebury} dataset $f_0$, $f_1$, the image difference $f_0-f_1$, the computed optical flow using without warping (first iteration of Algorithm \ref{algo:warping}) and with warping (Algorithm \ref{algo:warping}), the computed optical flow using the adaptive warping (Algorithm \ref{algo:warping-while-ctf}) and the \textit{ground-truth} optical flow. The color-coded images representing optical flow fields are normalized by the maximum motion of the \textit{ground-truth} flow data and black areas of the ground truth data represent unknown flow information, e.g. due to occlusion.

Figure \ref{fig:meshes} depicts a sequence of adapted meshes with the criterion \eqref{eq:grid-ref-indicator-reg-shifted}. The mesh seems to be finer near the edges of the \textit{ground-truth}, which is where the movement is concentrated.

Note that the results are similar to the ones in \cite{Alkamper2024}, where the adaptive finite element method have been used to solve the semi-smooth Newton iterations.

\begin{table}[!htbp]
    \centering
    \begin{tabular}{ll|ccccc}
        Benchmark & Algorithm & EE-mean & EE-stddev & AE-mean & AE-stddev & Time (s) \\
        \hline\hline
        Dimetrodon  & Uniform without warping              & 1.47 & 0.74 & 0.63 & 0.36 & 429.44 \\
                    & Uniform with warping                    & 0.26 & 0.25 & 0.19 & 0.49 & 2,598.29 \\
                    & Adaptive warping, Algorithm \ref{algo:warping-while-ctf} & 0.42 & 0.37 & 0.23 & 0.48 & 1,066.91 \\
        \hline
        Grove2      & Uniform without warping              & 2.79 & 0.47 & 0.98 & 0.18 & 726.93 \\
                    & Uniform with warping                   & 1.09 & 1.16 & 0.35 & 0.42 & 9,211.81 \\
                    & Adaptive warping, Algorithm \ref{algo:warping-while-ctf} & 0.38 & 0.62 & 0.10 & 0.18 & 1,686.60 \\
        \hline
        Grove3      & Uniform without warping              & 3.51 & 2.35 & 0.86 & 0.34 & 723.02 \\
                    & Uniform with warping                   & 1.22 & 1.75 & 1.18 & 0.30 & 13,010.78 \\
                    & Adaptive warping, Algorithm \ref{algo:warping-while-ctf} & 1.32 & 1.67 & 0.22 & 0.36 & 1,823.77 \\
        \hline
        Hydrangea   & Uniform without warping              & 3.10 & 1.20 & 0.78 & 0.29 & 349.94 \\
                    & Uniform with warping                   & 0.25 & 0.53 & 0.14 & 0.41 & 6,830.48 \\
                    & Adaptive warping, Algorithm \ref{algo:warping-while-ctf} & 0.45 & 0.69 & 0.18 & 0.43 & 1,307.07 \\
        \hline
        RubberWhale & Uniform without warping              & 0.54 & 0.60 & 0.30 & 0.32 & 424.89 \\
                    & Uniform with warping                   & 0.29 & 0.54 & 0.17 & 0.34 & 1,799.16 \\
                    & Adaptive warping, Algorithm \ref{algo:warping-while-ctf} & 0.47 & 0.61 & 0.29 & 0.38 & 1,257.92 \\
        \hline
        Urban2      & Uniform without warping              & 7.95 & 8.12 & 0.82 & 0.42 &  793.44 \\
                    & Uniform with warping                   & 6.19 & 7.70 & 0.27 & 0.32 &10,212.01 \\
                    & Adaptive warping, Algorithm \ref{algo:warping-while-ctf} & 1.52 & 2.30 & 0.17 & 0.31 & 1,493.25 \\
        \hline
        Urban3      & Uniform without warping              & 6.89 & 4.45 & 1.02 & 0.29 & 2,401.56 \\
                    & Uniform with warping                   & 4.86 & 4.11 & 0.49 & 0.62 & 11,855.79 \\
                    & Adaptive warping, Algorithm \ref{algo:warping-while-ctf} & 2.38 & 3.59 & 0.39 & 0.79 & 1,401.86 \\
        \hline
        Venus       & Uniform without warping              & 3.32 & 1.72 & 0.82 & 0.26 & 273.16 \\
                    & Uniform with warping                   & 0.70 & 1.01 & 0.19 & 0.44 & 4,514.02 \\
                    & Adaptive warping, Algorithm \ref{algo:warping-while-ctf} & 1.16 & 1.33 & 0.32 & 0.59 & 744.36 \\
    \end{tabular}\vspace{0.4cm}
    
    \caption{Quantitative results: endpoint errors (EE) and angular errors (AE), each given by mean and standard deviation, and computational times in seconds.}
    \label{tab:optical-flow}
\end{table}

\begin{figure}[!htbp]
	\centering
	\subfloat{
		\includegraphics[width=2.2cm]{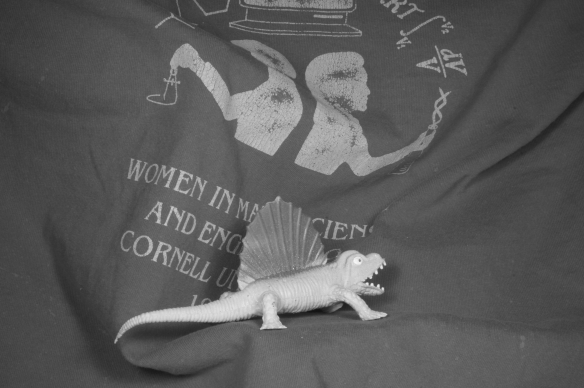}
	}
	\subfloat{
		\includegraphics[width=2.2cm]{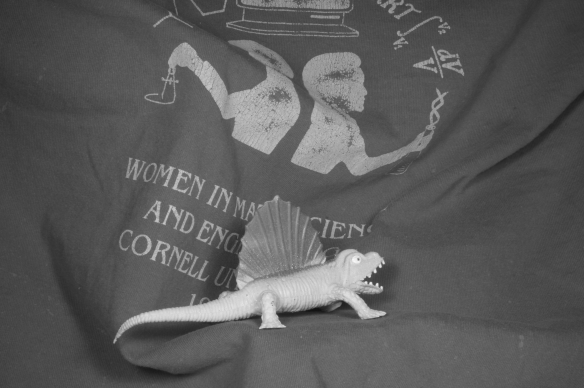}
	}
	\subfloat{
		\includegraphics[width=2.2cm]{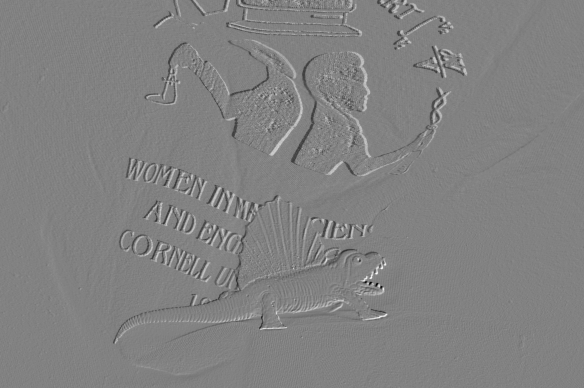}
	}
    \subfloat{
		\includegraphics[width=2.2cm]{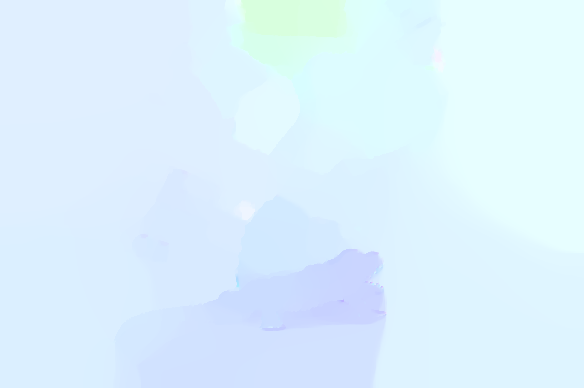}
	}
        \subfloat{
		\includegraphics[width=2.2cm]{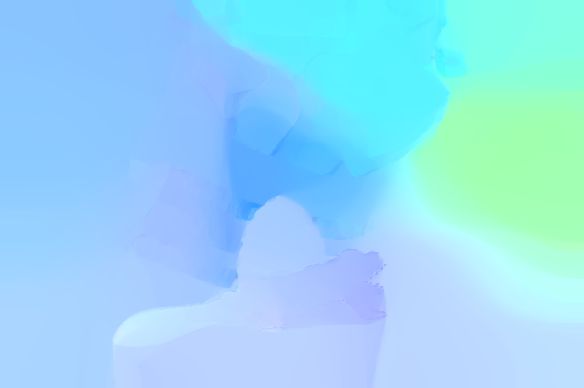}
	}
        \subfloat{
		\includegraphics[width=2.2cm]{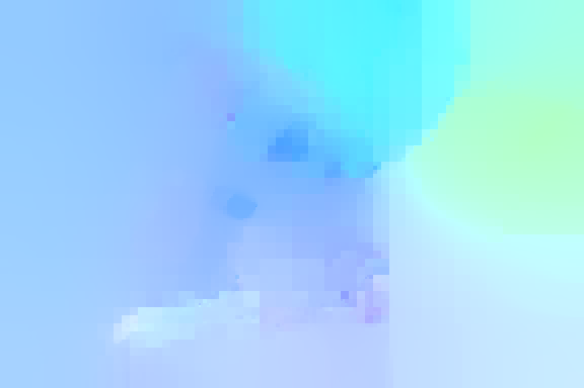}
	}
        \subfloat{
		\includegraphics[width=2.2cm]{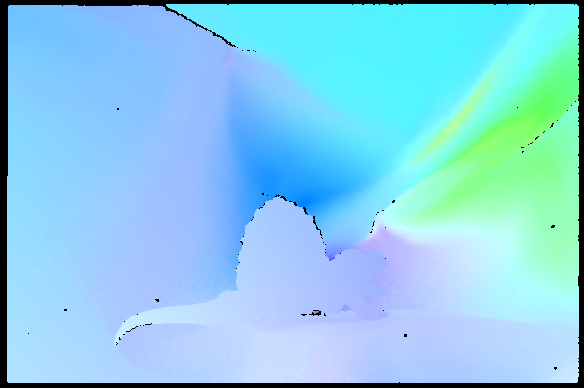}
	}
    \quad
    \subfloat{
		\includegraphics[width=2.2cm]{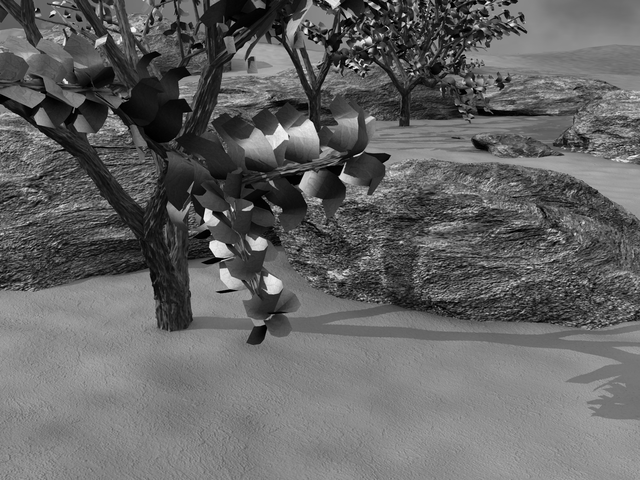}
	}
	\subfloat{
		\includegraphics[width=2.2cm]{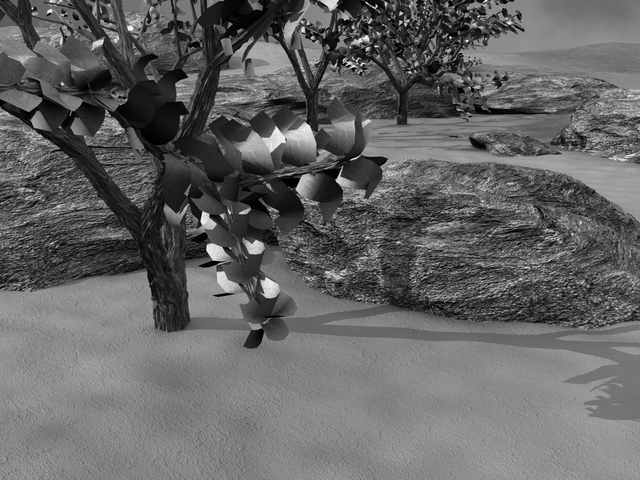}
	}
	\subfloat{
		\includegraphics[width=2.2cm]{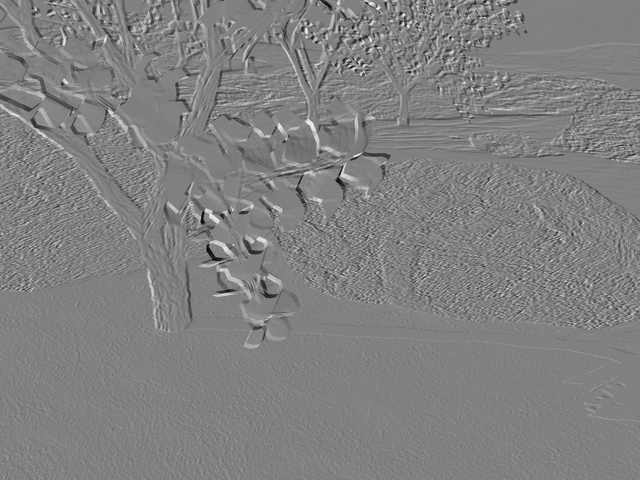}
	}
    \subfloat{
		\includegraphics[width=2.2cm]{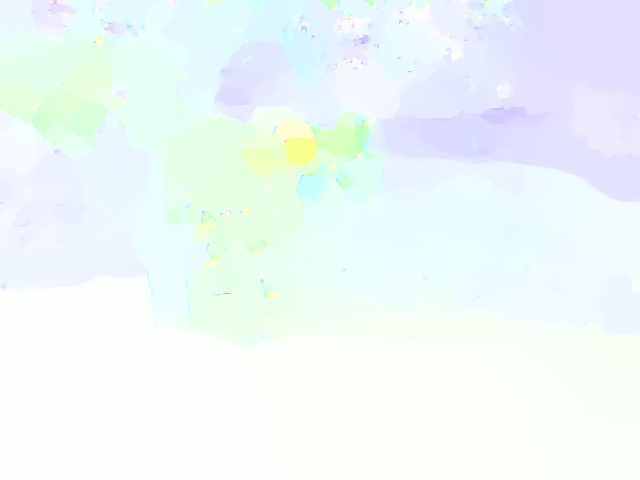}
	}
    \subfloat{
		\includegraphics[width=2.2cm]{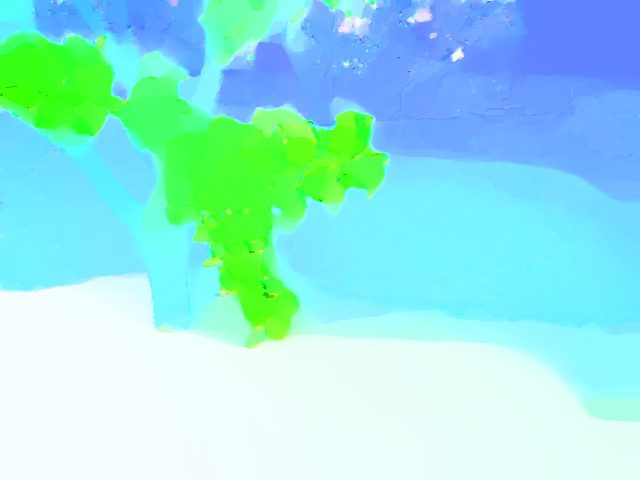}
	}
        \subfloat{
	\includegraphics[width=2.2cm]{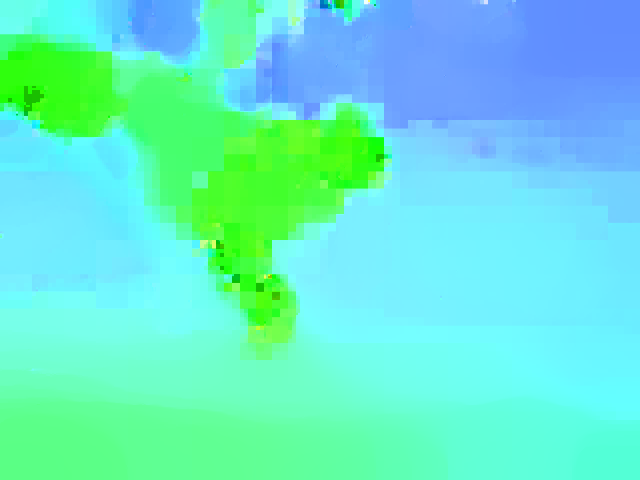}
	}
        \subfloat{
		\includegraphics[width=2.2cm]{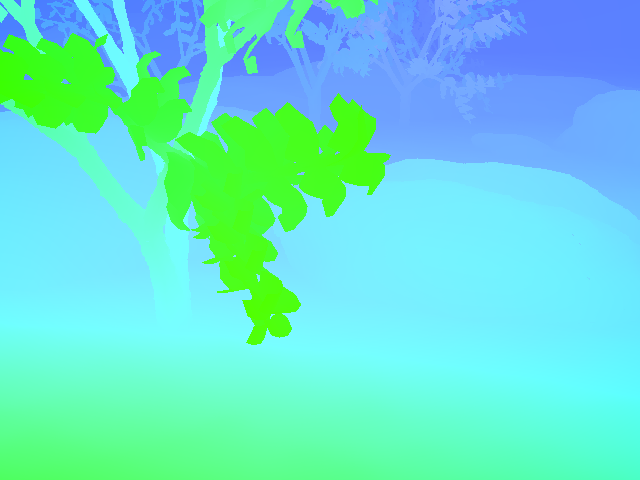}
	}
    \quad
    \subfloat{
		\includegraphics[width=2.2cm]{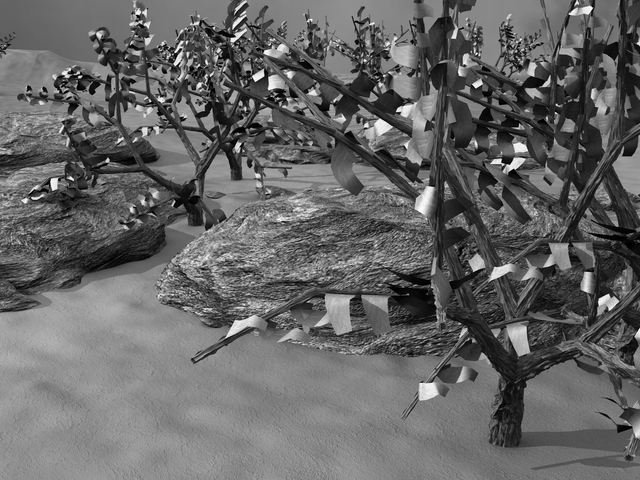}
	}
	\subfloat{
		\includegraphics[width=2.2cm]{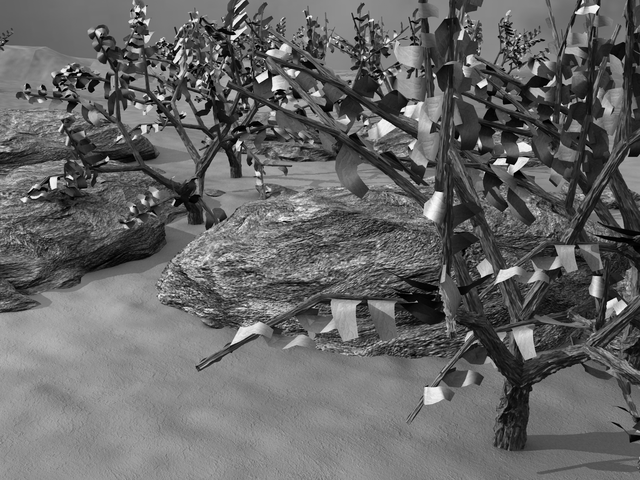}
	}
	\subfloat{
		\includegraphics[width=2.2cm]{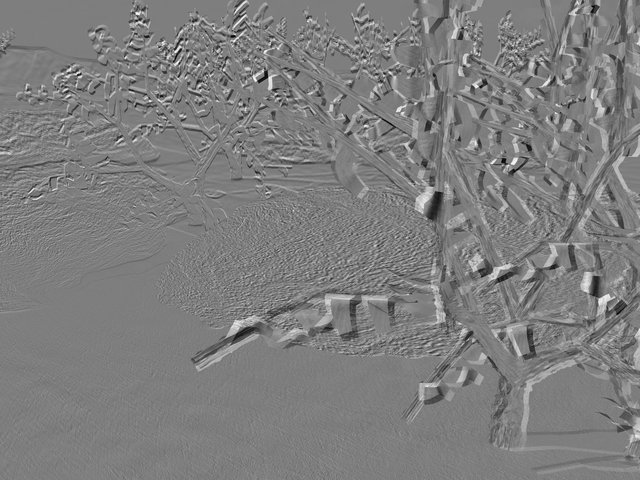}
	}
        \subfloat{
		\includegraphics[width=2.2cm]{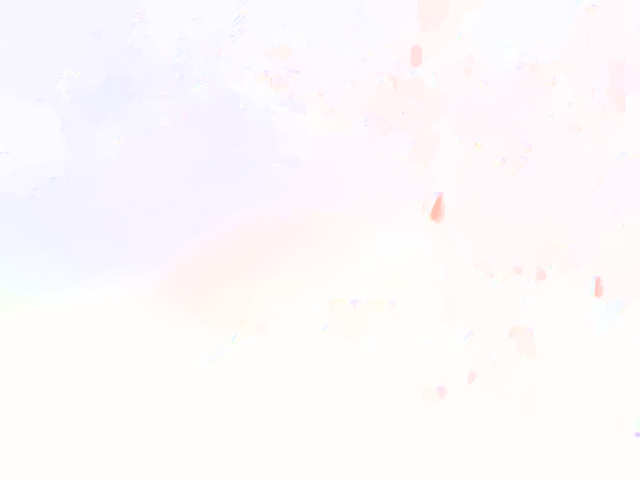}
	}
        \subfloat{
		\includegraphics[width=2.2cm]{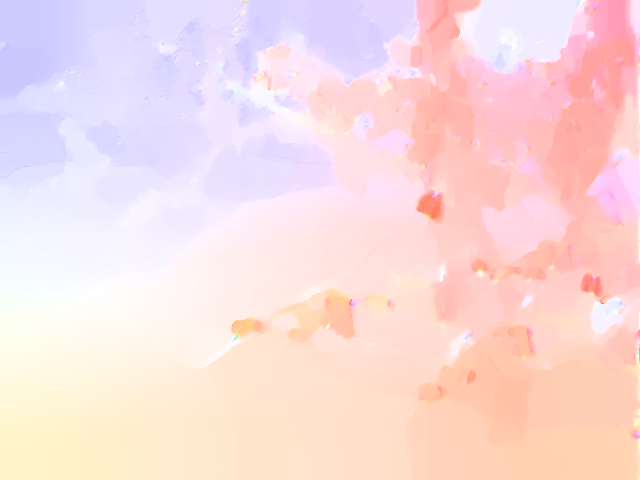}
	}
        \subfloat{
		\includegraphics[width=2.2cm]{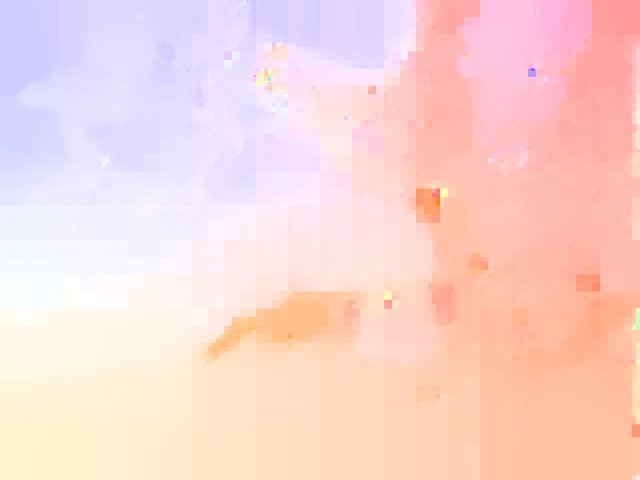}
	}
        \subfloat{
		\includegraphics[width=2.2cm]{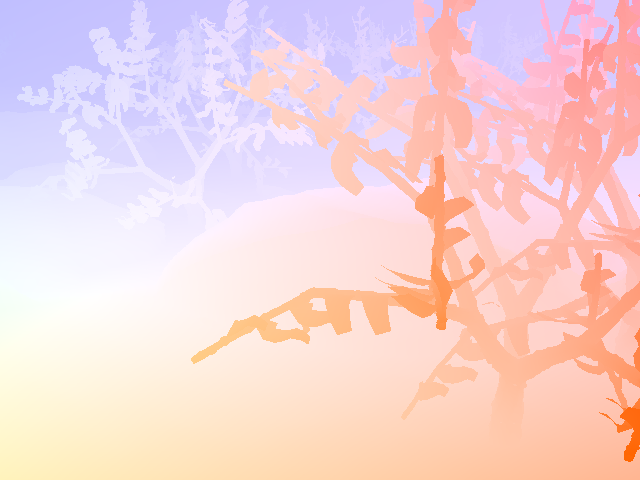}
	}
        \quad
        \subfloat{
		\includegraphics[width=2.2cm]{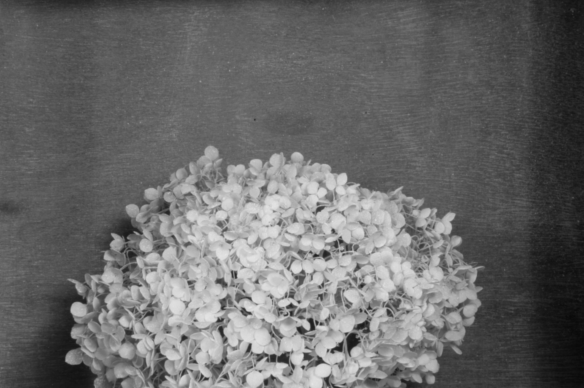}
	}
	\subfloat{
		\includegraphics[width=2.2cm]{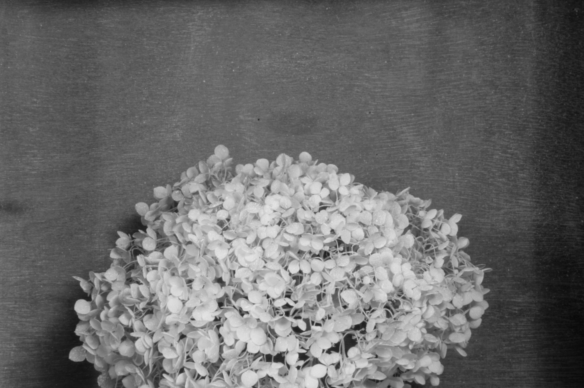}
	}
	\subfloat{
		\includegraphics[width=2.2cm]{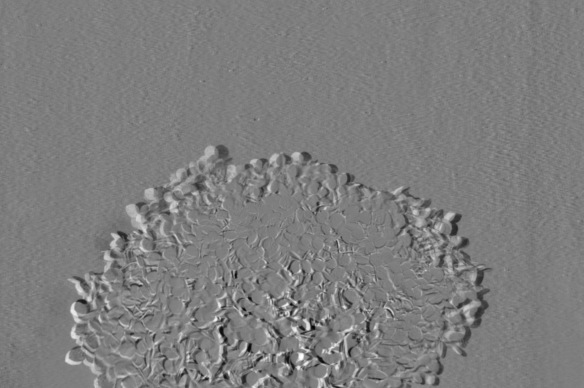}
	}
        \subfloat{
		\includegraphics[width=2.2cm]{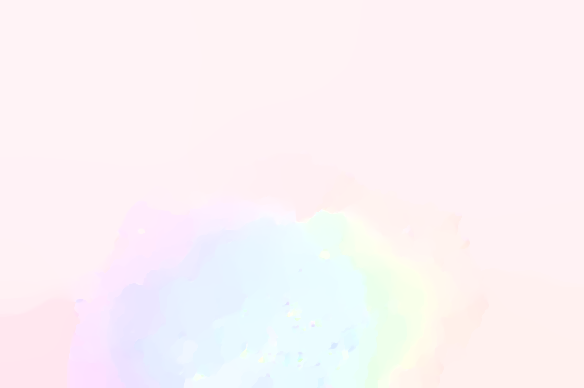}
	}
        \subfloat{
		\includegraphics[width=2.2cm]{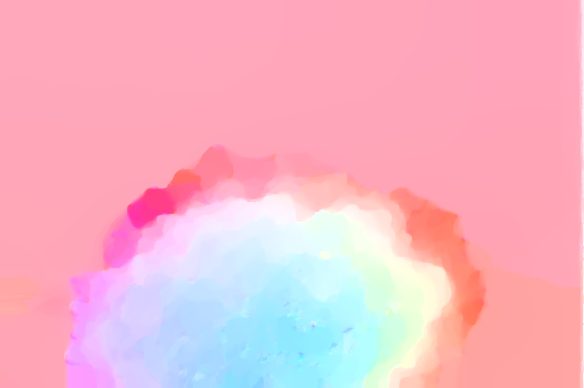}
	}
        \subfloat{
		\includegraphics[width=2.2cm]{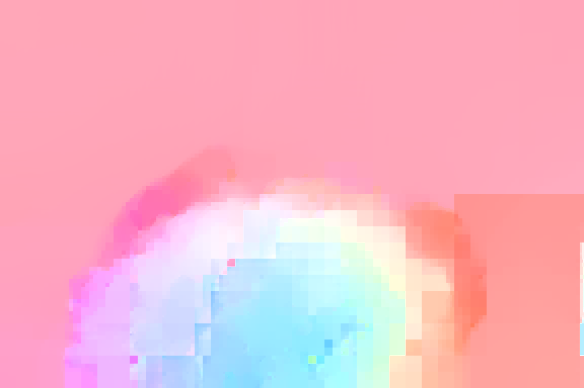}
	}
        \subfloat{
		\includegraphics[width=2.2cm]{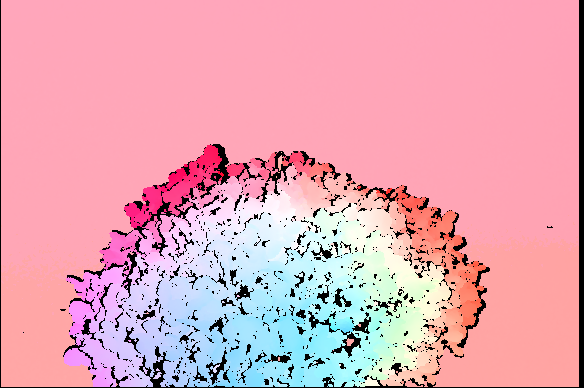}
	}
        \quad
        \subfloat{
		\includegraphics[width=2.2cm]{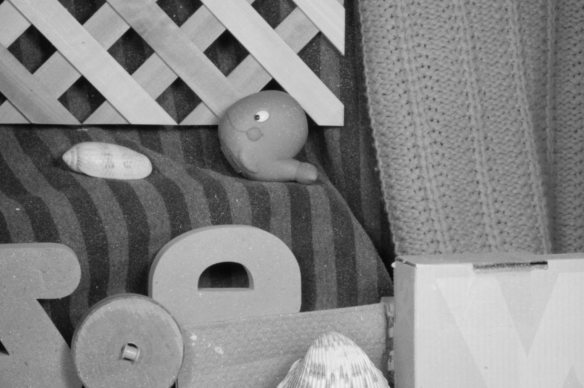}
	}
	\subfloat{
		\includegraphics[width=2.2cm]{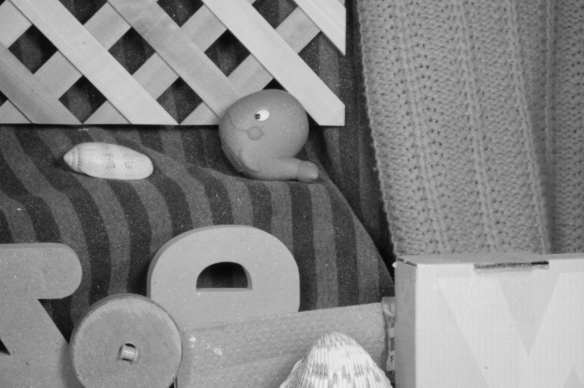}
	}
	\subfloat{
		\includegraphics[width=2.2cm]{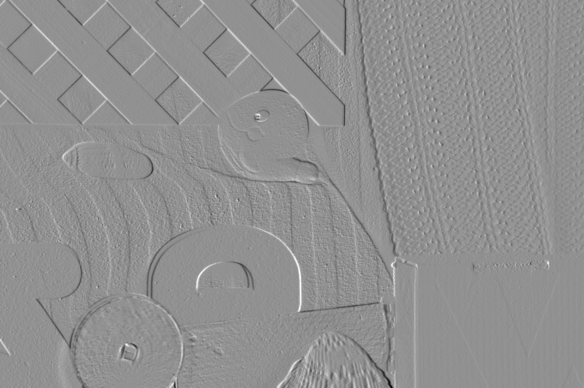}
	}
        \subfloat{
		\includegraphics[width=2.2cm]{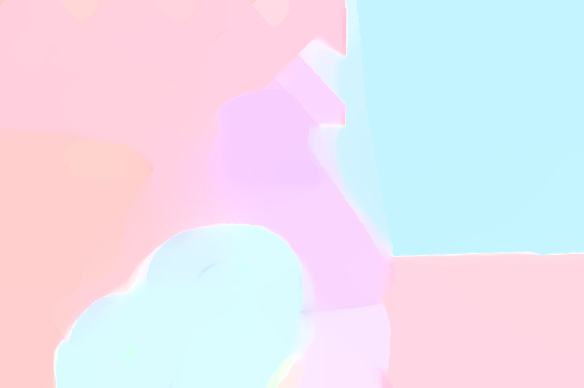}
	}
        \subfloat{
		\includegraphics[width=2.2cm]{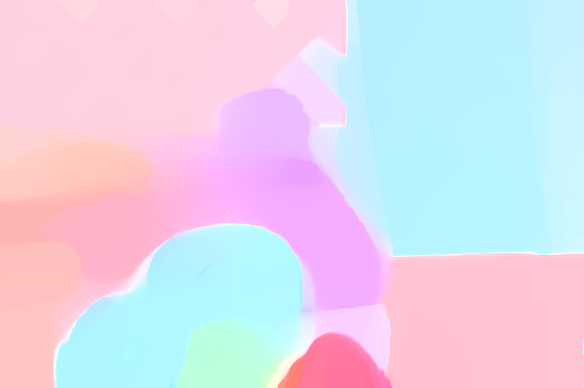}
	}
        \subfloat{
		\includegraphics[width=2.2cm]{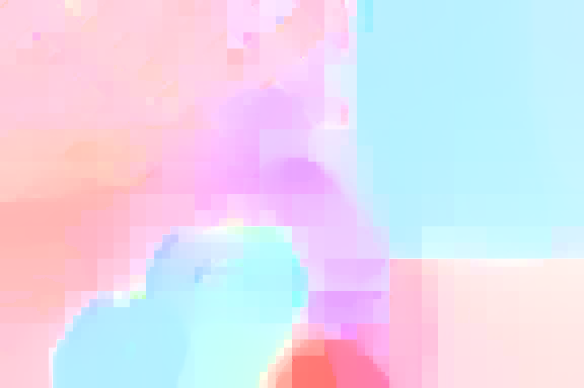}
	}
        \subfloat{
		\includegraphics[width=2.2cm]{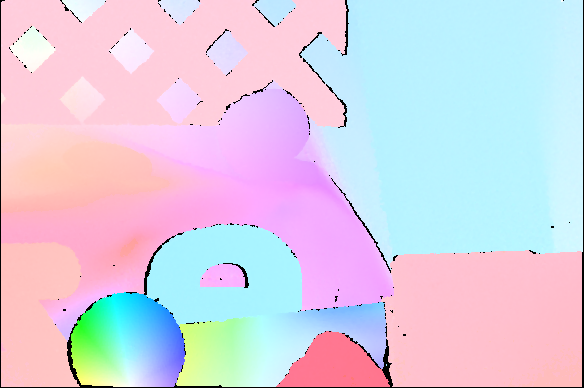}
	}
        \quad
        \subfloat{
		\includegraphics[width=2.2cm]{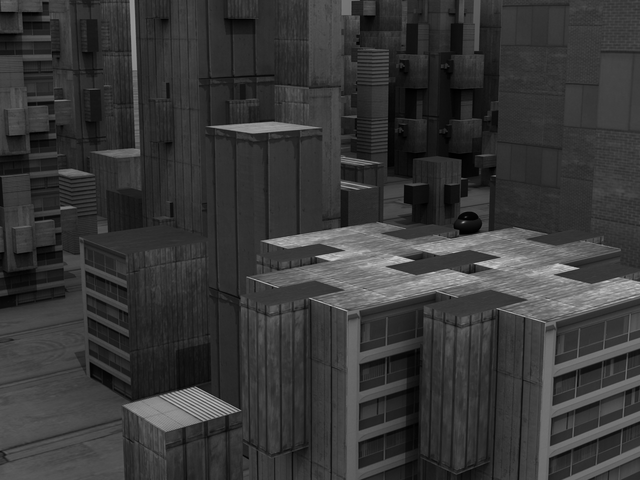}
	}
	\subfloat{
		\includegraphics[width=2.2cm]{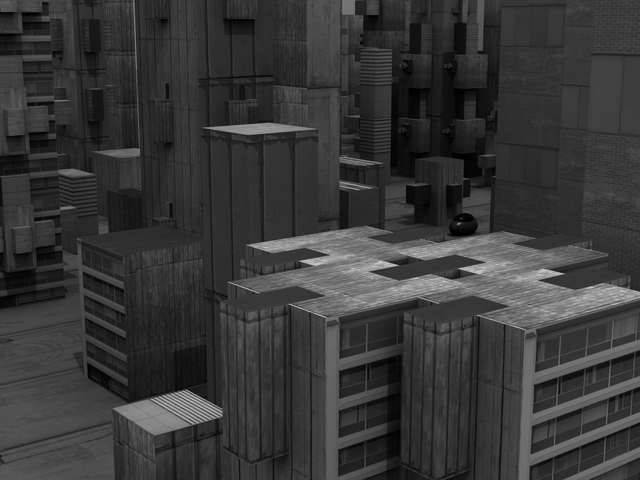}
	}
	\subfloat{
		\includegraphics[width=2.2cm]{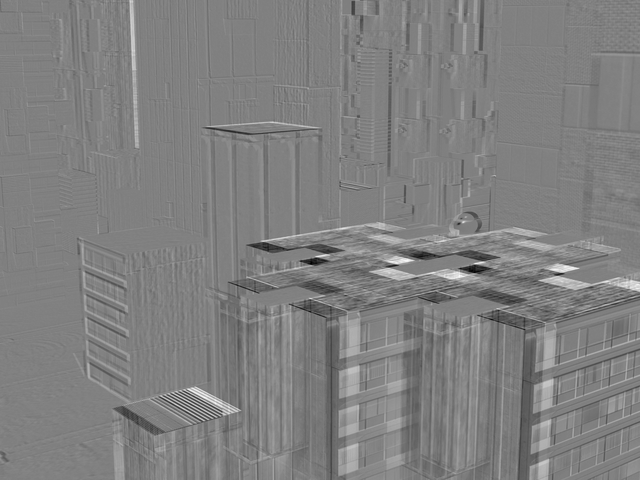}
	}
        \subfloat{
		\includegraphics[width=2.2cm]{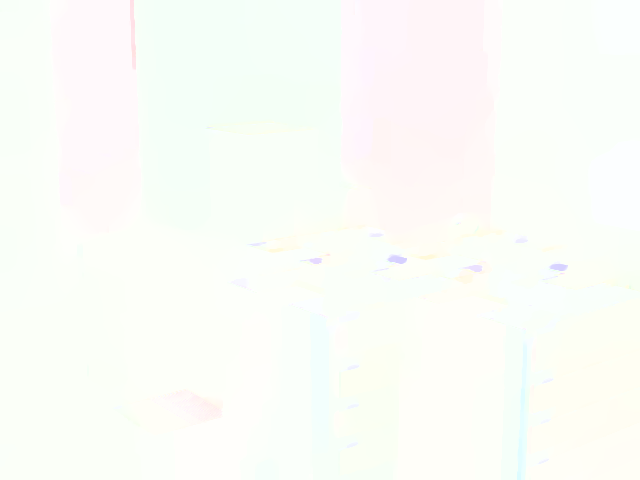}
	}
        \subfloat{
		\includegraphics[width=2.2cm]{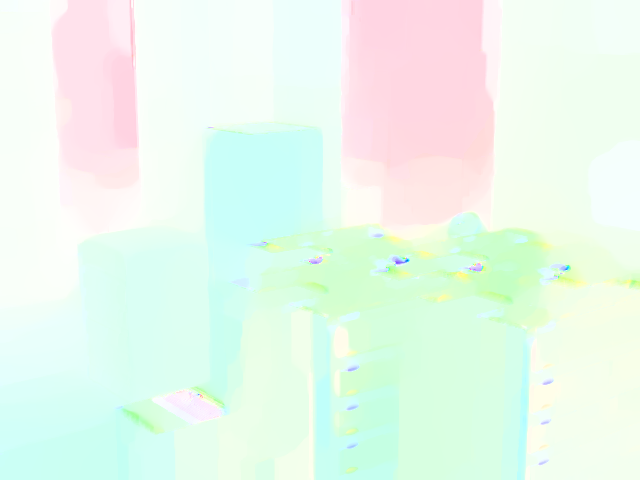}
	}
        \subfloat{
		\includegraphics[width=2.2cm]{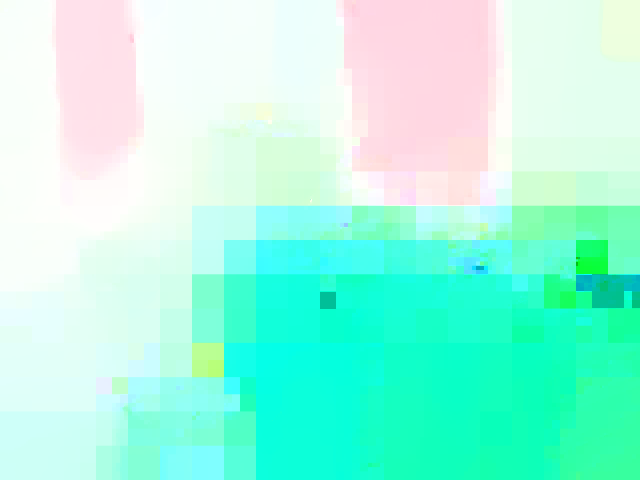}
	}
        \subfloat{
		\includegraphics[width=2.2cm]{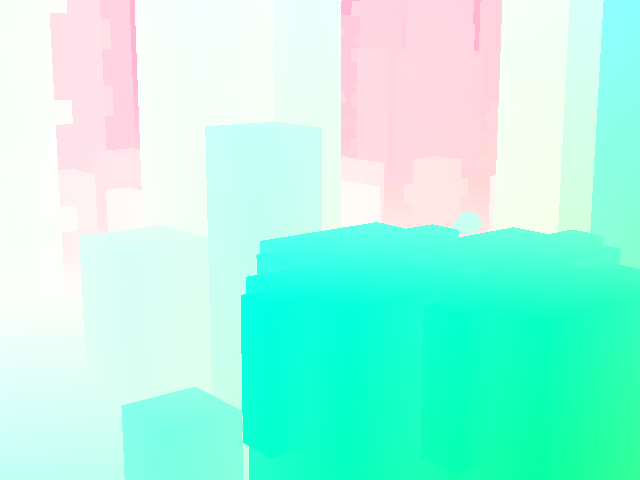}
	}
        \quad
        \subfloat{
		\includegraphics[width=2.2cm]{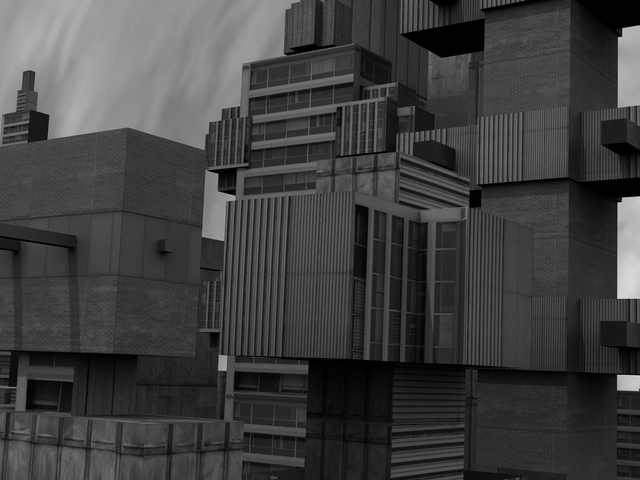}
	}
	\subfloat{
		\includegraphics[width=2.2cm]{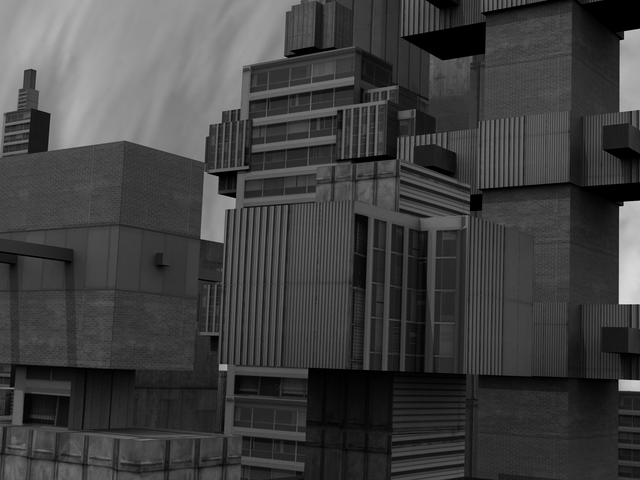}
	}
	\subfloat{
		\includegraphics[width=2.2cm]{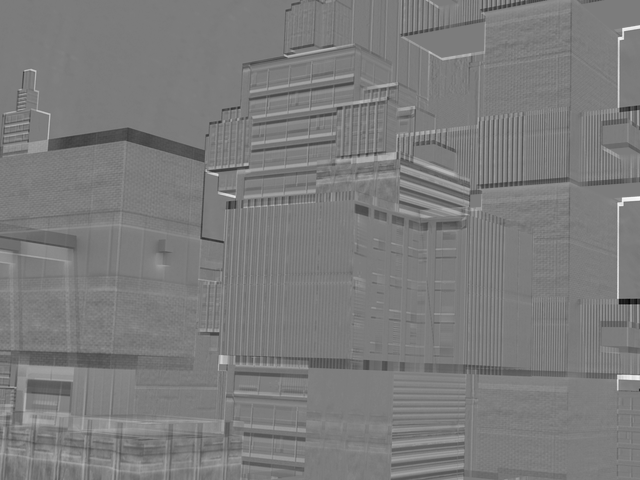}
	}
        \subfloat{
		\includegraphics[width=2.2cm]{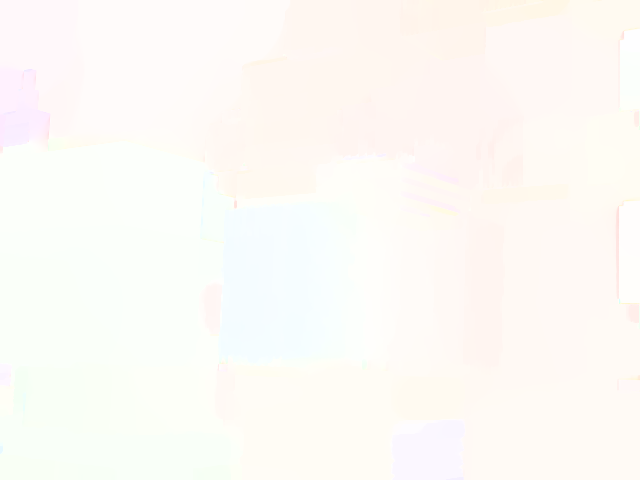}
	}
        \subfloat{
		\includegraphics[width=2.2cm]{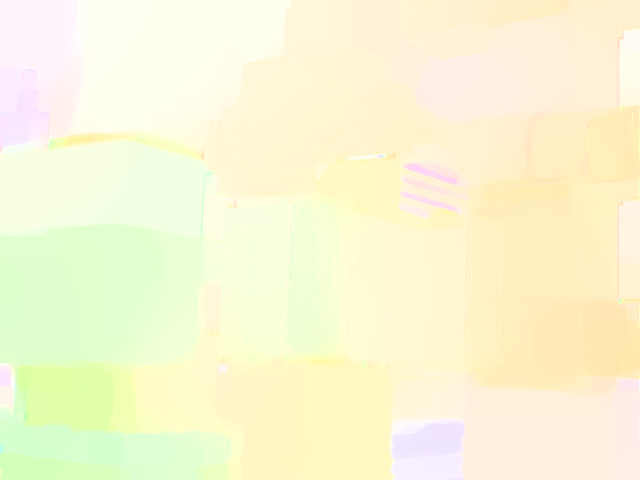}
	}
        \subfloat{
		\includegraphics[width=2.2cm]{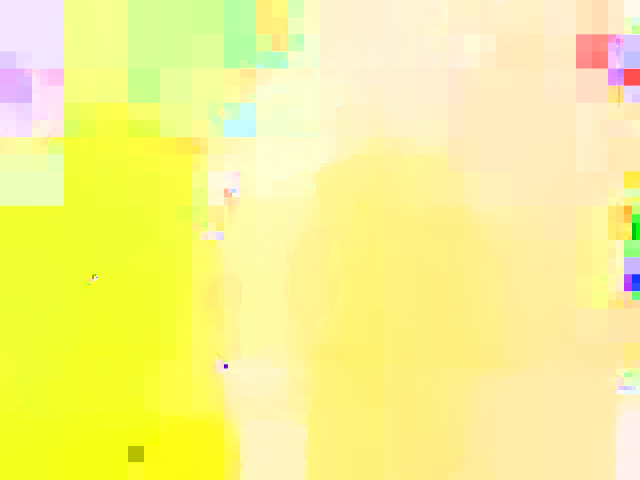}
	}
        \subfloat{
		\includegraphics[width=2.2cm]{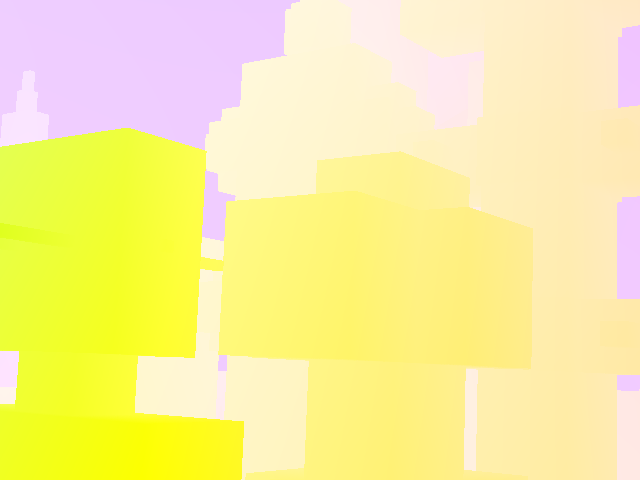}
	}
        \quad
        \subfloat{
		\includegraphics[width=2.2cm]{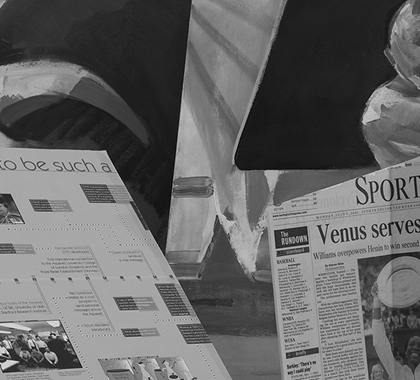}
	}
	\subfloat{
		\includegraphics[width=2.2cm]{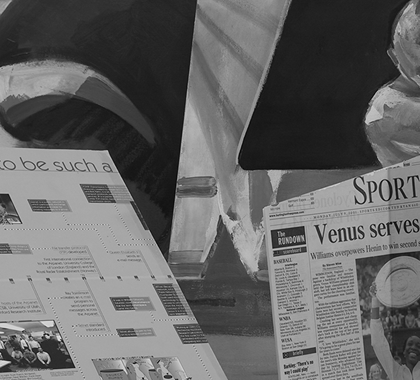}
	}
	\subfloat{
		\includegraphics[width=2.2cm]{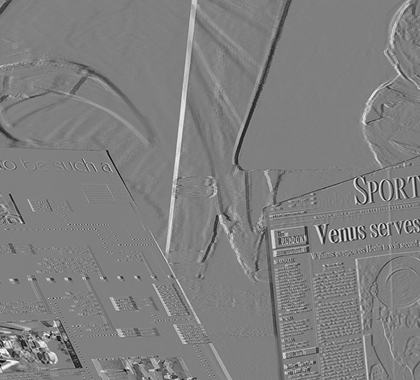}
	}
        \subfloat{
		\includegraphics[width=2.2cm]{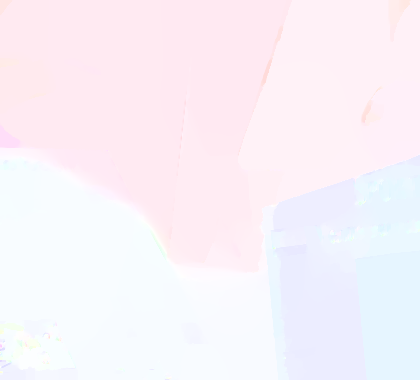}
	}
        \subfloat{
		\includegraphics[width=2.2cm]{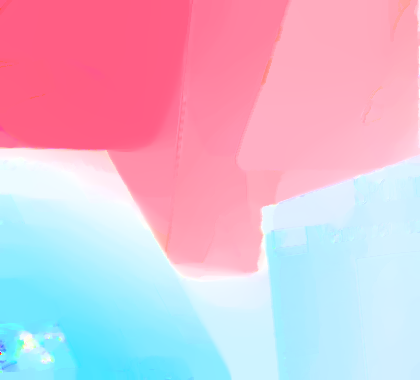}
	}
        \subfloat{
		\includegraphics[width=2.2cm]{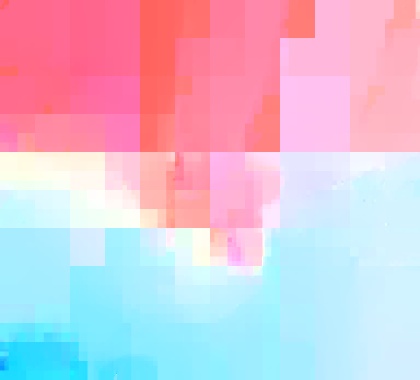}
	}
        \subfloat{
		\includegraphics[width=2.2cm]{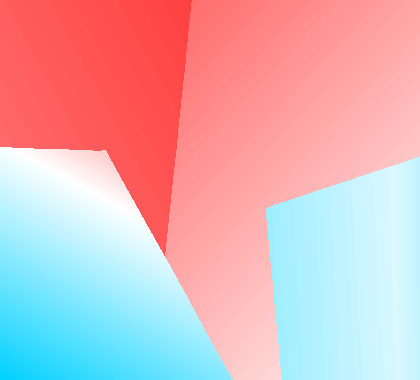}
	}
 
	\caption{Middlebury Optical Flow Benchmark; columns from left to right: $f_0$, $f_1$, image difference $f_0-f_1$, computed optical flow $\mathbf{u}$ using without warping and with warping, computed optical flow $\mathbf{u}$ using Algorithm \ref{algo:warping-while-ctf} (adaptive warping), ground truth optical flow. Benchmarks from top to bottom: \textit{Dimetrodon}, \textit{Grove2}, \textit{Grove3}, \textit{Hydrangea}, \textit{RubberWhale}, \textit{Urban2}, \textit{Urban3}, \textit{Venus}.}
	\label{fig:optical-flow}
\end{figure}

\def\ofwidth{3.5cm}
\begin{figure}[!htbp]
	\centering
	\subfloat[dofs: 54]{
		\includegraphics[width=\ofwidth]{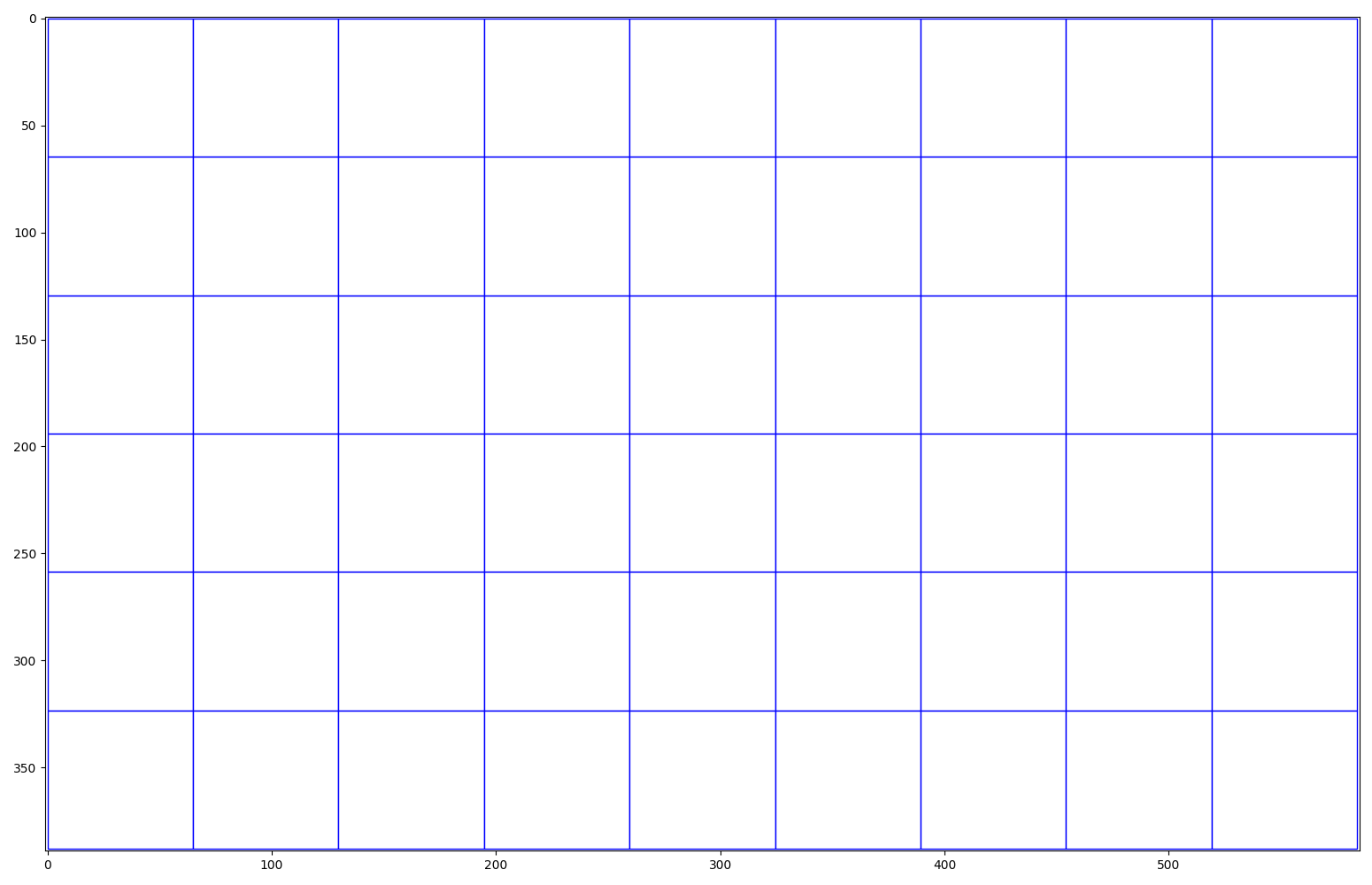}
	}
	\subfloat[dofs: 174]{
		\includegraphics[width=\ofwidth]{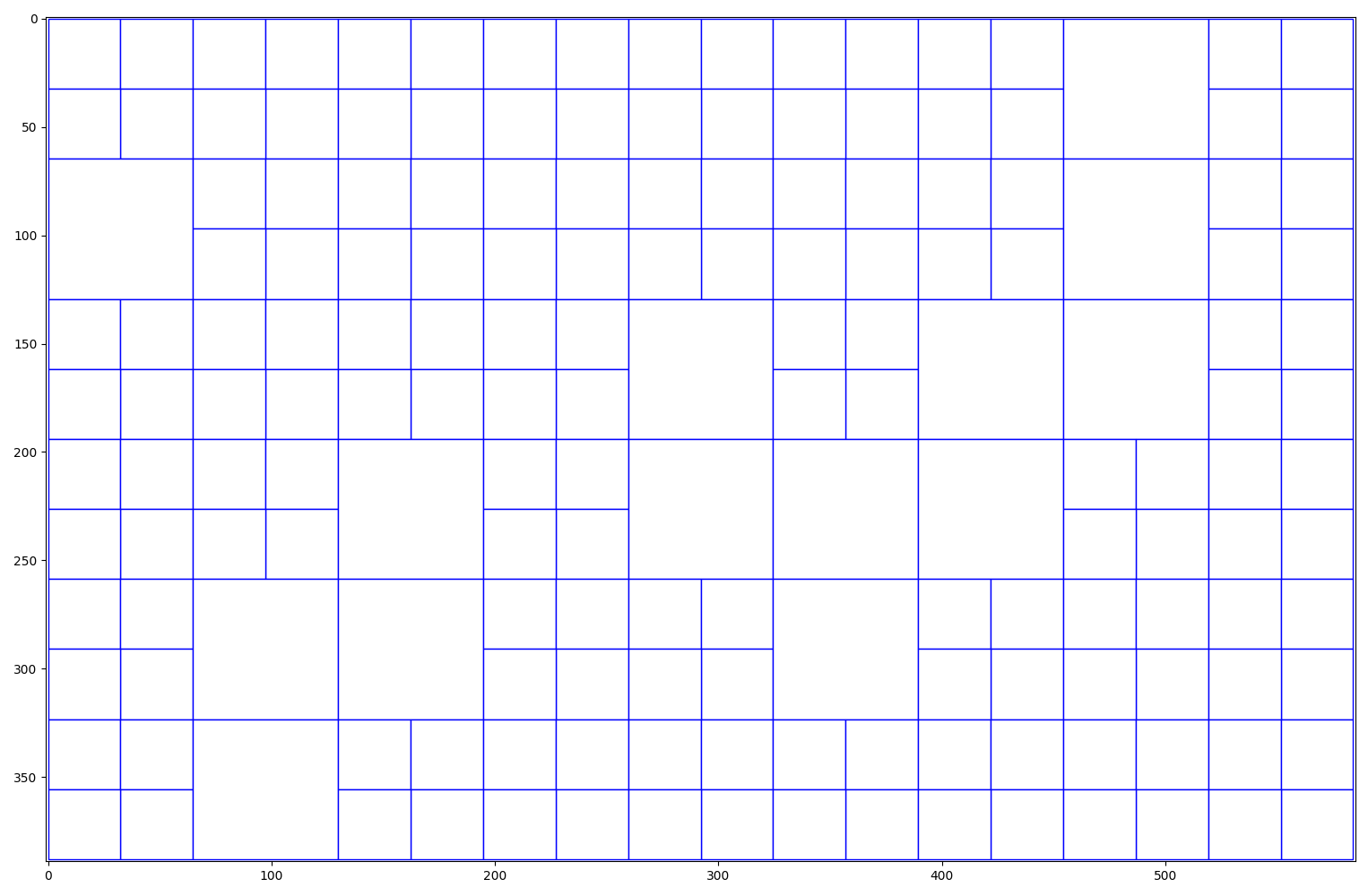}
	}
	\subfloat[dofs: 597]{
		\includegraphics[width=\ofwidth]{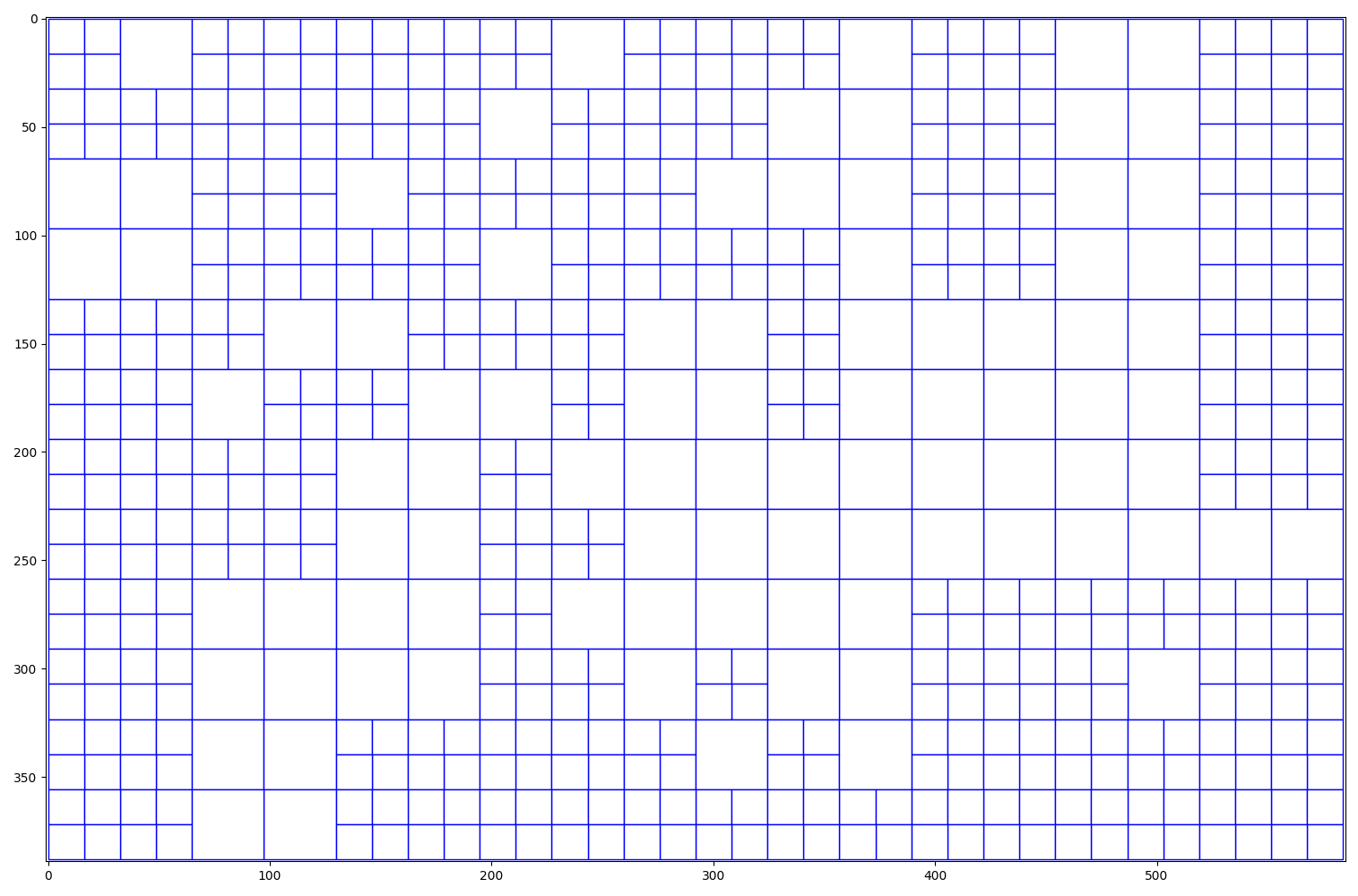}
	}
	\subfloat[dofs: 2,043]{
		\includegraphics[width=\ofwidth]{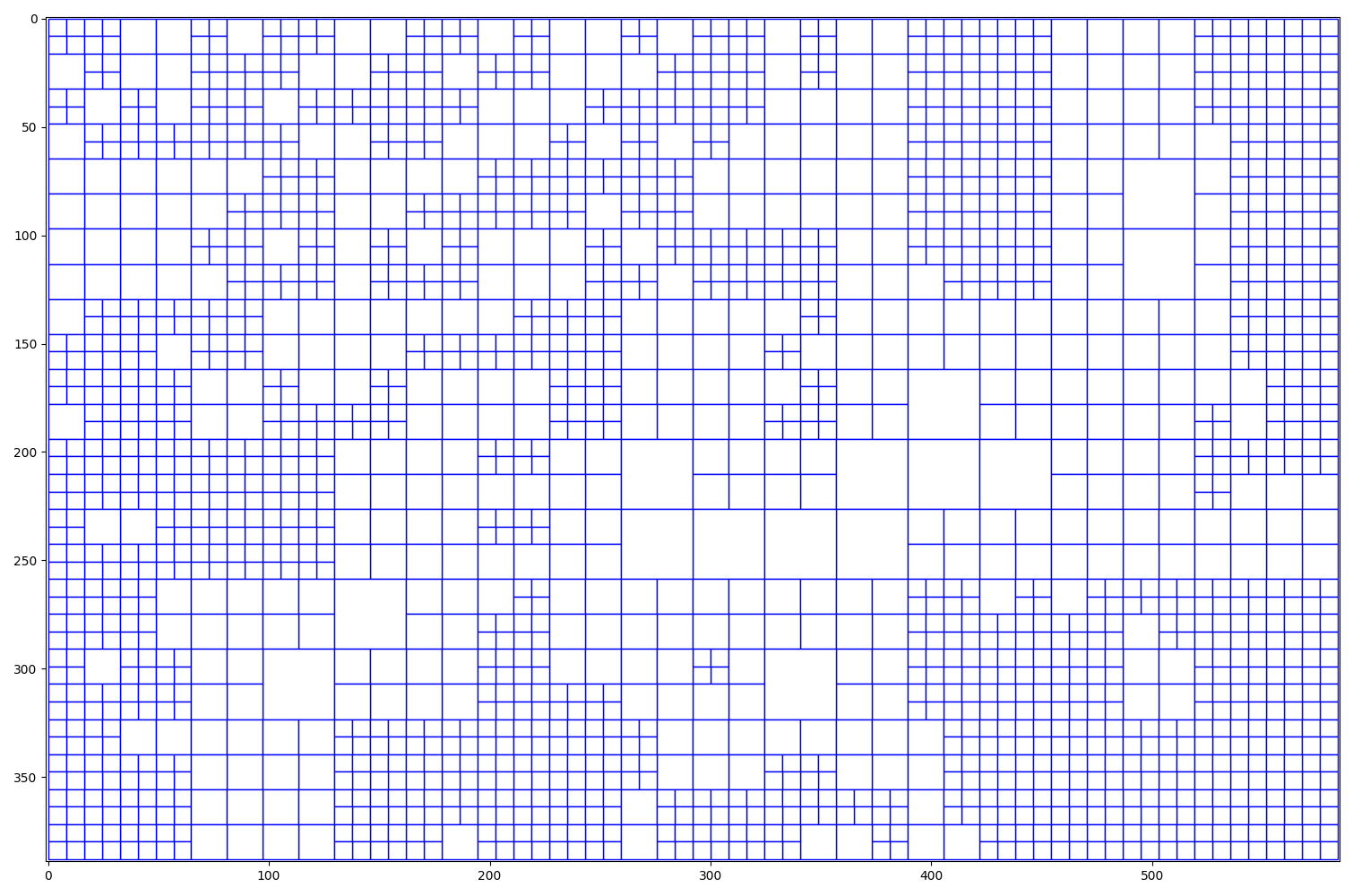}
	}
	
	\subfloat[dofs: 6,933]{
		\includegraphics[width=\ofwidth]{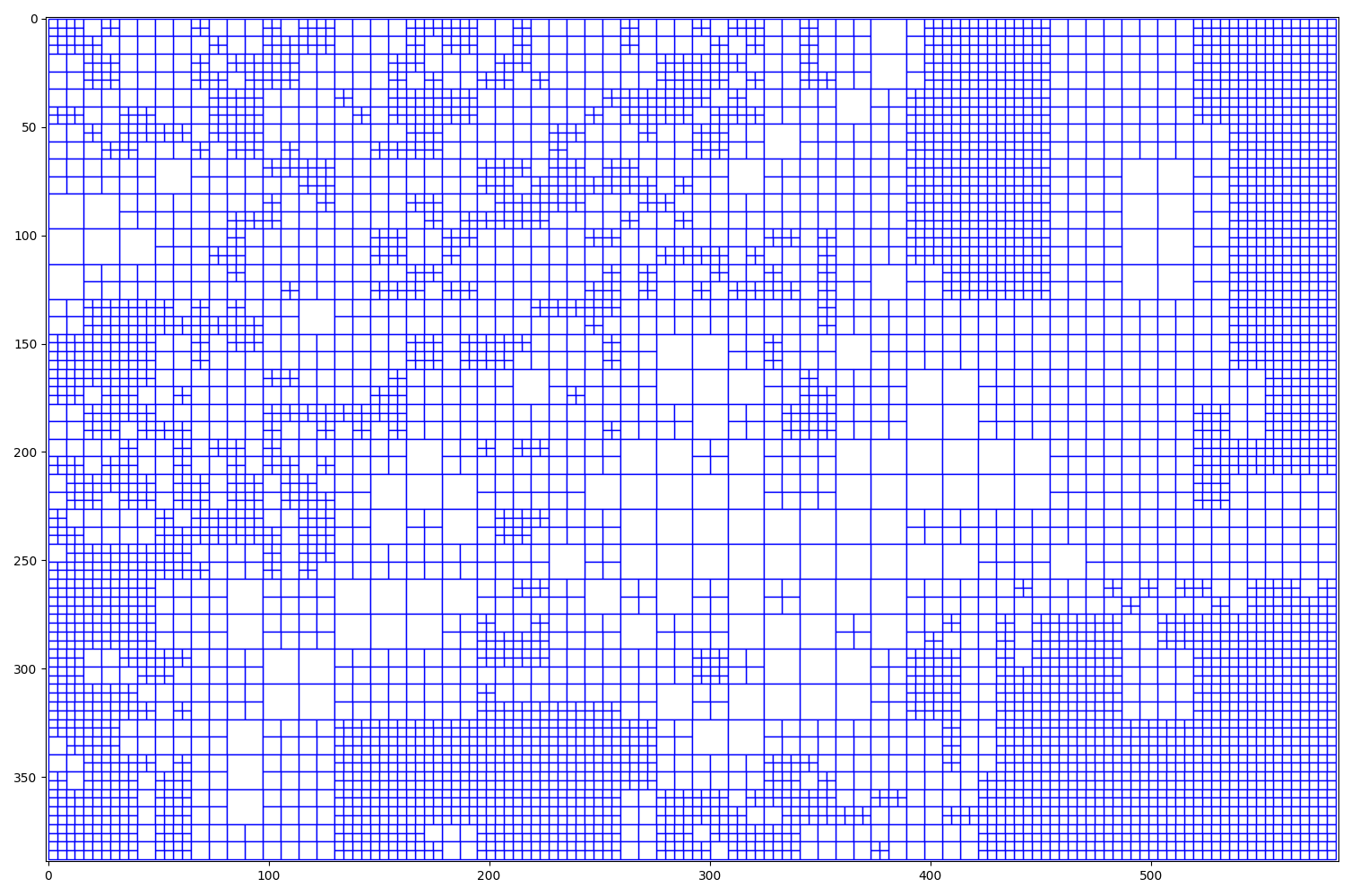}
	}
	\subfloat[dofs: 23,238]{
		\includegraphics[width=\ofwidth]{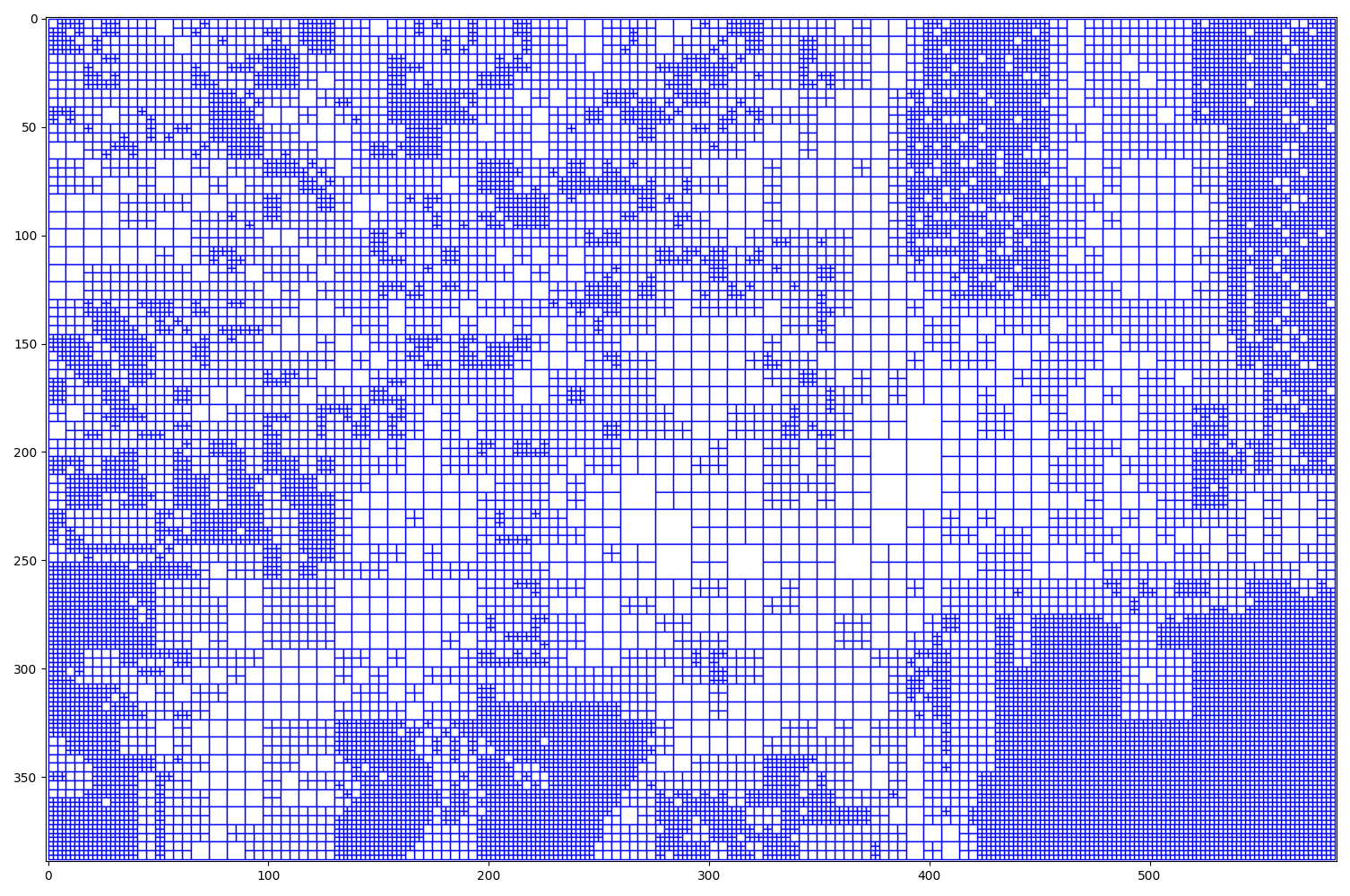}
	}
	\subfloat[dofs: 78,663]{
		\includegraphics[width=\ofwidth]{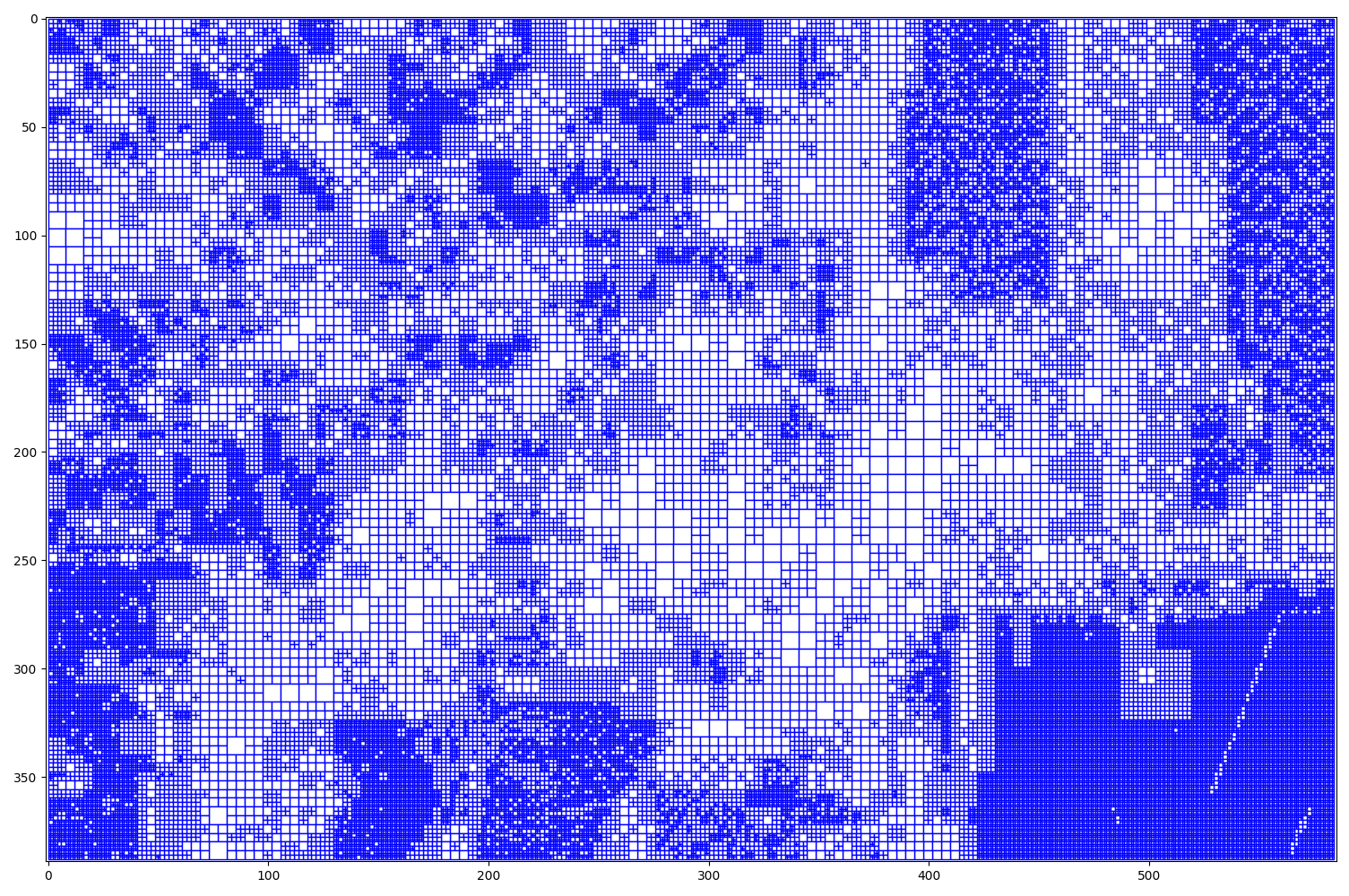}
	}
	
	\caption{Adapted meshes sequence for Algorithm \ref{algo:warping-while-ctf} on the \textit{RubberWhale} example.}
	\label{fig:meshes}
\end{figure}


\section{Conclusion}

We introduced a novel adaptive finite difference scheme to solve the $L^1$-$L^2$-$\TV$ problem numerically, leading to a non-uniform discretization of the image domain. 
From our numerical experiments we conclude that an adaptive meshing in a finite difference setting can be beneficial. In particular the degrees of freedom can be significantly reduced. In the task of estimating the optical flow we could additionally observe that our coarse-to-fine scheme reduces computation time and sometimes even improves the quality of the result as well, promoting our approach.

While adaptively changing the underlying mesh can be advantageous in a computational setting, we also investigated the influence of a refinement on the discrete total variation. It turns out that the discrete total variation is indeed dependent on the grid. In particular, a refinement may increase, but does not decrease, the total variation of a function, cf. Section \ref{sec:analysis-adaptive}. While this property is plausible and obviously has an effect on the minimization of the total variation, it is not clear to us if and how it influences the $\Gamma$-convergence of the discrete total variation to the continuous one.

\bibliographystyle{abbrv}
\bibliography{main}

\appendix

\section{Proof of the Primal-Dual Gap Error Estimation}
\label{app:pd-gap-proof}

The energy $E$ satisfies the following property (see for example \cite[Lemma 3.7]{Hilb2023}): \\

\begin{lemma} \label{lem:B-norm}
	If $\mathbf{u}$ is a minimizer of \eqref{eq:primal}, then we have
	$$ \frac{1}{2} \|\mathbf{v} - \mathbf{u}\|_B^2 \leq E(\mathbf{v}) - E(\mathbf{u}), $$
	for all $\mathbf{v}$ in $H^1(\Omega)^m$.
\end{lemma}

Now we give the proof of Proposition \ref{prop:pd-gap-error}.
\begin{proof}[Proof of Proposition \ref{prop:pd-gap-error}]
	We recall that $I_h$ denotes the piecewise constant interpolation. By Lemma \ref{lem:B-norm} we have,
	\begin{align*}
		\frac{1}{2}\|I_h \mathbf{u_h}-\mathbf{u}\|_B^2 & \leq E(I_h \mathbf{u_h}) - E(\mathbf{u}).
	\end{align*}
	The strong duality of the continuous problems leads to,
	\begin{align*}
		\frac{1}{2}\|I_h \mathbf{u_h}-\mathbf{u}\|_B^2 & \leq E(I_h \mathbf{u_h}) - D(p_1,\mathbf{p_2}) \leq E(I_h \mathbf{u_h}) - D(I_h p_{h,1},I_h \mathbf{p_{h,2}}) \\
		& = \eta_h(\mathbf{u_h}, p_{h,1},\mathbf{p_{h,2}}) + \big( E(I_h \mathbf{u_h}) - E_h(\mathbf{u_h}) \big) - \big( D(I_h p_{h,1},I_h \mathbf{p_{h,2}}) - D_h(p_{h,1},\mathbf{p_{h,2}}) \big) \\
		& \leq \eta_h(\mathbf{v_h}, q_{h,1},\mathbf{q_{h,2}}) + \big( E(I_h \mathbf{u_h}) - E_h(\mathbf{u_h}) \big) + \big( D_h(p_{h,1},\mathbf{p_{h,2}}) - D(I_h p_{h,1},I_h \mathbf{p_{h,2}}) \big).
	\end{align*}
	The last term in the above inequality can be written as
	\begin{multline} \label{eq:dual-diff}
		D_h(p_{h,1},\mathbf{p_{h,2}}) - D(I_h p_{h,1},I_h \mathbf{p_{h,2}}) = \frac{\alpha_2}{2}\Big( \|g_h\|_{L^2(\Omega_h)}^2 - \|g\|_{L^2(\Omega)}^2 \Big) + \Big( \langle g, I_h p_{h,1} \rangle_{L^2(\Omega)} - \langle g_h, p_{h,1} \rangle_{L^2(\Omega_h)} \Big) \\ +\frac{1}{2}\Big(\|T^* I_h p_{h,1} + \nabla^*I_h \mathbf{p_{h,2}} + \alpha_2 T^* g\|_{B^{-1}}^2 - \| T_h^* q_{h,1} + \nabla_h^*\mathbf{q_{h,2}} - \alpha_2 T_h^* g_h \|_{B_h^{-1}}^2 \Big).
	\end{multline}
	The first term can be bounded from above as follows
	\begin{align*}
		\|g_h\|_{L^2(\Omega_h)}^2 - \|g\|_{L^2(\Omega)}^2 & = \sum_{i=1}^N h_i^2 (I_h g_h)(x_i)^2 - \|g\|_{L^2(\Omega)}^2 
		 = \sum_{i=1}^N \int_{E_i} (I_h g_h)(x_i)^2\ dx - \|g\|_{L^2(\Omega)}^2 \\
		& = \sum_{i=1}^N \int_{E_i} (I_h g_h)^2\ dx - \|g\|_{L^2(\Omega)}^2
		 = \|I_h g_h\|_{L^2(\Omega)}^2 - \|g\|_{L^2(\Omega)}^2 \\
		& \leq \|g-I_h g_h\|_{L^2(\Omega)} \|g + I_h g_h\|_{L^2(\Omega)}. 
	\end{align*}
	Since by definition, $I_h g_h$ is the $L^2$-projection of $g$ into the space of piecewise constant functions, it follows using \cite[Lemma 6.54]{aliprantis_infinite_2006} that $\|I_h g_h\|_{L^2(\Omega)}\leq \|g\|_{L^2(\Omega)}$ and thus
	$$ \|g_h\|_{L^2(\Omega_h)}^2 - \|g\|_{L^2(\Omega)}^2 \leq 2\|g\|_{L^2(\Omega)} \|g-I_h g_h\|_{L^2(\Omega)}. $$
	For the second term in \eqref{eq:dual-diff}, we have
	\begin{align*}
		\langle g, I_h p_{h,1} \rangle_{L^2(\Omega)} - \langle g_h, p_{h,1} \rangle_{L^2(\Omega_h)} & = \langle g, I_h p_{h,1} \rangle_{L^2(\Omega)} - \langle I_h g_h, I_h p_{h,1} \rangle_{L^2(\Omega)} 
		 = \langle g - I_h g_h, I_h p_{h,1} \rangle_{L^2(\Omega)} \\
		& \leq \|I_h p_{h,1}\|_{L^2(\Omega)}\|g-I_h g_h\|_{L^2(\Omega)} 
		 \leq \alpha_1|\Omega|\|g-I_h g_h\|_{L^2(\Omega)}.
	\end{align*}
	So that
	$$ D_h(p_{h,1},\mathbf{p_{h,2}}) - D(I_h p_{h,1},I_h \mathbf{p_{h,2}}) \leq  c\|g-I_h g_h\|_{L^2(\Omega)} + Q_h(p_{h,1},\mathbf{p_{h,2}}) $$
	with $c\geq 0$, where
	\begin{multline*}
		Q_h(p_{h,1},\mathbf{p_{h,2}}):=\frac{1}{2}\Big(\|T^* I_h p_{h,1} + \nabla^*I_h \mathbf{p_{h,2}} + \alpha_2 T^* g\|_{B^{-1}}^2 - \| T_h^* q_{h,1} + \nabla_h^*\mathbf{q_{h,2}} - \alpha_2 T_h^* g_h \|_{B_h^{-1}}^2 \Big).
	\end{multline*}
\end{proof}

\end{document}